\documentclass[a4paper, 11pt]{amsart}

\usepackage{amssymb}
\usepackage{amsmath}
\usepackage{amsthm}
\usepackage{enumitem}
\usepackage{mathtools}
\usepackage{pxfonts}
\usepackage[colorlinks=true]{hyperref}
\usepackage{dsfont}
\usepackage{stmaryrd}
\usepackage[backend=biber, citestyle=numeric, bibstyle=ieee, sorting=nyt, maxnames=100]{biblatex}
\usepackage{verbatim}

\newtheorem{thm}{Theorem}[section]
\newtheorem{lemma}[thm]{Lemma}
\newtheorem{cor}[thm]{Corollary}
\newtheorem{prop}[thm]{Proposition}
\theoremstyle{definition}

\newtheorem{example}[thm]{Example}

\theoremstyle{remark}
\newtheorem{remark}[thm]{Remark}
\newtheorem{claim}[thm]{Claim}

\newtheorem{convention}[thm]{Convention}

\numberwithin{equation}{section}

\newcommand{\acts}{\curvearrowright}
\newcommand{\G}{\mathcal{G}}
\newcommand{\h}{\mathcal{H}}
\newcommand{\K}{\mathcal{K}}
\newcommand{\R}{\mathcal{R}}
\newcommand{\s}{\mathcal{S}}

\newcommand{\Z}{\mathbf{Z}}

\usepackage{soul}
\usepackage[dvipsnames]{xcolor}

\usepackage{tikz-cd}

\DeclareMathAlphabet{\mathdutchcal}{U}{dutchcal}{m}{n}
\SetMathAlphabet{\mathdutchcal}{bold}{U}{dutchcal}{b}{n}
\DeclareMathAlphabet{\mathdutchbcal}{U}{dutchcal}{b}{n}

\DeclareMathAlphabet{\mathpzc}{OT1}{pzc}{m}{it}

\DeclareMathAlphabet{\mathboondoxfrak}{U}{BOONDOX-frak}{m}{n}

\DeclareMathOperator{\Aut}{Aut}

\makeatletter
\def\l@section{\@tocline{1}{5pt}{0pc}{}{}}
\renewcommand{\tocsection}[3]{%
	\indentlabel{\@ifnotempty{#2}{\makebox[20pt][l]{%
				\ignorespaces#1 #2.\hfill}}}\sc #3\dotfill}

\newdimen{\tocsubsecmarg}
\def\l@subsection{\@tocline{2}{2pt}{0pc}{\tocsubsecmarg}{}}
\renewcommand{\tocsubsection}[3]{%
	\indentlabel{\@ifnotempty{#2}{\makebox[30pt][l]{%
				\ignorespaces#1 #2.\hfill}}}#3\dotfill}
\makeatother
\let\oldtocsubsection=\tocsubsection
\renewcommand{\tocsubsection}[2]{\hspace{2em} \oldtocsubsection{#1}{#2}}

\begin{document}
\title[Measurable splittings and wreath products]{Measurable splittings and the measured group theoretic structure of wreath products}

\author[Robin Tucker-Drob]{Robin Tucker-Drob}
\address[Robin Tucker-Drob]{Department of Mathematics, University of Florida, Gainesville, FL, USA}
\email{r.tuckerdrob@ufl.edu}
\author[Konrad Wr\'{o}bel]{Konrad Wr\'{o}bel}
\address[Konrad Wr\'{o}bel]{Department of Mathematics, Jagiellonian University, Krak\'ow, Poland}
\email{konrad1.wrobel@uj.edu.pl}

\begin{abstract}
    Let $\Gamma$ be a countable group that admits an essential measurable splitting (for instance, any group measure equivalent to a free product of nontrivial groups). 
    We show: (1) for any two nontrivial countable groups $B$ and $C$ that are measure equivalent, the wreath product groups $B\wr\Gamma$ and $C\wr\Gamma$ are measure equivalent (in fact, orbit equivalent) -- this is interesting even in the case when the groups $B$ and $C$ are finite; and (2) the groups $B\wr \Gamma$ and $(B\times\mathbf{Z})\wr\Gamma$ are measure equivalent (in fact, orbit equivalent) for every nontrivial countable group $B$.
    On the other hand, we show that certain wreath product actions are not even stably orbit equivalent if $\Gamma$ is instead assumed to be a sofic icc group that is Bernoulli superrigid, and $B$ and $C$ have different cardinalities. 
\end{abstract}

\maketitle

\tableofcontents

\section{Introduction}
Two groups $\Gamma$ and $\Lambda$ are said to be \textbf{measure equivalent} if they admit a \textbf{measure equivalence coupling}, i.e.\  commuting measure preserving actions on a standard nonzero $\sigma-$finite measure space $(\Omega,\mu)$ such that both the actions $\Gamma\acts (\Omega,\mu)$ and $\Lambda\acts(\Omega,\mu)$ admit a finite measure fundamental domain. 
Such a coupling comes with a \textbf{coupling index}, defined to be the ratio between the measures of the fundamental domains of the two actions.  
If $\Gamma$ and $\Lambda$ admit an ergodic measure equivalence coupling with coupling index 1, then we say that $\Gamma$ and $\Lambda$ are \textbf{orbit equivalent}. It is easy to observe that finite groups form a measure equivalence class and the index of any measure equivalence coupling is the ratio of their cardinalities. 
The landmark article \cite{OW87} of Ornstein and Weiss implies that all infinite amenable groups are orbit equivalent. 
Since amenability is an invariant of measure equivalence, the class of infinite amenable groups forms its own measure equivalence class. 

Over the past 25 years, there has been significant progress in developing rigidity results related to measure equivalence and orbit equivalence. 
Two of the most celebrated rigidity results are Popa's cocycle superrigidity theorems, which imply that any orbit equivalence (defined in $\S$\ref{section:OEME}) to a Bernoulli shift action of a wide variety of groups must entail an isomorphism both of the acting groups as well as of the actions \cite{Popa2007,Po08}, and the product and radical rigidity theorems of Monod and Shalom, which recover group theoretic structure from an orbit equivalence of actions under certain negative curvature assumptions on the groups \cite{MonodShalom2006}.
At the other end of the spectrum, results showing algebraically dissimilar groups are orbit equivalent appear to be less common, with a prominent exception being the work \cite{OW87} of Ornstein and Weiss. 

\subsection{Anti-rigidity for wreath products} Our first main theorem is an anti-rigidity result in the nonamenable setting for wreath product groups. 
Given groups $B$ and $\Gamma$, we let $\Gamma$ act on $\bigoplus_\Gamma B$ via the left shift $(\gamma . b)_\delta = b_{\gamma ^{-1}\delta}$ for $b\in \bigoplus_{\Gamma} B$, $\gamma ,\delta\in\Gamma$, and we  denote by $B \wr \Gamma$ the restricted regular \textbf{wreath product} of $B$ with $\Gamma$, viewed as an internal semidirect product of $\bigoplus _\Gamma B$ and $\Gamma$:
\[
B\wr\Gamma = \Big\langle \textstyle{\bigoplus _\Gamma B , \ \Gamma \, {\Big{|}} \,  \gamma b \gamma ^{-1} = \gamma . b \ \ (\gamma \in \Gamma , \ b\in \bigoplus _{\Gamma}B ) }\Big\rangle .
\]
We call the subgroups $\bigoplus_\Gamma B$ and $\Gamma$ the \textbf{wreath kernel} and \textbf{wreath complement} respectively, of $B\wr\Gamma$.

Starting with a free ergodic probability measure preserving (p.m.p.) action of $B$ on a standard probability space  $(Z,\eta )$, one obtains the associated \textbf{wreath product action} of $B\wr\Gamma$ on the probability space $(Z^\Gamma ,\eta ^\Gamma )$, defined by $(b\gamma . z)_{\delta}= b_{\gamma^{-1}\delta} . z_{\gamma^{-1}\delta}$ for $z\in Z^\Gamma$, $\gamma ,\delta \in \Gamma$, and $b\in \bigoplus_\Gamma B$. 
As a consequence of the aforementioned work of Ornstein and Weiss, if $A$ is an amenable group then, up to orbit equivalence of actions, the wreath product action of $A\wr\Gamma$ only depends on the cardinality of $A$. 
While the cardinality of $A$ does matter in general (see Theorem \ref{thm:rigid}), by imposing certain restrictions on $\Gamma$ we obtain the following.

\begin{thm}[see Corollary \ref{cor:freeFactorAmenableLampGroups}]\label{thm:cofeq}
Let $\Gamma$ be a countable group that contains an infinite amenable group as a free factor. 
Then the groups $A_0\wr\Gamma$ and $A_1\wr\Gamma$ are orbit equivalent for all nontrivial (possibly finite) amenable groups $A_0$ and $A_1$. 
In fact, the wreath product actions of $A_0\wr\Gamma$ and $A_1\wr\Gamma$ are orbit equivalent.
\end{thm}

Theorem \ref{thm:cofeq} may be compared to results of Genevois and Tessera \cite{GeTe21Asymp} and of Eskin, Fisher, and Whyte \cite{EskinFisherWhyteI, EskinFisherWhyteII} addressing the question of when two wreath product groups are quasi-isometric. 

Letting $\mathbf{F}_2$ denote the free group on 2 generators, $\mathbf{Z}$ denote the group of integers, and $\mathbf{C}_n$ denote the cyclic group of order $n$, we obtain the following as an immediate consequence of Theorem \ref{thm:cofeq}.

\begin{cor}
    The groups $\mathbf{Z}\wr\mathbf{F}_2$ and $\mathbf{C}_n\wr \mathbf{F}_2$ for $n\geq 2$ are all pairwise orbit equivalent.
\end{cor}

This was previously unknown, although as a consequence of a theorem of Bowen the group von Neumann algebras $L(\mathbf{C}_n\wr \mathbf{F}_2)$ and $L(\mathbf{C}_m\wr \mathbf{F}_2)$ are isomorphic for $n,m\geq 2$ \cite{BowenBernshiftflex} (see also \cite{MRV2013}). 

Toward proving Theorem \ref{thm:cofeq}, we isolate the notion of a bi-cofinitely equivariant map between shift spaces. 
Bi-cofinitely equivariant maps have previously appeared implicitly in the study of finitary isomorphisms between shifts \cite{Kr83}. 
For simplicity, assume the alphabets $S_0$ and $S_1$ are finite. For $i\in\{0,1\}$ and $x,y\in S_i^\Gamma$ we use $x\sim y$ to denote that $x$ and $y$ differ in only finitely many coordinates. 
A bijective map $\phi:S_0^\Gamma\rightarrow S_1^\Gamma$ between shift spaces is said to be bi-cofinitely equivariant if for every $x\sim y\in S_0^\Gamma$, $z\sim w\in S_1^\Gamma$, and for every $\gamma\in \Gamma$, equivariance holds modulo finite errors, i.e., $\gamma . \phi(x)\sim \phi(\gamma . y)$ and $\gamma.\phi^{-1}(z)\sim\phi^{-1}(\gamma.w)$. 

Our proof of Theorem \ref{thm:cofeq} parallels Bowen's construction from \cite{BowenBernshiftflex} insofar as the orbit equivalence that we define between the wreath product actions of $B\wr \Gamma$ and $C \wr \Gamma$ is obtained as a ``coset-by-coset lift'' of an orbit equivalence between wreath product actions of $B\wr \Lambda$ and $C\wr\Lambda$ for an appropriate subgroup $\Lambda$ of $\Gamma$. One way in which our setting differs from Bowen's is that it is critical for us that the input orbit equivalence between the $B\wr \Lambda$ and $C\wr\Lambda$ actions is bi-cofinitely equivariant. To this end, we make essential use of Golodets and Sinel'shchikov's cocycle generalization \cite{Golodets1994} of Dye's Theorem \cite{Dye1959} to obtain bi-cofinitely equivariant orbit equivalence maps between shift spaces in the amenable setting.

By working in the setting of groupoids, we can strengthen the first statement in Theorem \ref{thm:cofeq} in two directions: we need only assume that $\Gamma$ admits an essential measurable splitting, and we can replace the assumption on amenability of the wreath kernels with an assumption that we believe to be near optimal.

\begin{thm}[see Theorem \ref{thm:MELampGroups}]\label{thm:antirigid}
Let $\Gamma$ be a countable group that admits an essential measurable splitting.
Let $B$ and $C$ be nontrivial countable groups.
If $B\times \mathbf{Z}$ and $C\times \mathbf{Z}$ are measure equivalent, then $B\wr\Gamma$ and $C\wr\Gamma$ are orbit equivalent. 
In particular, for all such $\Gamma$ and nontrivial $B$ and $C$ we have: 
\begin{enumerate}
    \item if $B$ is measure equivalent to $C$, then $B\wr\Gamma$ is orbit equivalent to $C\wr\Gamma$;
    \item $B\wr\Gamma$ is orbit equivalent to $(B\times\mathbf{Z})\wr\Gamma$.
\end{enumerate}
\end{thm}

Every group that is measure equivalent to a free product of two nontrivial groups admits an essential measurable splitting (see Proposition \ref{prop:MEtoFreeProductsAdmitEssentialSplittings}), including for instance fundamental groups of genus $g$ surfaces for $g\geq 1$ and infinite amenable groups. For a nonamenable group $\Gamma$, it is known that if $\Gamma$ admits an essential measurable splitting, then the first $\ell^2$-Betti number of $\Gamma$ is positive \cite{AlvarezGaboriau2012}. 
An important open problem (first raised by Alvarez and Gaboriau in \cite{AlvarezGaboriau2012}) is whether the converse is true. 

The proof of Theorem \ref{thm:antirigid} relies crucially on the flexibility inherent in the groupoid framework. 
This flexibility is perhaps most evident in the critical observation, established in Lemma \ref{lem:freeDecomposableImpliesTreeableFreeFactor}, that every measure preserving equivalence relation admitting an essential splitting must contain an aperiodic amenable free factor; the analogous statement in the group setting is simply false.

The flexibility afforded by the groupoid setting is also important for constructing measure equivalence couplings when the wreath complements differ.

\begin{thm}[see Theorem \ref{thm:METopGroups}]\label{thm:fixedWreathKernel}
    Let $\Gamma _1$, $\Gamma _2$, and $B$ be countable groups, and assume that $\Gamma_1$ and $\Gamma_2$ are measure equivalent. 
    Then the groups $(\bigoplus _{\mathbf{N}}B)\wr \Gamma_1$ and $(\bigoplus _{\mathbf{N}}B)\wr\Gamma_2$ are measure equivalent.

It follows that if $A$ is an infinite amenable group, then the groups $A\wr \Gamma_1$ and $A\wr\Gamma_2$ are measure equivalent.
\end{thm}

\subsection{Rigidity} There is a rich collection of rigidity results related to wreath products \cite{ChIo10,CIOS21,ChifanPopaSize2012,DrVa21,Io07,Io11,IPV13,IsMa19,Po06,Po06a,Popa2007,Po08,Sako09}.
Recently, Drimbe and Vaes \cite{DrVa21} proved orbit equivalence superrigidity for the wreath product actions of permutational wreath products $B\wr_I \Gamma$ with certain conditions on the action $\Gamma\acts I$.
We take a moment to single out the rigidity results found in Chifan, Popa, and Sizemore \cite{ChifanPopaSize2012} which generalize results of Sako \cite{Sako09}. 
Under general assumptions on the groups $B$, $C$, $\Gamma$, and $\Lambda$ (e.g., $\Gamma$ and $\Lambda$ either both having property (T) or both being a nonamenable product of infinite groups, and $B$ and $C$ both being either amenable or being icc with property (T)), they show that if $B\wr\Gamma$ and $C\wr\Lambda$ are measure equivalent then $\Gamma$ and $\Lambda$ are measure equivalent, and in the case where all the groups have property (T) then $\bigoplus_{\Gamma}B$ and $\bigoplus_{\Lambda}C$ must be orbit equivalent.
In \S\ref{section:rigidity} we study orbit equivalence rigidity for a more limited class of actions (e.g., wreath product actions), but only impose conditions on the group $\Gamma$.

We note that, as a consequence of results of Popa-Vaes \cite{PV14}, Popa's cocycle superrigidity theorem \cite{Popa2007}, and the theory of sofic entropy \cite{Bowen2010}, the group von Neumann algebras $L(\mathbf{C}_2\wr\Gamma)$ and $L(\mathbf{C}_3\wr\Gamma)$ are not isomorphic whenever $\Gamma$ is a sofic icc hyperbolic group with property (T). 
While it remains an open question whether there exists a group $\Gamma$ such that $\mathbf{C}_2\wr \Gamma$ is not measure equivalent to $\mathbf{C}_3\wr \Gamma$, 
in direct contrast to Theorem \ref{thm:cofeq} (and the philosophy behind Theorems \ref{thm:antirigid} and \ref{thm:fixedWreathKernel}),
 we obtain the following.

\begin{thm}[see Corollary \ref{cor:Rigid}]\label{thm:rigid}
Let $\Gamma$ be a countably infinite group having no nontrivial finite normal subgroups, and whose Bernoulli shift action is $\mathcal{G}_{\mathrm{ctble}}$-cocycle superrigid. 
Let $\Lambda$ be an arbitrary infinite countable group. 

Let $B$ and $C$ be two countable groups acting in a free p.m.p.\ manner on $(X,\mu)$ and $(Y,\nu)$ respectively, and let
$\alpha: B\wr \Gamma\acts (X^\Gamma, \mu^\Gamma)$ and $\beta:C\wr\Lambda \acts (Y^\Lambda, \nu^\Lambda )$ be the associated wreath product actions. Assume $\alpha$ and $\beta$ are stably orbit equivalent. 

Then $\Gamma$ and $\Lambda$ are isomorphic, and $\bigoplus_\Gamma B$ and $\bigoplus_\Lambda C$ are orbit equivalent. Moreover, if $\Gamma$ is sofic and the actions of both $B$ and $C$ are ergodic, then $B$ and $C$ have the same cardinality.
\end{thm}

As an application of Theorem \ref{thm:rigid} we analyze the structure of the automorphism groups of certain wreath product equivalence relations. 

\medskip
\noindent {\bf Acknowledgements:} We thank Adrian Ioana for discussions which originally motivated looking into Theorem \ref{thm:cofeq}, and for his comments on an earlier draft of this article. 
RTD was partially supported by NSF grant DMS 2246684. KW was partially supported by the Dioscuri program initiated by the Max Planck Society, jointly managed by the National Science Centre (Poland), and mutually funded by the Polish Ministry of Science and Higher Education and the German Federal Ministry of Education and Research.

\section{Preliminaries}

\subsection{Orbit equivalence and measure equivalence}\label{section:OEME} Let $\Gamma\acts (X,\mu)$ and $\Lambda \acts(Y,\nu)$ be p.m.p.\ ergodic actions of discrete countable groups. The actions are {\bf stably orbit equivalent} if there exist positive measure subsets $X_0\subseteq X$ and $Y_0\subseteq Y$ as well as a measure class preserving isomorphism $\phi: (X_0,\mu_{X_0})\rightarrow (Y_0,\nu_{Y_0})$ that satisfies $\phi(\Gamma. x\cap X_0)=\Lambda .\phi(x)\cap Y_0$ for almost every $x\in X_0$ where $\mu_{X_0}$ and $\nu_{Y_0}$ are the normalized restrictions. The map $\phi$ is a {\bf stable orbit equivalence (from $X$ to $Y$)} of {\bf index} $\frac{\nu(Y_0)}{\mu(X_0)}$. 
The actions are called {\bf orbit equivalent} if the sets $X_0$ and $Y_0$ can be taken to be conull.

Two groups $\Gamma$ and $\Lambda$ admit an ergodic measure equivalence coupling of index $\alpha$ if and only if they admit free ergodic p.m.p.\ actions $\Gamma\acts (X,\mu)$ and $\Lambda\acts (Y,\nu)$ that are stably orbit equivalent of index $\alpha$; $\Gamma$ and $\Lambda$ admit an ergodic measure equivalence coupling of coupling index $1$ if and only if they admit orbit equivalent free ergodic p.m.p.\ actions \cite{Furman1999OERigidity}.

\subsection{Wreath product actions}

Let $\Gamma$ be a countable group. Given a p.m.p.\ action $B\acts (Z,\eta)$ of a countable group $B$ on a standard probability space $(Z,\eta)$, define the {associated \bf wreath product action} of $B \wr \Gamma$ to be the p.m.p.\ action of $B\wr \Gamma$ on $(Z^\Gamma,\eta^\Gamma)$ determined by $(b. z)_{\delta} \coloneqq b_{\delta}.z_{\delta}$ and $(\gamma .z )
_{\delta}\coloneqq z_{\gamma ^{-1}\delta}$ for $z\in Z^{\Gamma}$, $b\in \bigoplus _{\Gamma}B$, and $\gamma ,\delta \in \Gamma$.

\begin{prop}\label{OElampsWreath}
Let $\Gamma$, $B$, and $C$ be countable discrete groups and assume the p.m.p.\ actions $B\acts(Y,\nu)$ and $C\acts(Z,\eta)$ are orbit equivalent. Then the associated wreath product actions $B\wr\Gamma\acts(Y^\Gamma,\nu^\Gamma)$ and $C\wr\Gamma\acts(Z^\Gamma,\eta^\Gamma)$ are orbit equivalent. 
\end{prop}

\begin{proof}
Let $\phi:(Y,\nu)\rightarrow(Z,\eta)$ be a measure space isomorphism that witnesses the orbit equivalence between $B\acts(Y,\nu)$ and $C\acts(Z,\eta)$. The map $\phi^\Gamma:(Y^\Gamma,\nu^\Gamma)\rightarrow (Z^\Gamma,\eta^\Gamma)$ witnesses the orbit equivalence between the wreath product actions. 
\end{proof}

\begin{remark}
In \cite{DelabieKoivistoLeMaitreTessera2022}, Delabie--Koivisto--Le Maitre--Tessera prove a version of the above proposition where the group $\Gamma$ is replaced by a principal groupoid.
We explore this further in $\S\ref{section:wreathME}$. 
\end{remark}

\begin{thm}[Dye, Ornstein-Weiss]
Let $A_0\acts(Y,\nu)$ and $A_1\acts(Z,\eta)$ be free ergodic p.m.p.\ actions of countable amenable groups $A_0$ and $A_1$. 
Then these actions are orbit equivalent if and only if $A_0$ and $A_1$ have the same cardinality.
\end{thm}

Hence, when $A$ is an amenable group, this means that, up to orbit equivalence, there is a unique wreath product action $A\wr\Gamma\acts(Z^\Gamma,\eta^\Gamma)$ in which the direct sum $\bigoplus_\Gamma A$ acts freely and ergodically. 
We call this action the {\bf canonical wreath product action} of $A\wr\Gamma$.

\subsection{Bi-cofinitely equivariant maps}

For $x,y\in Y^\Gamma$ and a fixed free action $B\acts Y$ of a group $B$, we write $x\sim_B y$ to denote that $x$ and $y$ differ in only finitely many coordinates and for each coordinate $\gamma\in \Gamma$ in which they differ, there exists $\lambda\in B$ such that $\lambda .x_\gamma=y_\gamma$. 

Let $B\acts (Y,\nu )$ and $C\acts (Z,\eta )$ be p.m.p.\ actions of countable groups $B$ and $C$ on standard probability spaces $(Y,\nu )$ and $(Z,\eta )$, and let $\Gamma$ be a countable group. 
A Borel map $\phi:Y^\Gamma\rightarrow Z^\Gamma$ is said to be \textbf{cofinitely equivariant (with respect to the given actions of $B$ and $C$)} if for almost every $y\in Y^\Gamma$, for every $x\sim_{B} y$ and every $\gamma\in \Gamma$ we have $\phi(\gamma .x)\sim_{C} \gamma .\phi(y)$. 
A measure space isomorphism $\phi : (Y^{\Gamma}, \nu ^{\Gamma})\rightarrow (Z^{\Gamma},\eta ^{\Gamma})$ is called \textbf{bi-cofinitely equivariant} if both $\phi$ and $\phi ^{-1}$ are cofinitely equivariant.

We say that the wreath product actions $B\wr\Gamma\curvearrowright (Y^\Gamma ,\nu ^\Gamma )$ and $C\wr\Gamma \curvearrowright (Z^\Gamma ,\eta ^\Gamma )$ are \textbf{cofinitely isomorphic} if there exists a bi-cofinitely equivariant measure space isomorphism between the underlying spaces.

It is clear that if $\phi :Y^{\Gamma}\to Z^{\Gamma}$ is a bi-cofinitely equivariant measure space isomorphism then $\phi$ is simultaneously an orbit equivalence between the wreath product actions of $B\wr \Gamma$ and $C\wr \Gamma$  as well as an orbit equivalence between the actions of the direct sum subgroups $\bigoplus_\Gamma B$ and $\bigoplus_\Gamma C$.
The converse is not quite true, but we have the following under the assumption that the actions of $B$ on $(Y,\nu )$ and of $C$ on $(Z,\eta )$ are ergodic.

\begin{prop}\label{cofSimOE}
Let $\Gamma$ be a countable group and let $B$ and $C$ be countable groups with ergodic p.m.p.\ actions $B\acts (Y,\nu)$ and $C\acts (Z,\eta )$. 

Suppose that $\phi: (Y^\Gamma,\nu ^\Gamma )\to (Z^\Gamma ,\eta ^\Gamma )$ is simultaneously an orbit equivalence between the wreath product actions of $B\wr \Gamma$ and $C\wr \Gamma$, as well as an orbit equivalence between the actions of the direct sum subgroups $\bigoplus_\Gamma B$ and $\bigoplus_\Gamma C$.
Then there exists an automorphism $\rho$ of $\Gamma$ such that 
\[
\phi (\gamma .y)\in \rho (\gamma )(\bigoplus _{\Gamma}C).\phi (y)
\]
for all $\gamma \in \Gamma$ and a.e.\ $y\in Y^{\Gamma}$, and the map 
$\theta : Y^{\Gamma} \to Z^{\Gamma}$, defined by 
\[
\theta (y)_{\delta} \coloneqq  \phi (y)_{\rho (\delta )}
\]
for $\delta \in \Gamma$ and $y\in Y^{\Gamma}$, is a bi-cofinitely equivariant measure space isomorphism.  
\end{prop}

\begin{proof}
Note that ergodicity of the action of $B$ on $(Y,\nu )$ implies ergodicity of the action of $\bigoplus _{\Gamma} B$ on $(Y^{\Gamma},\nu ^{\Gamma})$.
By assumption, there exists a measurable map $\alpha : B\wr \Gamma\times Y^\Gamma \to \Gamma$ such that 
\[
\phi (g.y)\in \alpha (g,y)(\bigoplus _{\Gamma} C).\phi (y) \ \ \text{ and } \ \ \alpha (b,y)=e
\]
for all $g\in B\wr\Gamma$, $b\in \bigoplus _{\Gamma}B$, and a.e.\ $y\in Y^{\Gamma}$.
The map $\alpha$ is a cocycle, i.e., it satisfies the cocycle identity $\alpha (gh,y)=\alpha (g,h.y)\alpha (h,y)$ for $g,h\in B\wr\Gamma$ and a.e.\ $y\in Y$.
It follows that 
\[
\alpha (\gamma b \gamma ^{-1},\gamma .y ) = \alpha (\gamma , b.y )\alpha (b,y) \alpha (\gamma ,y)^{-1}
\]
for all $b \in \bigoplus _{\Gamma} B$, $\gamma \in \Gamma$, and a.e.\ $y\in Y^\Gamma$.
Since both $\alpha (b ,y )$ and $\alpha (\gamma b \gamma ^{-1} , \gamma .y )$ are trivial, we conclude that for each fixed $\gamma \in \Gamma$ the map $y\mapsto \alpha (\gamma ,y)$ is $\bigoplus _{\Gamma}B$-invariant, and hence by ergodicity there is some $\rho (\gamma )\in \Gamma$ such that $\alpha (\gamma ,y)=\rho (\gamma )$ for a.e.\ $y\in Y^{\Gamma}$. 
It follows from the cocycle identity that $\rho$ is a homomorphism. 
Likewise, there is a homomorphism $\sigma :\Gamma \to \Gamma$ such that $\phi ^{-1}(\gamma .z)\in \sigma (\gamma )(\bigoplus _{\Gamma}B).\phi ^{-1}(z)$ for a.e.\ $z\in C^{\Gamma}$.
It follows that $\sigma = \rho ^{-1}$ and hence $\rho$ is an automorphism of $\Gamma$.
It is now clear that the map $\theta$ is a bi-cofinitely equivariant measure space isomorphism.
\end{proof}

\begin{example}
    In Proposition \ref{cofSimOE},  ergodicity of the actions $B\acts (Y,\nu)$ and $C\acts(Z,\eta)$ is necessary. 
    Indeed, take $\Gamma=\mathbf{Z}$ and standard probability spaces $(Y,\nu)$ and $(Z,\eta)$ with different Shannon entropies. Let both $B$ and $C$ act trivially on $(Y,\nu)$ and $(Z,\eta)$ respectfully.  
    The wreath product actions $B\wr\mathbf{Z}\acts (Y^\mathbf{Z},\nu^\mathbf{Z})$ and $C\wr\mathbf{Z}\acts(Z^\mathbf{Z},\eta^\mathbf{Z})$ admit an orbit equivalence between them. 
    Additionally, every measure space isomorphism is an orbit equivalence between the $\bigoplus_\mathbf{Z} B$ and $\bigoplus_\mathbf{Z} C$ relations.  
    In this context, for every $x,y\in Y^\mathbf{Z}$ we have $x\sim_{B} y$ if and only if $x=y$ and similarly for the $\sim_{C}$ relation. Therefore, any cofinitely equivariant map would be an equivariant measure space isomorphism between Bernoulli shifts of different entropies, which is impossible.
\end{example}

It should be noted that an equivariant map is not necessarily cofinitely equivariant in our sense as shown in the following example.

\begin{example}\label{ex:eqnoncof}
Consider a nonuniform Bernoulli $\mathbf{Z}$-shift of entropy $\ln 2$. By Ornstein's Isomorphism Theorem \cite{Ornstein1970}, there is an equivariant measure space isomorphism between this Bernoulli shift and the uniform 2-shift since they have the same entropy. 
Taking transitive group actions on the two base spaces, the wreath product action associated to the uniform 2-shift is p.m.p.\  and the wreath product action associated to the nonuniform shift does not admit an equivalent invariant measure. Hence, the actions are not orbit equivalent and the original shifts do not admit a bi-cofinitely equivariant measure class preserving measure space isomorphism between them.
\end{example} 

Proposition \ref{cofSimOE} and the following rigidity result of Monod and Shalom imply that, for countable groups $A_0$, $A_1$, and $\Gamma$, with $A_0$ and $A_1$ amenable, and $\Gamma$ nonamenable and measure equivalent to a free product of nontrivial groups, if the canonical wreath product actions of the groups $A_0\wr \Gamma$ and $A_1\wr \Gamma$ are orbit equivalent, 
then they are already cofinitely isomorphic. 

\begin{thm}[Monod-Shalom \cite{MonodShalom2006}]
Let $\Gamma$ and $\Lambda$ be discrete countable groups with normal amenable subgroups $M\triangleleft \Gamma$, $N\triangleleft \Lambda$ such that the quotients $\Bar{\Gamma}=\Gamma/M$ and $\Bar{\Lambda}=\Lambda/N$ are torsion free and in the class $\mathcal{C}$ of Monod and Shalom. 
Let $(X,\mu)$ and $(Y,\nu)$ be probability $\Gamma-$ and $\Lambda-$spaces respectively on which $M$ and $N$ act ergodically. If the actions of $\Gamma$ and $\Lambda$ are orbit equivalent, then there is a group isomorphism $f:\bar{\Gamma}\rightarrow\bar{\Lambda}$ and a measure space isomorphism $\phi: (X,\mu)\rightarrow (Y,\nu)$ satisfying $\phi (\gamma M.x) = f(\gamma M).\phi (x)$ for all $\gamma \in \Gamma$ and almost every $x\in X$. In particular, $\phi$ is simultaneously an orbit equivalence between the $\Gamma-$ and $\Lambda-$ actions and an orbit equivalence between the $M-$ and $N-$actions.
\end{thm}

This statement does not explicitly appear in \cite{MonodShalom2006}, although  \cite[Theorem 1.12]{MonodShalom2006} covers the case where $\Gamma = \Lambda$, so we take the liberty of including a derivation from \cite[Theorem 2.23]{MonodShalom2006} here.

\begin{proof}
By \cite[Theorem 2.12]{MonodShalom2006}, there is a measure equivalence coupling $\Omega$ of $\Gamma$ with $\Lambda$, as well as a measure $1$ simultaneous fundamental domain $Z\subseteq \Omega$ for the $\Gamma-$ and $\Lambda-$actions on $\Omega$, 
and measure space isomorphisms $\phi_X: X\rightarrow Z$ and $\phi_Y: Y\rightarrow Z$ that are equivariant for the associated actions of $\Gamma$ and $\Lambda$ respectively on $Z$.
By \cite[Theorem 2.23]{MonodShalom2006}, 
there exists a measurable $\Gamma\times\Lambda $-equivariant map $\Psi: \Omega\to \bar{\Omega}$, where $\bar{\Omega}$ is a coupling of $\bar{\Gamma}$ and $\bar{\Lambda}$ such that both of the actions $\bar{\Gamma}\curvearrowright \bar{\Omega}$ and $\bar{\Lambda}\curvearrowright \bar{\Omega}$ are free and transitive. 
Fix a base point $x_0\in \bar{\Omega}$ and let $f: \bar{\Gamma} \to \bar{\Lambda}$ be the map sending $\bar{\gamma} \in \bar{\Gamma}$ to the unique $\bar{\lambda} \in \bar{\Lambda}$ for which $\bar{\gamma} .x_0 = \bar{\lambda} ^{-1}.x_0$, so that $f$ is a group isomorphism.
The set $D\coloneqq \Psi^{-1}(x_0)$ is $M\times N$-invariant and contains measure 1 fundamental domains $Z_\Gamma$ and $Z_\Lambda$ for the $\Gamma$ and $\Lambda$ action respectively on $\Omega$.

We claim that the action of $M\times N$ on $D$ is ergodic. Since $M$ acts ergodically on $X$, the $M\times \Lambda$ action on $\Omega$ is ergodic. Therefore, given a positive measure $M\times N$-invariant subset $D_0$ of $D$, the union of all translates of $D_0$ by a complete collection of left-coset representatives of $N$ in $\Lambda$ is a positive measure $M\times\Lambda$-invariant subset of $\Omega$, hence this union is conull in $\Omega$. Since the only such translate that meets $D$ is $D_0$ itself, we conclude that $D_0$ is conull in $D$, thereby proving the claim.

Therefore, since the sets $Z_\Lambda$ and $Z_\Gamma$ are both measure 1 subsets of $D$, after discarding a null set there exist measurable maps $Z_{\Gamma}\to M$, $z\mapsto m_z$, and $Z_{\Gamma}\to N$, $z\mapsto n_z$, such that the transformation $Tz\coloneqq m_z.n_z.z$ is a bijection from $Z_{\Gamma}$ to $Z_{\Lambda}$.
Define $T_M: Z_\Gamma\to D$ by $z\mapsto m_z.z$ and $T_N: Z_\Lambda\to D$ by $z\mapsto n_{T^{-1}z}^{-1}.z$. 
Define $\tilde Z\coloneqq T_M(Z_\Gamma)$. Then also $\tilde{Z}=T_N(Z_\Lambda)$, so that $\tilde{Z}$ is a simultaneous fundamental domain for the $\Gamma$ and $\Lambda$ actions on $\Omega$, and $\tilde{Z}$ is moreover contained in $D$.

Let $T_\Gamma: Z\to Z_\Gamma$ and $T_\Lambda: Z\rightarrow Z_\Lambda$ be the unique maps with $T_\Gamma(z)\in \Gamma .z$ and $T_\Lambda(z)\in \Lambda .z$ for all $z\in Z$.
The actions $\Gamma\acts\tilde Z$ and $\Lambda\acts\tilde Z$, coming from the identification of $\tilde{Z}$ with $\Omega /\Gamma$ and $\Omega /\Lambda$ respectively, are now conjugate to the initial actions $\Gamma\acts(X,\mu)$ and $\Lambda \acts (Y,\nu)$ via the maps $T_M\circ T_\Gamma\circ\phi_X$ and $T_N\circ T_\Lambda\circ\phi_Y$, respectively. Since $\tilde Z\subseteq \Psi^{-1}(x_0)$, the orbits of the actions $M\acts \tilde Z$ and $N\acts\tilde Z$ coincide. 
The map $\phi\coloneqq(T_N\circ T_\Lambda\circ\phi_Y)^{-1}\circ (T_M\circ T_\Gamma\circ\phi_X)$ then satisfies the conclusion of the theorem. 
\end{proof}

\section{Cofinite isomorphism}

\subsection{The amenable setting}
In this subsection we show that, for nontrivial amenable groups $A_0$ and $A_1$, the canonical wreath product actions of $A_0\wr\Z$ and $A_1\wr\Z$ are cofinitely isomorphic. 
To this end, we first state a result of Feldman-Sutherland-Zimmer \cite{Feldman1988} from the perspective of group actions which makes use of the work of Golodets-Sinel'shchikov \cite{Golodets1994} (see also Aaronson-Hamachi-Schmidt \cite{AaronsonHamachiSchmidt1995}).

{\bf Choice functions and the index cocycle.} Let $\mathcal{R}$ be an ergodic p.m.p.\ countable Borel equivalence relation on a standard probability space $(X,\mu)$ and let $\mathcal{S}$ be a Borel subequivalence relation of $\mathcal{R}$. 
By ergodicity of $\mathcal{R}$, the map sending $x\in X$ to the number of $\mathcal{S}$-classes contained in the $\mathcal{R}$-class of $x$ is constant on a conull set.
This constant is called the {\bf index} of $\mathcal{S}$ in $\mathcal{R}$. 
Let $N\in \mathbf{N} \cup \{ \aleph _0 \}$ denote this constant. 
A collection of {\bf choice functions} for $\mathcal{S}$ in $\mathcal{R}$ is an indexed collection $(\varphi_i )_{i\in I}$ of Borel functions from $X$ to $X$ whose index set $I$ has cardinality $N$, and satisfies that for each $x\in X$ the collection $(\varphi_i(x))_{i\in I}$ meets each $\mathcal{S}$-class contained in $[x]_{\mathcal{R}}$ exactly once.  
Given such a collection of choice functions, define the associated {\bf index cocycle} $\sigma:\mathcal{R}\rightarrow \mathrm{Sym}(I)$ by $\sigma(y,x)(i)=j$ if and only if $[\varphi_i(x)]_{\mathcal{S}}=[\varphi_j(y)]_{\mathcal{S}}$.

\begin{thm}[Feldman-Sutherland-Zimmer \cite{Feldman1988}, building on Golodets-Sinel'shchikov \cite{Golodets1994}]\label{FSZ}
Let $\Gamma _1$ and $\Gamma _2$ be amenable groups with amenable normal subgroups $M_1\triangleleft \Gamma _1$ and $M_2\triangleleft \Gamma _2$, and assume that $\psi:\Gamma_1/M_1\rightarrow \Gamma_2/M_2$ is an isomorphism between the associated quotient groups. 
Let $\Gamma_1\acts (X_1,\mu_1)$ and $\Gamma_2\acts(X_2,\mu_2)$ be free p.m.p.\ ergodic actions, and assume that both of the restricted actions $M_i\acts(X_i,\mu_i)$, $i=1,2$, are ergodic. 
Then there exists an orbit equivalence $\theta: X_1\rightarrow X_2$ between the $\Gamma_i$-actions such that $\theta(\gamma M_1 .x)=\psi(\gamma M_1).\theta( x)$ for all $\gamma\in\Gamma_1$ and almost every $x\in X_1$.
\end{thm}

Although this statement does not appear directly in the work of Feldman-Sutherland-Zimmer, the following proof is very similar to the proof of \cite[Theorem 3.2]{Feldman1988}.

\begin{proof}
For $i\in \{1,2\}$, let $\R_i$ be the orbit equivalence relation of the action $\Gamma_i\acts(X_i,\mu_i)$ and let $\s_i$ be the orbit equivalence relation of the ergodic action $M_i\acts(X_i,\mu_i)$. 
By choosing a section $s_i:\Gamma _i /M_i \rightarrow \Gamma _i$ of the projection map from $\Gamma _i$ to $\Gamma _i/M_i$, 
we obtain a collection of choice functions $(\varphi _{\delta M_i})_{\delta M_i\in \Gamma _i/M_i}$ for $\mathcal{S}_i$ in $\mathcal{R}_i$ indexed by $\Gamma _i/M_i$, defined by $\varphi _{\delta M_i}(x)\coloneqq s _i (\delta M_i).x$. 
The associated index cocycle $\sigma_i: \R_i\rightarrow \mathrm{Sym}(\Gamma _i/M_i )$ is then given by $\sigma_i(\gamma .x, x)=\rho_i (\gamma M_i )$ for $\gamma \in \Gamma_i$ and $x\in X$, where $\rho _i :\Gamma _i /M_i \rightarrow \mathrm{Sym}(\Gamma _i/M_i )$ denotes the right regular representation of $\Gamma _i/M_i$ given by $\rho _i(\gamma M_i)(\delta M_i)\coloneqq \delta \gamma ^{-1}M_i$ for $\gamma M_i, \delta M_i \in \Gamma _i/M_i$.

Let $\hat\psi : \rho _1 (\Gamma _1/M_1 )\rightarrow \rho _2 (\Gamma _2/M_2 )$ be the composition $\hat\psi \coloneqq \rho _2\circ\psi \circ \rho _1  ^{-1}$, and let $\hat\sigma _1 : \R_1\rightarrow \rho _2(\Gamma_2/M_2)$ be given by $\hat\sigma _1 \coloneqq\hat\psi\circ\sigma_1$. 
Since $\mathcal{S}_0$ and $\mathcal{S}_1$ are ergodic, the cocycles $\sigma_2$ and $\hat \sigma _1$ each take values and have dense range (in the sense of \cite[p. 458]{Golodets1994}) in the group $\rho _2(\Gamma _2/M_2)$, and both $\R_0$ and $\R_1$ are hyperfinite by Ornstein-Weiss \cite{OW87}.
Hence, by \cite[Lemma 1.12]{Golodets1994} there is an orbit equivalence $\bar\theta: X_1\rightarrow X_2$ from $\R_1$ to $\R_2$ such that $\hat\sigma _1\circ(\bar\theta\times\bar\theta)^{-1}$ and $\sigma_2$ are cohomologous as $\rho _2 (\Gamma _2/M_2 )$-valued cocycles. 
By \cite[Theorem 1.6b]{Feldman1988}, after replacing $\bar\theta$ by its composition with an element of the full group of $\R_2$, we can assume that $(\bar\theta \times \bar\theta )(\s_1 )=\s_2$. 

Since $\hat\sigma _1\circ(\bar\theta\times\bar\theta)^{-1}$ and $\sigma _2$ are cohomologous, after discarding a null set we may find a measurable function $f: X_2\rightarrow \rho _2(\Gamma_2/M_2)$ satisfying 
\begin{equation}\label{eqn:cohomologous}
f(y)\hat\sigma _1(\bar\theta ^{-1}(y),\bar\theta ^{-1} (x))f(x)^{-1}= \sigma_2(y,x)    
\end{equation} 
for all $(y,x)\in\mathcal{R}_2$. 
For $(y,x)\in\s_2$, we have $f(y)=f(x)$ since both cocycles are trivial on $\s_2$. 
By ergodicity of $\s_2$, after discarding another null set there exists some $\lambda \in \Gamma _2$ such that $f(x)=\rho_2(\lambda M_2)$ for all $x\in X_2$. 
The map $x\mapsto \lambda  ^{-1}\bar\theta (x)$ is bijective; denote its inverse by $\theta$, so that $\theta ^{-1}(x)=\lambda ^{-1}\bar\theta (x)$ for all $x\in X_2$. 
Then $\psi\circ\rho_1^{-1}\circ\sigma_1=\rho_2^{-1}\circ \sigma_2\circ (\theta\times\theta)$ by \eqref{eqn:cohomologous} and the definition of $\hat{\psi}$. Therefore
\[
\theta (\gamma M_1 .x)
=(\rho_2^{-1}\circ\sigma _2) (\theta (\gamma M_1.x),\theta (x)).\theta (x) 
= \psi ((\rho _1^{-1}\circ \sigma _1)(\gamma M_1.x ,x)).\theta (x)
=\psi (\gamma M_1 ).\theta (x)
\]
for all $x\in X_1$ and $\gamma \in \Gamma _1$.
\end{proof}

Using the preceding theorem, we obtain our first nontrivial examples of bi-cofinitely equivariant maps. 

\begin{thm} \label{cofAmenable}
Let $\Gamma$ be amenable and let $A_0$ and $A_1$ be two nontrivial amenable groups acting in a free p.m.p.\ ergodic manner on $(Y,\nu)$ and $(Z,\eta)$, respectively. Then, with respect to these actions, there exists a bi-cofinitely equivariant measure space isomorphism between $(Y^\Gamma,\nu^\Gamma)$ and $(Z^\Gamma,\eta^\Gamma)$.
\end{thm}

\begin{proof}
The wreath product groups $A_0\wr\Gamma$ and $A_1\wr \Gamma$ are amenable since $\Gamma$, $A_0$, and $A_1$ are all amenable. 
Consider the canonical wreath product actions of $A_0\wr\Gamma$ and $A_1\wr \Gamma$ on $(Y^\Gamma,\nu^\Gamma)$ and $(Z^\Gamma,\eta^\Gamma)$ respectively. 
The normal subgroups $\tilde{A}_0\coloneqq \bigoplus _{\Gamma}A_0$ and $\tilde{A}_1\coloneqq \bigoplus _{\Gamma}A_1$ each act ergodically.  
Let $\psi :(A_0\wr\Gamma )/\tilde{A}_0\rightarrow (A_1\wr\Gamma )/\tilde{A}_1$ denote the isomorphism given by $\psi(\gamma \tilde{A}_0)=\gamma \tilde{A}_1$ for $\gamma \in \Gamma$.

By Theorem \ref{FSZ} there exists an orbit equivalence $\theta: Y^\Gamma\rightarrow Z^\Gamma$ such that $\theta(\gamma \tilde{A} .y)=\gamma \tilde{A}_1.\theta( y)$ for all $\gamma\in\Gamma$ and almost every $y\in Y^\Gamma$. Therefore, $\theta$ is a bi-cofinitely equivariant measure space isomorphism.
\end{proof}

\subsection{Free products and finite index extensions}

In order to convey the main ideas of the proof of Theorem \ref{thm:cofequivfreefactorgroupoids}, we first present the proof in the setting of group actions before moving to the general groupoid setting.

\begin{thm} \label{cofFinCosetReprError}
Let $B\curvearrowright (Y,\nu )$ and $C\curvearrowright (Z,\eta )$ be p.m.p.\ actions of countable groups $B$ and $C$ on $(Y,\nu)$ and $(Z,\eta)$, respectively. 
Let $\Lambda$ be a subgroup of the countable group $\Gamma$ and suppose that there exist
\begin{itemize}
    \item[(i)] a set $S$ of left coset representatives for $\Lambda$ in $\Gamma$ satisfying $|\gamma S\triangle S|<\infty$ for all $\gamma\in \Gamma$, and
    \item[(ii)] a bi-cofinitely $\Lambda$-equivariant measure space isomorphism  $\phi _{\Lambda} :(Y^\Lambda,\nu ^\Lambda)\rightarrow(Z^\Lambda,\eta ^\Lambda)$.    
\end{itemize}
Then, with respect to these actions of $B$ and $C$, there exists a bi-cofinitely $\Gamma$-equivariant measure space isomorphism $\phi :(Y^\Gamma,\nu ^\Gamma)\rightarrow(Z^\Gamma,\eta ^\Gamma)$.
\end{thm}

\begin{proof}
A good example to keep in mind throughout the proof is the case where the group $\Gamma$ is the free group with free generators $a$ and $b$, the subgroup $\Lambda$ is the cyclic subgroup generated by $a$, the set $S$ consists of the identity element together with all reduced words whose rightmost letter is $b$ or $b^{-1}$, and the map $\phi _{\Lambda}$ is obtained from Theorem \ref{cofAmenable}. 

Let $\sigma :\Gamma \to S$ be the map sending the coset $\alpha \Lambda$ to its representative in $S$, and let $\rho : \Gamma \times \Gamma /\Lambda \rightarrow \Lambda$ denote the associated cocycle into $\Lambda$, defined by $\rho (\gamma , \alpha \Lambda ) = \sigma (\gamma \alpha )^{-1}\gamma \sigma (\alpha )$. Define $\phi :Y^{\Gamma} \rightarrow Z^{\Gamma}$ by 
\[
\phi (y)_\gamma=\phi _{\Lambda} \big((\sigma (\gamma )^{-1}.y)|_{\Lambda}\big)_{\sigma (\gamma )^{-1}\gamma} .
\]
See Figure \ref{fig:cof} for an illustration. 
Note that $\phi$ is bijective with inverse given by $\phi ^{-1}(z)_\gamma = \phi _{\Lambda}^{-1} ((\sigma (\gamma )^{-1}.z)|_{\Lambda})_{\sigma (\gamma )^{-1}\gamma}$. 

\tikzset{every picture/.style={line width=0.75pt}} 
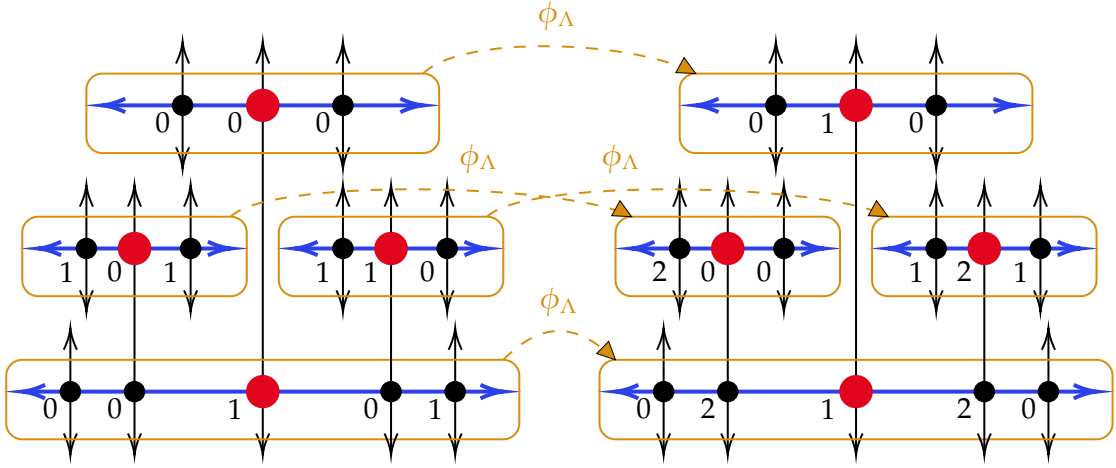
\begin{figure}[ht]
    \centering
    \begin{tikzpicture}[x=0.75pt,y=0.75pt,yscale=-1,xscale=1]

\draw [color={rgb, 255:red, 43; green, 64; blue, 233 }  ,draw opacity=1 ][line width=1.5]    (3,240) -- (237,240) ;
\draw [shift={(240,240)}, rotate = 180] [color={rgb, 255:red, 43; green, 64; blue, 233 }  ,draw opacity=1 ][line width=1.5]    (11.37,-3.42) .. controls (7.23,-1.45) and (3.44,-0.31) .. (0,0) .. controls (3.44,0.31) and (7.23,1.45) .. (11.37,3.42)   ;
\draw [shift={(0,240)}, rotate = 0] [color={rgb, 255:red, 43; green, 64; blue, 233 }  ,draw opacity=1 ][line width=1.5]    (11.37,-3.42) .. controls (7.23,-1.45) and (3.44,-0.31) .. (0,0) .. controls (3.44,0.31) and (7.23,1.45) .. (11.37,3.42)   ;
\draw [color={rgb, 255:red, 43; green, 64; blue, 233 }  ,draw opacity=1 ][line width=1.5]    (11,168) -- (101,168) ;
\draw [shift={(104,168)}, rotate = 180] [color={rgb, 255:red, 43; green, 64; blue, 233 }  ,draw opacity=1 ][line width=1.5]    (11.37,-3.42) .. controls (7.23,-1.45) and (3.44,-0.31) .. (0,0) .. controls (3.44,0.31) and (7.23,1.45) .. (11.37,3.42)   ;
\draw [shift={(8,168)}, rotate = 0] [color={rgb, 255:red, 43; green, 64; blue, 233 }  ,draw opacity=1 ][line width=1.5]    (11.37,-3.42) .. controls (7.23,-1.45) and (3.44,-0.31) .. (0,0) .. controls (3.44,0.31) and (7.23,1.45) .. (11.37,3.42)   ;
\draw [color={rgb, 255:red, 43; green, 64; blue, 233 }  ,draw opacity=1 ][line width=1.5]    (43,96) -- (197,96) ;
\draw [shift={(200,96)}, rotate = 180] [color={rgb, 255:red, 43; green, 64; blue, 233 }  ,draw opacity=1 ][line width=1.5]    (11.37,-3.42) .. controls (7.23,-1.45) and (3.44,-0.31) .. (0,0) .. controls (3.44,0.31) and (7.23,1.45) .. (11.37,3.42)   ;
\draw [shift={(40,96)}, rotate = 0] [color={rgb, 255:red, 43; green, 64; blue, 233 }  ,draw opacity=1 ][line width=1.5]    (11.37,-3.42) .. controls (7.23,-1.45) and (3.44,-0.31) .. (0,0) .. controls (3.44,0.31) and (7.23,1.45) .. (11.37,3.42)   ;
\draw    (120,66) -- (120,270) ;
\draw [shift={(120,272)}, rotate = 270] [color={rgb, 255:red, 0; green, 0; blue, 0 }  ][line width=0.75]    (10.93,-3.29) .. controls (6.95,-1.4) and (3.31,-0.3) .. (0,0) .. controls (3.31,0.3) and (6.95,1.4) .. (10.93,3.29)   ;
\draw [shift={(120,64)}, rotate = 90] [color={rgb, 255:red, 0; green, 0; blue, 0 }  ][line width=0.75]    (10.93,-3.29) .. controls (6.95,-1.4) and (3.31,-0.3) .. (0,0) .. controls (3.31,0.3) and (6.95,1.4) .. (10.93,3.29)   ;
\draw    (56,138) -- (56,270) ;
\draw [shift={(56,272)}, rotate = 270] [color={rgb, 255:red, 0; green, 0; blue, 0 }  ][line width=0.75]    (10.93,-3.29) .. controls (6.95,-1.4) and (3.31,-0.3) .. (0,0) .. controls (3.31,0.3) and (6.95,1.4) .. (10.93,3.29)   ;
\draw [shift={(56,136)}, rotate = 90] [color={rgb, 255:red, 0; green, 0; blue, 0 }  ][line width=0.75]    (10.93,-3.29) .. controls (6.95,-1.4) and (3.31,-0.3) .. (0,0) .. controls (3.31,0.3) and (6.95,1.4) .. (10.93,3.29)   ;
\draw    (24,210) -- (24,270) ;
\draw [shift={(24,272)}, rotate = 270] [color={rgb, 255:red, 0; green, 0; blue, 0 }  ][line width=0.75]    (10.93,-3.29) .. controls (6.95,-1.4) and (3.31,-0.3) .. (0,0) .. controls (3.31,0.3) and (6.95,1.4) .. (10.93,3.29)   ;
\draw [shift={(24,208)}, rotate = 90] [color={rgb, 255:red, 0; green, 0; blue, 0 }  ][line width=0.75]    (10.93,-3.29) .. controls (6.95,-1.4) and (3.31,-0.3) .. (0,0) .. controls (3.31,0.3) and (6.95,1.4) .. (10.93,3.29)   ;
\draw    (216,210) -- (216,270) ;
\draw [shift={(216,272)}, rotate = 270] [color={rgb, 255:red, 0; green, 0; blue, 0 }  ][line width=0.75]    (10.93,-3.29) .. controls (6.95,-1.4) and (3.31,-0.3) .. (0,0) .. controls (3.31,0.3) and (6.95,1.4) .. (10.93,3.29)   ;
\draw [shift={(216,208)}, rotate = 90] [color={rgb, 255:red, 0; green, 0; blue, 0 }  ][line width=0.75]    (10.93,-3.29) .. controls (6.95,-1.4) and (3.31,-0.3) .. (0,0) .. controls (3.31,0.3) and (6.95,1.4) .. (10.93,3.29)   ;
\draw    (32,138) -- (32,198) ;
\draw [shift={(32,200)}, rotate = 270] [color={rgb, 255:red, 0; green, 0; blue, 0 }  ][line width=0.75]    (10.93,-3.29) .. controls (6.95,-1.4) and (3.31,-0.3) .. (0,0) .. controls (3.31,0.3) and (6.95,1.4) .. (10.93,3.29)   ;
\draw [shift={(32,136)}, rotate = 90] [color={rgb, 255:red, 0; green, 0; blue, 0 }  ][line width=0.75]    (10.93,-3.29) .. controls (6.95,-1.4) and (3.31,-0.3) .. (0,0) .. controls (3.31,0.3) and (6.95,1.4) .. (10.93,3.29)   ;
\draw    (84,138) -- (84,198) ;
\draw [shift={(84,200)}, rotate = 270] [color={rgb, 255:red, 0; green, 0; blue, 0 }  ][line width=0.75]    (10.93,-3.29) .. controls (6.95,-1.4) and (3.31,-0.3) .. (0,0) .. controls (3.31,0.3) and (6.95,1.4) .. (10.93,3.29)   ;
\draw [shift={(84,136)}, rotate = 90] [color={rgb, 255:red, 0; green, 0; blue, 0 }  ][line width=0.75]    (10.93,-3.29) .. controls (6.95,-1.4) and (3.31,-0.3) .. (0,0) .. controls (3.31,0.3) and (6.95,1.4) .. (10.93,3.29)   ;
\draw    (80,66) -- (80,126) ;
\draw [shift={(80,128)}, rotate = 270] [color={rgb, 255:red, 0; green, 0; blue, 0 }  ][line width=0.75]    (10.93,-3.29) .. controls (6.95,-1.4) and (3.31,-0.3) .. (0,0) .. controls (3.31,0.3) and (6.95,1.4) .. (10.93,3.29)   ;
\draw [shift={(80,64)}, rotate = 90] [color={rgb, 255:red, 0; green, 0; blue, 0 }  ][line width=0.75]    (10.93,-3.29) .. controls (6.95,-1.4) and (3.31,-0.3) .. (0,0) .. controls (3.31,0.3) and (6.95,1.4) .. (10.93,3.29)   ;
\draw    (160,66) -- (160,126) ;
\draw [shift={(160,128)}, rotate = 270] [color={rgb, 255:red, 0; green, 0; blue, 0 }  ][line width=0.75]    (10.93,-3.29) .. controls (6.95,-1.4) and (3.31,-0.3) .. (0,0) .. controls (3.31,0.3) and (6.95,1.4) .. (10.93,3.29)   ;
\draw [shift={(160,64)}, rotate = 90] [color={rgb, 255:red, 0; green, 0; blue, 0 }  ][line width=0.75]    (10.93,-3.29) .. controls (6.95,-1.4) and (3.31,-0.3) .. (0,0) .. controls (3.31,0.3) and (6.95,1.4) .. (10.93,3.29)   ;
\draw [line width=1.5]    (24,240) ;
\draw [shift={(24,240)}, rotate = 0] [color={rgb, 255:red, 0; green, 0; blue, 0 }  ][fill={rgb, 255:red, 0; green, 0; blue, 0 }  ][line width=1.5]      (0, 0) circle [x radius= 4.36, y radius= 4.36]   ;
\draw [line width=1.5]    (56,240) ;
\draw [shift={(56,240)}, rotate = 0] [color={rgb, 255:red, 0; green, 0; blue, 0 }  ][fill={rgb, 255:red, 0; green, 0; blue, 0 }  ][line width=1.5]      (0, 0) circle [x radius= 4.36, y radius= 4.36]   ;
\draw [line width=1.5]    (32,168) ;
\draw [shift={(32,168)}, rotate = 0] [color={rgb, 255:red, 0; green, 0; blue, 0 }  ][fill={rgb, 255:red, 0; green, 0; blue, 0 }  ][line width=1.5]      (0, 0) circle [x radius= 4.36, y radius= 4.36]   ;
\draw [line width=1.5]    (84,168) ;
\draw [shift={(84,168)}, rotate = 0] [color={rgb, 255:red, 0; green, 0; blue, 0 }  ][fill={rgb, 255:red, 0; green, 0; blue, 0 }  ][line width=1.5]      (0, 0) circle [x radius= 4.36, y radius= 4.36]   ;
\draw [line width=1.5]    (80,96) ;
\draw [shift={(80,96)}, rotate = 0] [color={rgb, 255:red, 0; green, 0; blue, 0 }  ][fill={rgb, 255:red, 0; green, 0; blue, 0 }  ][line width=1.5]      (0, 0) circle [x radius= 4.36, y radius= 4.36]   ;
\draw [line width=1.5]    (160,96) ;
\draw [shift={(160,96)}, rotate = 0] [color={rgb, 255:red, 0; green, 0; blue, 0 }  ][fill={rgb, 255:red, 0; green, 0; blue, 0 }  ][line width=1.5]      (0, 0) circle [x radius= 4.36, y radius= 4.36]   ;
\draw [line width=1.5]    (184,240) ;
\draw [shift={(184,240)}, rotate = 0] [color={rgb, 255:red, 0; green, 0; blue, 0 }  ][fill={rgb, 255:red, 0; green, 0; blue, 0 }  ][line width=1.5]      (0, 0) circle [x radius= 4.36, y radius= 4.36]   ;
\draw [line width=1.5]    (216,240) ;
\draw [shift={(216,240)}, rotate = 0] [color={rgb, 255:red, 0; green, 0; blue, 0 }  ][fill={rgb, 255:red, 0; green, 0; blue, 0 }  ][line width=1.5]      (0, 0) circle [x radius= 4.36, y radius= 4.36]   ;
\draw [color={rgb, 255:red, 226; green, 5; blue, 31 }  ,draw opacity=1 ][line width=3]    (120,240) ;
\draw [shift={(120,240)}, rotate = 0] [color={rgb, 255:red, 226; green, 5; blue, 31 }  ,draw opacity=1 ][fill={rgb, 255:red, 226; green, 5; blue, 31 }  ,fill opacity=1 ][line width=3]      (0, 0) circle [x radius= 6.37, y radius= 6.37]   ;
\draw [color={rgb, 255:red, 226; green, 5; blue, 31 }  ,draw opacity=1 ][line width=3]    (56,168) ;
\draw [shift={(56,168)}, rotate = 0] [color={rgb, 255:red, 226; green, 5; blue, 31 }  ,draw opacity=1 ][fill={rgb, 255:red, 226; green, 5; blue, 31 }  ,fill opacity=1 ][line width=3]      (0, 0) circle [x radius= 6.37, y radius= 6.37]   ;
\draw [color={rgb, 255:red, 226; green, 5; blue, 31 }  ,draw opacity=1 ][line width=3]    (120,96) ;
\draw [shift={(120,96)}, rotate = 0] [color={rgb, 255:red, 226; green, 5; blue, 31 }  ,draw opacity=1 ][fill={rgb, 255:red, 226; green, 5; blue, 31 }  ,fill opacity=1 ][line width=3]      (0, 0) circle [x radius= 6.37, y radius= 6.37]   ;
\draw  [color={rgb, 255:red, 218; green, 140; blue, 10 }  ,draw opacity=1 ] (-8,232) .. controls (-8,227.58) and (-4.42,224) .. (0,224) -- (240,224) .. controls (244.42,224) and (248,227.58) .. (248,232) -- (248,256) .. controls (248,260.42) and (244.42,264) .. (240,264) -- (0,264) .. controls (-4.42,264) and (-8,260.42) .. (-8,256) -- cycle ;
\draw  [color={rgb, 255:red, 218; green, 140; blue, 10 }  ,draw opacity=1 ] (0,160) .. controls (0,155.58) and (3.58,152) .. (8,152) -- (104,152) .. controls (108.42,152) and (112,155.58) .. (112,160) -- (112,184) .. controls (112,188.42) and (108.42,192) .. (104,192) -- (8,192) .. controls (3.58,192) and (0,188.42) .. (0,184) -- cycle ;
\draw  [color={rgb, 255:red, 218; green, 140; blue, 10 }  ,draw opacity=1 ] (32,88) .. controls (32,83.58) and (35.58,80) .. (40,80) -- (200,80) .. controls (204.42,80) and (208,83.58) .. (208,88) -- (208,112) .. controls (208,116.42) and (204.42,120) .. (200,120) -- (40,120) .. controls (35.58,120) and (32,116.42) .. (32,112) -- cycle ;
\draw [color={rgb, 255:red, 43; green, 64; blue, 233 }  ,draw opacity=1 ][line width=1.5]    (139,168) -- (229,168) ;
\draw [shift={(232,168)}, rotate = 180] [color={rgb, 255:red, 43; green, 64; blue, 233 }  ,draw opacity=1 ][line width=1.5]    (11.37,-3.42) .. controls (7.23,-1.45) and (3.44,-0.31) .. (0,0) .. controls (3.44,0.31) and (7.23,1.45) .. (11.37,3.42)   ;
\draw [shift={(136,168)}, rotate = 0] [color={rgb, 255:red, 43; green, 64; blue, 233 }  ,draw opacity=1 ][line width=1.5]    (11.37,-3.42) .. controls (7.23,-1.45) and (3.44,-0.31) .. (0,0) .. controls (3.44,0.31) and (7.23,1.45) .. (11.37,3.42)   ;
\draw    (184,138) -- (184,270) ;
\draw [shift={(184,272)}, rotate = 270] [color={rgb, 255:red, 0; green, 0; blue, 0 }  ][line width=0.75]    (10.93,-3.29) .. controls (6.95,-1.4) and (3.31,-0.3) .. (0,0) .. controls (3.31,0.3) and (6.95,1.4) .. (10.93,3.29)   ;
\draw [shift={(184,136)}, rotate = 90] [color={rgb, 255:red, 0; green, 0; blue, 0 }  ][line width=0.75]    (10.93,-3.29) .. controls (6.95,-1.4) and (3.31,-0.3) .. (0,0) .. controls (3.31,0.3) and (6.95,1.4) .. (10.93,3.29)   ;
\draw    (160,138) -- (160,198) ;
\draw [shift={(160,200)}, rotate = 270] [color={rgb, 255:red, 0; green, 0; blue, 0 }  ][line width=0.75]    (10.93,-3.29) .. controls (6.95,-1.4) and (3.31,-0.3) .. (0,0) .. controls (3.31,0.3) and (6.95,1.4) .. (10.93,3.29)   ;
\draw [shift={(160,136)}, rotate = 90] [color={rgb, 255:red, 0; green, 0; blue, 0 }  ][line width=0.75]    (10.93,-3.29) .. controls (6.95,-1.4) and (3.31,-0.3) .. (0,0) .. controls (3.31,0.3) and (6.95,1.4) .. (10.93,3.29)   ;
\draw    (212,138) -- (212,198) ;
\draw [shift={(212,200)}, rotate = 270] [color={rgb, 255:red, 0; green, 0; blue, 0 }  ][line width=0.75]    (10.93,-3.29) .. controls (6.95,-1.4) and (3.31,-0.3) .. (0,0) .. controls (3.31,0.3) and (6.95,1.4) .. (10.93,3.29)   ;
\draw [shift={(212,136)}, rotate = 90] [color={rgb, 255:red, 0; green, 0; blue, 0 }  ][line width=0.75]    (10.93,-3.29) .. controls (6.95,-1.4) and (3.31,-0.3) .. (0,0) .. controls (3.31,0.3) and (6.95,1.4) .. (10.93,3.29)   ;
\draw [line width=1.5]    (160,168) ;
\draw [shift={(160,168)}, rotate = 0] [color={rgb, 255:red, 0; green, 0; blue, 0 }  ][fill={rgb, 255:red, 0; green, 0; blue, 0 }  ][line width=1.5]      (0, 0) circle [x radius= 4.36, y radius= 4.36]   ;
\draw [line width=1.5]    (212,168) ;
\draw [shift={(212,168)}, rotate = 0] [color={rgb, 255:red, 0; green, 0; blue, 0 }  ][fill={rgb, 255:red, 0; green, 0; blue, 0 }  ][line width=1.5]      (0, 0) circle [x radius= 4.36, y radius= 4.36]   ;
\draw [color={rgb, 255:red, 226; green, 5; blue, 31 }  ,draw opacity=1 ][line width=3]    (184,168) ;
\draw [shift={(184,168)}, rotate = 0] [color={rgb, 255:red, 226; green, 5; blue, 31 }  ,draw opacity=1 ][fill={rgb, 255:red, 226; green, 5; blue, 31 }  ,fill opacity=1 ][line width=3]      (0, 0) circle [x radius= 6.37, y radius= 6.37]   ;
\draw  [color={rgb, 255:red, 218; green, 140; blue, 10 }  ,draw opacity=1 ] (128,160) .. controls (128,155.58) and (131.58,152) .. (136,152) -- (232,152) .. controls (236.42,152) and (240,155.58) .. (240,160) -- (240,184) .. controls (240,188.42) and (236.42,192) .. (232,192) -- (136,192) .. controls (131.58,192) and (128,188.42) .. (128,184) -- cycle ;
\draw [color={rgb, 255:red, 43; green, 64; blue, 233 }  ,draw opacity=1 ][line width=1.5]    (299,240) -- (533,240) ;
\draw [shift={(536,240)}, rotate = 180] [color={rgb, 255:red, 43; green, 64; blue, 233 }  ,draw opacity=1 ][line width=1.5]    (11.37,-3.42) .. controls (7.23,-1.45) and (3.44,-0.31) .. (0,0) .. controls (3.44,0.31) and (7.23,1.45) .. (11.37,3.42)   ;
\draw [shift={(296,240)}, rotate = 0] [color={rgb, 255:red, 43; green, 64; blue, 233 }  ,draw opacity=1 ][line width=1.5]    (11.37,-3.42) .. controls (7.23,-1.45) and (3.44,-0.31) .. (0,0) .. controls (3.44,0.31) and (7.23,1.45) .. (11.37,3.42)   ;
\draw [color={rgb, 255:red, 43; green, 64; blue, 233 }  ,draw opacity=1 ][line width=1.5]    (307,168) -- (397,168) ;
\draw [shift={(400,168)}, rotate = 180] [color={rgb, 255:red, 43; green, 64; blue, 233 }  ,draw opacity=1 ][line width=1.5]    (11.37,-3.42) .. controls (7.23,-1.45) and (3.44,-0.31) .. (0,0) .. controls (3.44,0.31) and (7.23,1.45) .. (11.37,3.42)   ;
\draw [shift={(304,168)}, rotate = 0] [color={rgb, 255:red, 43; green, 64; blue, 233 }  ,draw opacity=1 ][line width=1.5]    (11.37,-3.42) .. controls (7.23,-1.45) and (3.44,-0.31) .. (0,0) .. controls (3.44,0.31) and (7.23,1.45) .. (11.37,3.42)   ;
\draw [color={rgb, 255:red, 43; green, 64; blue, 233 }  ,draw opacity=1 ][line width=1.5]    (339,96) -- (493,96) ;
\draw [shift={(496,96)}, rotate = 180] [color={rgb, 255:red, 43; green, 64; blue, 233 }  ,draw opacity=1 ][line width=1.5]    (11.37,-3.42) .. controls (7.23,-1.45) and (3.44,-0.31) .. (0,0) .. controls (3.44,0.31) and (7.23,1.45) .. (11.37,3.42)   ;
\draw [shift={(336,96)}, rotate = 0] [color={rgb, 255:red, 43; green, 64; blue, 233 }  ,draw opacity=1 ][line width=1.5]    (11.37,-3.42) .. controls (7.23,-1.45) and (3.44,-0.31) .. (0,0) .. controls (3.44,0.31) and (7.23,1.45) .. (11.37,3.42)   ;
\draw    (416,66) -- (416,270) ;
\draw [shift={(416,272)}, rotate = 270] [color={rgb, 255:red, 0; green, 0; blue, 0 }  ][line width=0.75]    (10.93,-3.29) .. controls (6.95,-1.4) and (3.31,-0.3) .. (0,0) .. controls (3.31,0.3) and (6.95,1.4) .. (10.93,3.29)   ;
\draw [shift={(416,64)}, rotate = 90] [color={rgb, 255:red, 0; green, 0; blue, 0 }  ][line width=0.75]    (10.93,-3.29) .. controls (6.95,-1.4) and (3.31,-0.3) .. (0,0) .. controls (3.31,0.3) and (6.95,1.4) .. (10.93,3.29)   ;
\draw    (352,138) -- (352,270) ;
\draw [shift={(352,272)}, rotate = 270] [color={rgb, 255:red, 0; green, 0; blue, 0 }  ][line width=0.75]    (10.93,-3.29) .. controls (6.95,-1.4) and (3.31,-0.3) .. (0,0) .. controls (3.31,0.3) and (6.95,1.4) .. (10.93,3.29)   ;
\draw [shift={(352,136)}, rotate = 90] [color={rgb, 255:red, 0; green, 0; blue, 0 }  ][line width=0.75]    (10.93,-3.29) .. controls (6.95,-1.4) and (3.31,-0.3) .. (0,0) .. controls (3.31,0.3) and (6.95,1.4) .. (10.93,3.29)   ;
\draw    (320,210) -- (320,270) ;
\draw [shift={(320,272)}, rotate = 270] [color={rgb, 255:red, 0; green, 0; blue, 0 }  ][line width=0.75]    (10.93,-3.29) .. controls (6.95,-1.4) and (3.31,-0.3) .. (0,0) .. controls (3.31,0.3) and (6.95,1.4) .. (10.93,3.29)   ;
\draw [shift={(320,208)}, rotate = 90] [color={rgb, 255:red, 0; green, 0; blue, 0 }  ][line width=0.75]    (10.93,-3.29) .. controls (6.95,-1.4) and (3.31,-0.3) .. (0,0) .. controls (3.31,0.3) and (6.95,1.4) .. (10.93,3.29)   ;
\draw    (512,210) -- (512,270) ;
\draw [shift={(512,272)}, rotate = 270] [color={rgb, 255:red, 0; green, 0; blue, 0 }  ][line width=0.75]    (10.93,-3.29) .. controls (6.95,-1.4) and (3.31,-0.3) .. (0,0) .. controls (3.31,0.3) and (6.95,1.4) .. (10.93,3.29)   ;
\draw [shift={(512,208)}, rotate = 90] [color={rgb, 255:red, 0; green, 0; blue, 0 }  ][line width=0.75]    (10.93,-3.29) .. controls (6.95,-1.4) and (3.31,-0.3) .. (0,0) .. controls (3.31,0.3) and (6.95,1.4) .. (10.93,3.29)   ;
\draw    (328,138) -- (328,198) ;
\draw [shift={(328,200)}, rotate = 270] [color={rgb, 255:red, 0; green, 0; blue, 0 }  ][line width=0.75]    (10.93,-3.29) .. controls (6.95,-1.4) and (3.31,-0.3) .. (0,0) .. controls (3.31,0.3) and (6.95,1.4) .. (10.93,3.29)   ;
\draw [shift={(328,136)}, rotate = 90] [color={rgb, 255:red, 0; green, 0; blue, 0 }  ][line width=0.75]    (10.93,-3.29) .. controls (6.95,-1.4) and (3.31,-0.3) .. (0,0) .. controls (3.31,0.3) and (6.95,1.4) .. (10.93,3.29)   ;
\draw    (380,138) -- (380,198) ;
\draw [shift={(380,200)}, rotate = 270] [color={rgb, 255:red, 0; green, 0; blue, 0 }  ][line width=0.75]    (10.93,-3.29) .. controls (6.95,-1.4) and (3.31,-0.3) .. (0,0) .. controls (3.31,0.3) and (6.95,1.4) .. (10.93,3.29)   ;
\draw [shift={(380,136)}, rotate = 90] [color={rgb, 255:red, 0; green, 0; blue, 0 }  ][line width=0.75]    (10.93,-3.29) .. controls (6.95,-1.4) and (3.31,-0.3) .. (0,0) .. controls (3.31,0.3) and (6.95,1.4) .. (10.93,3.29)   ;
\draw    (376,66) -- (376,126) ;
\draw [shift={(376,128)}, rotate = 270] [color={rgb, 255:red, 0; green, 0; blue, 0 }  ][line width=0.75]    (10.93,-3.29) .. controls (6.95,-1.4) and (3.31,-0.3) .. (0,0) .. controls (3.31,0.3) and (6.95,1.4) .. (10.93,3.29)   ;
\draw [shift={(376,64)}, rotate = 90] [color={rgb, 255:red, 0; green, 0; blue, 0 }  ][line width=0.75]    (10.93,-3.29) .. controls (6.95,-1.4) and (3.31,-0.3) .. (0,0) .. controls (3.31,0.3) and (6.95,1.4) .. (10.93,3.29)   ;
\draw    (456,66) -- (456,126) ;
\draw [shift={(456,128)}, rotate = 270] [color={rgb, 255:red, 0; green, 0; blue, 0 }  ][line width=0.75]    (10.93,-3.29) .. controls (6.95,-1.4) and (3.31,-0.3) .. (0,0) .. controls (3.31,0.3) and (6.95,1.4) .. (10.93,3.29)   ;
\draw [shift={(456,64)}, rotate = 90] [color={rgb, 255:red, 0; green, 0; blue, 0 }  ][line width=0.75]    (10.93,-3.29) .. controls (6.95,-1.4) and (3.31,-0.3) .. (0,0) .. controls (3.31,0.3) and (6.95,1.4) .. (10.93,3.29)   ;
\draw [line width=1.5]    (320,240) ;
\draw [shift={(320,240)}, rotate = 0] [color={rgb, 255:red, 0; green, 0; blue, 0 }  ][fill={rgb, 255:red, 0; green, 0; blue, 0 }  ][line width=1.5]      (0, 0) circle [x radius= 4.36, y radius= 4.36]   ;
\draw [line width=1.5]    (352,240) ;
\draw [shift={(352,240)}, rotate = 0] [color={rgb, 255:red, 0; green, 0; blue, 0 }  ][fill={rgb, 255:red, 0; green, 0; blue, 0 }  ][line width=1.5]      (0, 0) circle [x radius= 4.36, y radius= 4.36]   ;
\draw [line width=1.5]    (328,168) ;
\draw [shift={(328,168)}, rotate = 0] [color={rgb, 255:red, 0; green, 0; blue, 0 }  ][fill={rgb, 255:red, 0; green, 0; blue, 0 }  ][line width=1.5]      (0, 0) circle [x radius= 4.36, y radius= 4.36]   ;
\draw [line width=1.5]    (380,168) ;
\draw [shift={(380,168)}, rotate = 0] [color={rgb, 255:red, 0; green, 0; blue, 0 }  ][fill={rgb, 255:red, 0; green, 0; blue, 0 }  ][line width=1.5]      (0, 0) circle [x radius= 4.36, y radius= 4.36]   ;
\draw [line width=1.5]    (376,96) ;
\draw [shift={(376,96)}, rotate = 0] [color={rgb, 255:red, 0; green, 0; blue, 0 }  ][fill={rgb, 255:red, 0; green, 0; blue, 0 }  ][line width=1.5]      (0, 0) circle [x radius= 4.36, y radius= 4.36]   ;
\draw [line width=1.5]    (456,96) ;
\draw [shift={(456,96)}, rotate = 0] [color={rgb, 255:red, 0; green, 0; blue, 0 }  ][fill={rgb, 255:red, 0; green, 0; blue, 0 }  ][line width=1.5]      (0, 0) circle [x radius= 4.36, y radius= 4.36]   ;
\draw [line width=1.5]    (480,240) ;
\draw [shift={(480,240)}, rotate = 0] [color={rgb, 255:red, 0; green, 0; blue, 0 }  ][fill={rgb, 255:red, 0; green, 0; blue, 0 }  ][line width=1.5]      (0, 0) circle [x radius= 4.36, y radius= 4.36]   ;
\draw [line width=1.5]    (512,240) ;
\draw [shift={(512,240)}, rotate = 0] [color={rgb, 255:red, 0; green, 0; blue, 0 }  ][fill={rgb, 255:red, 0; green, 0; blue, 0 }  ][line width=1.5]      (0, 0) circle [x radius= 4.36, y radius= 4.36]   ;
\draw [color={rgb, 255:red, 226; green, 5; blue, 31 }  ,draw opacity=1 ][line width=3]    (416,240) ;
\draw [shift={(416,240)}, rotate = 0] [color={rgb, 255:red, 226; green, 5; blue, 31 }  ,draw opacity=1 ][fill={rgb, 255:red, 226; green, 5; blue, 31 }  ,fill opacity=1 ][line width=3]      (0, 0) circle [x radius= 6.37, y radius= 6.37]   ;
\draw [color={rgb, 255:red, 226; green, 5; blue, 31 }  ,draw opacity=1 ][line width=3]    (352,168) ;
\draw [shift={(352,168)}, rotate = 0] [color={rgb, 255:red, 226; green, 5; blue, 31 }  ,draw opacity=1 ][fill={rgb, 255:red, 226; green, 5; blue, 31 }  ,fill opacity=1 ][line width=3]      (0, 0) circle [x radius= 6.37, y radius= 6.37]   ;
\draw [color={rgb, 255:red, 226; green, 5; blue, 31 }  ,draw opacity=1 ][line width=3]    (416,96) ;
\draw [shift={(416,96)}, rotate = 0] [color={rgb, 255:red, 226; green, 5; blue, 31 }  ,draw opacity=1 ][fill={rgb, 255:red, 226; green, 5; blue, 31 }  ,fill opacity=1 ][line width=3]      (0, 0) circle [x radius= 6.37, y radius= 6.37]   ;
\draw  [color={rgb, 255:red, 218; green, 140; blue, 10 }  ,draw opacity=1 ] (288,232) .. controls (288,227.58) and (291.58,224) .. (296,224) -- (536,224) .. controls (540.42,224) and (544,227.58) .. (544,232) -- (544,256) .. controls (544,260.42) and (540.42,264) .. (536,264) -- (296,264) .. controls (291.58,264) and (288,260.42) .. (288,256) -- cycle ;
\draw  [color={rgb, 255:red, 218; green, 140; blue, 10 }  ,draw opacity=1 ] (296,160) .. controls (296,155.58) and (299.58,152) .. (304,152) -- (400,152) .. controls (404.42,152) and (408,155.58) .. (408,160) -- (408,184) .. controls (408,188.42) and (404.42,192) .. (400,192) -- (304,192) .. controls (299.58,192) and (296,188.42) .. (296,184) -- cycle ;
\draw  [color={rgb, 255:red, 218; green, 140; blue, 10 }  ,draw opacity=1 ] (328,88) .. controls (328,83.58) and (331.58,80) .. (336,80) -- (496,80) .. controls (500.42,80) and (504,83.58) .. (504,88) -- (504,112) .. controls (504,116.42) and (500.42,120) .. (496,120) -- (336,120) .. controls (331.58,120) and (328,116.42) .. (328,112) -- cycle ;
\draw [color={rgb, 255:red, 43; green, 64; blue, 233 }  ,draw opacity=1 ][line width=1.5]    (435,168) -- (525,168) ;
\draw [shift={(528,168)}, rotate = 180] [color={rgb, 255:red, 43; green, 64; blue, 233 }  ,draw opacity=1 ][line width=1.5]    (11.37,-3.42) .. controls (7.23,-1.45) and (3.44,-0.31) .. (0,0) .. controls (3.44,0.31) and (7.23,1.45) .. (11.37,3.42)   ;
\draw [shift={(432,168)}, rotate = 0] [color={rgb, 255:red, 43; green, 64; blue, 233 }  ,draw opacity=1 ][line width=1.5]    (11.37,-3.42) .. controls (7.23,-1.45) and (3.44,-0.31) .. (0,0) .. controls (3.44,0.31) and (7.23,1.45) .. (11.37,3.42)   ;
\draw    (480,138) -- (480,270) ;
\draw [shift={(480,272)}, rotate = 270] [color={rgb, 255:red, 0; green, 0; blue, 0 }  ][line width=0.75]    (10.93,-3.29) .. controls (6.95,-1.4) and (3.31,-0.3) .. (0,0) .. controls (3.31,0.3) and (6.95,1.4) .. (10.93,3.29)   ;
\draw [shift={(480,136)}, rotate = 90] [color={rgb, 255:red, 0; green, 0; blue, 0 }  ][line width=0.75]    (10.93,-3.29) .. controls (6.95,-1.4) and (3.31,-0.3) .. (0,0) .. controls (3.31,0.3) and (6.95,1.4) .. (10.93,3.29)   ;
\draw    (456,138) -- (456,198) ;
\draw [shift={(456,200)}, rotate = 270] [color={rgb, 255:red, 0; green, 0; blue, 0 }  ][line width=0.75]    (10.93,-3.29) .. controls (6.95,-1.4) and (3.31,-0.3) .. (0,0) .. controls (3.31,0.3) and (6.95,1.4) .. (10.93,3.29)   ;
\draw [shift={(456,136)}, rotate = 90] [color={rgb, 255:red, 0; green, 0; blue, 0 }  ][line width=0.75]    (10.93,-3.29) .. controls (6.95,-1.4) and (3.31,-0.3) .. (0,0) .. controls (3.31,0.3) and (6.95,1.4) .. (10.93,3.29)   ;
\draw    (508,138) -- (508,198) ;
\draw [shift={(508,200)}, rotate = 270] [color={rgb, 255:red, 0; green, 0; blue, 0 }  ][line width=0.75]    (10.93,-3.29) .. controls (6.95,-1.4) and (3.31,-0.3) .. (0,0) .. controls (3.31,0.3) and (6.95,1.4) .. (10.93,3.29)   ;
\draw [shift={(508,136)}, rotate = 90] [color={rgb, 255:red, 0; green, 0; blue, 0 }  ][line width=0.75]    (10.93,-3.29) .. controls (6.95,-1.4) and (3.31,-0.3) .. (0,0) .. controls (3.31,0.3) and (6.95,1.4) .. (10.93,3.29)   ;
\draw [line width=1.5]    (456,168) ;
\draw [shift={(456,168)}, rotate = 0] [color={rgb, 255:red, 0; green, 0; blue, 0 }  ][fill={rgb, 255:red, 0; green, 0; blue, 0 }  ][line width=1.5]      (0, 0) circle [x radius= 4.36, y radius= 4.36]   ;
\draw [line width=1.5]    (508,168) ;
\draw [shift={(508,168)}, rotate = 0] [color={rgb, 255:red, 0; green, 0; blue, 0 }  ][fill={rgb, 255:red, 0; green, 0; blue, 0 }  ][line width=1.5]      (0, 0) circle [x radius= 4.36, y radius= 4.36]   ;
\draw [color={rgb, 255:red, 226; green, 5; blue, 31 }  ,draw opacity=1 ][line width=3]    (480,168) ;
\draw [shift={(480,168)}, rotate = 0] [color={rgb, 255:red, 226; green, 5; blue, 31 }  ,draw opacity=1 ][fill={rgb, 255:red, 226; green, 5; blue, 31 }  ,fill opacity=1 ][line width=3]      (0, 0) circle [x radius= 6.37, y radius= 6.37]   ;
\draw  [color={rgb, 255:red, 218; green, 140; blue, 10 }  ,draw opacity=1 ] (424,160) .. controls (424,155.58) and (427.58,152) .. (432,152) -- (528,152) .. controls (532.42,152) and (536,155.58) .. (536,160) -- (536,184) .. controls (536,188.42) and (532.42,192) .. (528,192) -- (432,192) .. controls (427.58,192) and (424,188.42) .. (424,184) -- cycle ;
\draw [color={rgb, 255:red, 218; green, 140; blue, 10 }  ,draw opacity=1 ] [dash pattern={on 4.5pt off 4.5pt}]  (240,224) .. controls (258.91,204.84) and (275.41,205.11) .. (293.97,222.09) ;
\draw [shift={(296,224)}, rotate = 224.18] [fill={rgb, 255:red, 218; green, 140; blue, 10 }  ,fill opacity=1 ][line width=0.08]  [draw opacity=0] (8.93,-4.29) -- (0,0) -- (8.93,4.29) -- cycle    ;
\draw [color={rgb, 255:red, 218; green, 140; blue, 10 }  ,draw opacity=1 ] [dash pattern={on 4.5pt off 4.5pt}]  (104,152) .. controls (123.11,132.65) and (265.71,135.01) .. (301.46,150.76) ;
\draw [shift={(304,152)}, rotate = 208.6] [fill={rgb, 255:red, 218; green, 140; blue, 10 }  ,fill opacity=1 ][line width=0.08]  [draw opacity=0] (8.93,-4.29) -- (0,0) -- (8.93,4.29) -- cycle    ;
\draw [color={rgb, 255:red, 218; green, 140; blue, 10 }  ,draw opacity=1 ] [dash pattern={on 4.5pt off 4.5pt}]  (232,152) .. controls (251.11,132.65) and (393.71,135.01) .. (429.46,150.76) ;
\draw [shift={(432,152)}, rotate = 208.6] [fill={rgb, 255:red, 218; green, 140; blue, 10 }  ,fill opacity=1 ][line width=0.08]  [draw opacity=0] (8.93,-4.29) -- (0,0) -- (8.93,4.29) -- cycle    ;
\draw [color={rgb, 255:red, 218; green, 140; blue, 10 }  ,draw opacity=1 ] [dash pattern={on 4.5pt off 4.5pt}]  (200,80) .. controls (219.11,60.65) and (300.87,63.01) .. (333.58,78.76) ;
\draw [shift={(336,80)}, rotate = 208.6] [fill={rgb, 255:red, 218; green, 140; blue, 10 }  ,fill opacity=1 ][line width=0.08]  [draw opacity=0] (8.93,-4.29) -- (0,0) -- (8.93,4.29) -- cycle    ;

\draw (9,242) node [anchor=north west][inner sep=0.75pt]   [align=left] {0};
\draw (41,242) node [anchor=north west][inner sep=0.75pt]   [align=left] {0};
\draw (101,242) node [anchor=north west][inner sep=0.75pt]   [align=left] {1};
\draw (201,242) node [anchor=north west][inner sep=0.75pt]   [align=left] {1};
\draw (169,242) node [anchor=north west][inner sep=0.75pt]   [align=left] {0};
\draw (17,173) node [anchor=north west][inner sep=0.75pt]   [align=left] {1};
\draw (41,173) node [anchor=north west][inner sep=0.75pt]   [align=left] {0};
\draw (69,173) node [anchor=north west][inner sep=0.75pt]   [align=left] {1};
\draw (65,98) node [anchor=north west][inner sep=0.75pt]   [align=left] {0};
\draw (101,98) node [anchor=north west][inner sep=0.75pt]   [align=left] {0};
\draw (145,98) node [anchor=north west][inner sep=0.75pt]   [align=left] {0};
\draw (145,173) node [anchor=north west][inner sep=0.75pt]   [align=left] {1};
\draw (169,173) node [anchor=north west][inner sep=0.75pt]   [align=left] {1};
\draw (197,173) node [anchor=north west][inner sep=0.75pt]   [align=left] {0};
\draw (305,242) node [anchor=north west][inner sep=0.75pt]   [align=left] {0};
\draw (337,242) node [anchor=north west][inner sep=0.75pt]   [align=left] {2};
\draw (397,242) node [anchor=north west][inner sep=0.75pt]   [align=left] {1};
\draw (497,242) node [anchor=north west][inner sep=0.75pt]   [align=left] {0};
\draw (465,242) node [anchor=north west][inner sep=0.75pt]   [align=left] {2};
\draw (313,173) node [anchor=north west][inner sep=0.75pt]   [align=left] {2};
\draw (337,173) node [anchor=north west][inner sep=0.75pt]   [align=left] {0};
\draw (365,173) node [anchor=north west][inner sep=0.75pt]   [align=left] {0};
\draw (361,98) node [anchor=north west][inner sep=0.75pt]   [align=left] {0};
\draw (397,98) node [anchor=north west][inner sep=0.75pt]   [align=left] {1};
\draw (441,98) node [anchor=north west][inner sep=0.75pt]   [align=left] {0};
\draw (441,173) node [anchor=north west][inner sep=0.75pt]   [align=left] {1};
\draw (465,173) node [anchor=north west][inner sep=0.75pt]   [align=left] {2};
\draw (493,173) node [anchor=north west][inner sep=0.75pt]   [align=left] {1};
\draw (256,42) node [anchor=north west][inner sep=0.75pt]  [color={rgb, 255:red, 218; green, 140; blue, 10 }  ,opacity=1 ] [align=left] {$\displaystyle \phi _{\Lambda }$};
\draw (217,114) node [anchor=north west][inner sep=0.75pt]  [color={rgb, 255:red, 218; green, 140; blue, 10 }  ,opacity=1 ] [align=left] {$\displaystyle \phi _{\Lambda }$};
\draw (288,114) node [anchor=north west][inner sep=0.75pt]  [color={rgb, 255:red, 218; green, 140; blue, 10 }  ,opacity=1 ] [align=left] {$\displaystyle \phi _{\Lambda }$};
\draw (257,186) node [anchor=north west][inner sep=0.75pt]  [color={rgb, 255:red, 218; green, 140; blue, 10 }  ,opacity=1 ] [align=left] {$\displaystyle \phi _{\Lambda} $};
\end{tikzpicture}
    \caption{The map $\phi: \{0,1\}^{\mathbf{F}_2}\to \{0,1,2\}^{\mathbf{F}_2}$. Each rectangle represents a coset of $\Lambda=\langle a\rangle$ and the red vertices represent the elements of the set $S$.}
    \label{fig:cof}
\end{figure}

We first show that $\phi$ is cofinitely equivariant.
Given $y_0,y_1\in Y^{\Gamma}$ with $y_0\sim_B y_1$, we have the equality $(s^{-1}.y_0)|_{\Lambda}= (s^{-1}.y_1)|_{\Lambda}$ for all but finitely many $s\in S$. For each $s\in S$ for which this equality holds, we have that $\phi (y_0)_\gamma=\phi (y_1)_\gamma$ for all $\gamma \in s\Lambda$.
For each of the finitely many $s\in S$ for which the equality does not hold, we still have the almost equality $(s^{-1}.y_0)|_{\Lambda}\sim_B (s^{-1}.y_1)|_{\Lambda}$. 
Thus, since $\phi _{\Lambda}$ is bi-cofinitely equivariant we see that $\phi (y_0)_\gamma=\phi (y_1)_\gamma $ for all but finitely many $\gamma \in s\Lambda$, and therefore $\phi (y_0)\sim_C \phi (y_1)$.

To show that $\phi$ is cofinitely equivariant, it remains to show that for $\delta \in \Gamma$ and $y\in Y^{\Gamma}$ we have $\phi (\delta .y)\sim_C \delta .\phi (y)$. For each $s\in S$ and $\gamma \in s\Lambda$ we have 
\begin{align}
\nonumber \phi (\delta .y)_\gamma &= \phi _{\Lambda} \big((s^{-1}\delta .y )|_{\Lambda} \big)_{s^{-1}\gamma} \\
\label{eqn:cofeq} (\delta .\phi (y) )_\gamma  &= \phi _{\Lambda}\big((\sigma (\delta ^{-1}\gamma )^{-1}.y)|_{\Lambda}\big)_{\sigma (\delta ^{-1}\gamma )^{-1}\delta ^{-1}\gamma}  \\
\nonumber &= \phi _{\Lambda}\big(\rho (\delta ^{-1},s\Lambda ).((s^{-1}\delta .y)|_{\Lambda}) \big)_{\rho (\delta ^{-1}, s\Lambda )s^{-1}\gamma}.
\end{align}
If $s\in S\cap \delta S$ then $\rho (\delta ^{-1}, s\Lambda ) = e_{\Lambda}$, hence we have equality $\phi (\delta .y)_{\gamma} = (\delta .\phi (y) )_{\gamma}$ for all $\gamma \in s\Lambda$. 
If $s\in S\setminus \delta S$, then since $\phi _{\Lambda}$ is bi-cofinitely equivariant and $\rho (\delta ^{-1},s\Lambda )\in \Lambda$ we have 
\[
\phi _{\Lambda}\big(\rho (\delta ^{-1},s\Lambda ).((s^{-1}\delta .y)|_{\Lambda}) \big)\sim_C \rho (\delta ^{-1},s\Lambda ).\phi _{\Lambda}\big((s^{-1}\delta.y)|_{\Lambda} \big),
\]
and hence \eqref{eqn:cofeq} shows that $(\delta .\phi (y) )_{\gamma} =\phi (\delta .y)_{\gamma}$ for all but finitely many $\gamma \in s\Lambda$. 
Since $\delta S\triangle S$ is finite we conclude that $\delta .\phi (y)\sim_{C} \phi (\delta .y)$ and thus $\phi$ is cofinitely equivariant. 
Similarly, $\phi ^{-1}$ is cofinitely equivariant.

It remains to show that $\phi$ is a measure space isomorphism. For each $s\in S$ the map $Y^{\Gamma}\rightarrow Z^{s\Lambda}$, $y\mapsto \phi (y)|_{s\Lambda}$ only depends on $y|_{s\Lambda}$, and hence descends to a map $\phi _s:Y^{s\Lambda}\rightarrow Z^{s\Lambda}$. 
Therefore, under the natural identifications $Y^\Gamma = \prod _S Y^{s\Lambda}$ and $Z^{\Gamma} = \prod _S Z^{s\Lambda}$, we may identify $\phi$ with $\prod _S \phi _s$, and hence it is enough to show that each $\phi _s$ is a measure space isomorphism. 
Each of the maps $\phi _s$ is a composition $\phi _s = s_Z\circ \phi _{\Lambda}\circ s_Y^{-1}$, where $s_Y : Y^{\Lambda}\rightarrow Y^{s\Lambda}$ is the map $s_Y(y)_{s\lambda}=y_{\lambda}$, and similarly for $s_Z$. Therefore, $\phi _s$ is a measure space isomorphism, since each of the maps $s_Z$, $\phi _\Lambda$, and $s_Y$ are. 
\end{proof}

\begin{cor}\label{cofFreeProducts}
Assume that $\Lambda$ is a free factor of the countable group $\Gamma$ (i.e. $\Gamma = \Lambda \ast H$ for some subgroup $H$ of $\Gamma$). 
Let $B$ and $C$ be countable groups acting in a p.m.p.\ manner on $(Y,\nu)$ and $(Z,\eta)$ respectively.
With respect to these actions, if there exists a bi-cofinitely $\Lambda$-equivariant measure space isomorphism from $(Y^\Lambda,\nu^\Lambda)$ to $(Z^\Lambda,\eta^\Lambda)$, then there exists a bi-cofinitely $\Gamma$-equivariant measure space isomorphism  from $(Y^\Gamma,\nu^\Gamma)$ to $(Z^\Gamma,\eta^\Gamma)$.
\end{cor}
\begin{proof}

Assume $\Gamma= \Lambda \ast H$. We must find a set $S$ as in \textit{(i)} of Theorem \ref{cofFinCosetReprError}.
Let $S$ consist of the identity element of $\Gamma$, along with all non-identity group elements whose normal form, with respect to the free splitting $\Gamma = \Lambda\ast H$, ends with an element of $H \setminus\{e\}$ (i.e., the rightmost term of this normal form is in $H\setminus \{ e \}$). 
This forms a collection of left coset representatives for $\Lambda$ in $\Gamma$ with 
$hS=S$ for all $h\in H$ and $\lambda S\triangle S = \{ \lambda , e \}$ for all $\lambda \in \Lambda\setminus\{ e \}$. 
Since $\Lambda$ and $H$ generate $\Gamma$, it follows that $|\gamma S\triangle S|<\infty$ for all $\gamma \in \Gamma$.
\end{proof}

\begin{cor}\label{cofFinIndex}
Assume that $\Lambda$ is a finite index subgroup of the countable group $\Gamma$. 
Let $B$ and $C$ be countable groups acting in a p.m.p.\ manner on $(Y,\nu)$ and $(Z,\eta)$ respectively.
With respect to these actions, if there exists a bi-cofinitely $\Lambda$-equivariant measure space isomorphism from $(Y^\Lambda,\nu^\Lambda)$ to $(Z^\Lambda,\eta^\Lambda)$, then there exists a bi-cofinitely $\Gamma$-equivariant measure space isomorphism  from $(Y^\Gamma,\nu^\Gamma)$ to $(Z^\Gamma,\eta^\Gamma)$.
\end{cor}

The first part of the following corollary will be generalized in Theorem \ref{thm:MELampGroups} below. 

\begin{cor}\label{cor:freeFactorAmenableLampGroups}
Let $\Gamma$ be a countable group that has an infinite amenable free factor, and let $A_0$ and $A_1$ be nontrivial (possibly finite) amenable groups.
Then $A_0\wr\Gamma$ and $A_1\wr\Gamma$ are orbit equivalent.
In fact, the canonical wreath product actions of $A_0\wr\Gamma$ and $A_1\wr\Gamma$ are orbit equivalent.
\end{cor}

\begin{proof}
This follows immediately from Theorem \ref{cofAmenable} and Corollary \ref{cofFreeProducts}.
\end{proof}

\begin{remark}
By definition, a discrete group $\Gamma$ has more than one  end if and only if it contains an infinite, co-infinite subset $S$ satisfying $|\gamma S\triangle S|<\infty$ for all $\gamma \in \Gamma$. 
It is unclear to us how a subgroup $\Lambda$ must sit inside of $\Gamma$ in order to find such an $S$ that is furthermore a set of left coset representives for $\Lambda$ in $\Gamma$. 
For example, such an $S$ can be found whenever $\Lambda$ is a free factor of a finite index subgroup of $\Gamma$, but we do not know if the converse holds.  
\end{remark}

\section{Groupoid Preliminaries}\label{section:groupoids}

A groupoid is a small category in which every morphism is invertible. 
Given a groupoid $\mathcal{G}$, we will abuse notation and also write $\mathcal{G}$ for its set of morphisms.
We write $\mathrm{Ob}(\mathcal{G})$ for the set of objects of $\mathcal{G}$, and for $x\in\mathrm{Ob}(\mathcal{G})$ we write $e_x$ for the \textbf{unit} (i.e., identity morphism) at $x$. 
We call the set $\G ^0$, of all units of $\mathcal{G}$, the \textbf{unit space} of $\mathcal{G}$. 
We write $s:\G\rightarrow \mathrm{Ob}(\mathcal{G})$ and $r:\G \rightarrow \mathrm{Ob}(G)$ for the source and range maps of $\mathcal{G}$, determined by $e_{s(g)}=g^{-1}g$ and $e_{r(g)}=gg^{-1}$ for $g\in\mathcal{G}$.
We will write $s_{\mathcal{G}}$ for $s$ and $r_\mathcal{G}$ for $r$ if there is a need to notationally distinguish the source and range maps of $\mathcal{G}$ with those of another groupoid.
We say that a groupoid $\G$ is {\bf aperiodic} if $r^{-1}(x)$ is infinite for every $x\in \G^0$. 

Given a groupoid $\G$, the image of $\G$ in $\mathrm{Ob}(\mathcal{G})\times \mathrm{Ob}(\mathcal{G})$ under the map $g\mapsto (r(g),s(g))$ is an equivalence relation on $\G ^0$ which we denote by $\mathcal{R}_{\mathcal{G}}$.
A groupoid $\G$ is called {\bf principal} if the map $g\mapsto (r(g),s(g))$ is injective.
An equivalence relation $\mathcal{R}$ on a set $X$ is naturally a principal groupoid, with $\mathrm{Ob}(\mathcal{R})=X$, unit space the diagonal $\R^0=\{(x,x)\in X\times X: x\in X\}$, source and range maps the right and left projections respectively, and inverse and composition operations given by $(y,x)^{-1}\coloneqq (x,y)$ and $(z,y)(y,x)\coloneqq (z,x)$ for all mutually $\mathcal{R}$-related points $x,y,z\in X$. 
Each principal groupoid is naturally isomorphic as a groupoid to an equivalence relation via the map $g\mapsto (r(g),s(g))$. 
Given this correspondence, we will freely and frequently identify principal groupoids with their associated equivalence relation.

\begin{convention}
It will frequently be convenient to identify each object $x$ of a groupoid $\mathcal{G}$ with its corresponding unit $e_x$, in which case we will not distinguish between the sets $\mathrm{Ob}(\mathcal{G})$ and $\mathcal{G}^0$, and we will view $s$ and $r$ as taking values in $\mathcal{G}^0$. 
\end{convention}

A \textbf{discrete Borel groupoid} is a groupoid where both $\G$ and $\G^0$ are standard Borel spaces, the source and range maps $s$ and $r$ are both Borel and countable-to-one, and the composition and inverse operations are Borel. We call a discrete Borel groupoid that is an equivalence relation a \textbf{countable Borel equivalence relation}.

A \textbf{discrete p.m.p.\ groupoid} is a pair $(\G,\mu_{\G^0})$ where $\G$ is a discrete Borel groupoid and $\mu_{\G^0}$ is a Borel probability measure on $\G^0$ satisfying $\int_{\G^0} c_x^r d\mu_{\G^0}=\int_{\G^0} c_x^s d\mu_{\G^0}$ where $c_x^r$ and $c_x^s$ refer to the counting measure on $r^{-1}(x)$ and $s^{-1}(x)$ respectively. Set $\mu^1_\G:=\int_{\G^0} c_x^r d\mu_{\G^0}=\int_{\G^0} c_x^s d\mu_{\G^0}$ to be this sigma-finite measure on $\G$.

Given two subsets $D_0$ and $D_1$ of $\mathcal{G}$ we define 
$CD\coloneqq \{ gh : g\in C,\ h\in D,\  s(g)=r(h)\}$, and for $g\in \mathcal{G}$ we define $gD\coloneqq \{g\}D$ and $Dg\coloneqq D\{g\}$.

Given a positive measure subset $D$ of $\mathcal{G}^0$, we define the \textbf{reduction} (or \textbf{restriction}) of $\mathcal{G}$ to $D$ to be the discrete p.m.p.\ groupoid $\G_D\coloneqq D\mathcal{G}D$, where the measure $\mu _{\mathcal{G}_D^0}$ is the normalized restriction of $\mu _{\mathcal{G}^0}$ to $D$; for notational clarity we will sometimes write $\mathcal{G}|_D$ for $\mathcal{G}_D$. 
A \textbf{(Borel) partial bisection} of $\mathcal{G}$ is a (Borel) subset of $\G$ on which the maps $r$ and $s$ are both injective. 
For $\sigma$ a partial bisection, $x\in r(\sigma)$, and $y\in s(\sigma)$, we abuse notation and write $x\sigma$ and $\sigma y$ for the unique element of $x\sigma$ and $\sigma y$ respectively.

A subset $C$ of $\G^0$ is said to be $\G$-\textbf{invariant} if it is $\mathcal{R}_{\mathcal{G}}$-invariant, or equivalently $r(\G C)=C$.  
A discrete p.m.p.\ groupoid $\mathcal{G}$ is called \textbf{ergodic} if every $\G$-invariant Borel subset $C$ of $\G^0$ is either $\mu _{\mathcal{G}^0}$-null or $\mu _{\mathcal{G}^0}$-conull.

Let $\Gamma \acts (X,\mu)$ be a p.m.p.\ action of a countable group on a standard probability space $(X,\mu )$. 
The associated \textbf{action groupoid} is the discrete p.m.p.\ groupoid $\G = \Gamma\ltimes X$ with morphism set $\Gamma \times X$, unit space $\G^0=\{e_\Gamma\}\times X$ equipped with the measure $\delta _e \times \mu$, source and range maps given by $s(\gamma,x)= (e,x)$, and $r(\gamma,x)=(e,\gamma . x)$, and composition law given by $(\gamma _1,x)(\gamma _0,y)\coloneqq(\gamma _1 \gamma _0 ,y)$ whenever $x=\gamma _0 . y$. 

Let $X$ be a standard Borel space. A {\bf fibered Borel space} over $X$ is a standard Borel space $W$ together with a Borel map $q:W\rightarrow X$. 
We often suppress the fibering map from the notation, and we write $W_x$ for the fiber $q^{-1}(x)$ over $x\in X$. 
Suppose that $\mathcal{G}$ is a discrete Borel groupoid and $W\rightarrow \mathcal{G}^0$ is a fibered Borel space over $\mathcal{G}^0$. A (left) {\bf Borel action} of $\mathcal{G}$ on $W$ is an assignment of a Borel isomorphism $\alpha _g:W_{s(g)}\rightarrow W_{r(g)}$ to each $g\in \mathcal{G}$ satisfying $\alpha _{g^{-1}}=\alpha _g^{-1}$ and $\alpha _{g_1g_0}=\alpha _{g_1}\circ \alpha _{g_0}$ whenever $s(g_1)=r(g_0)$, and such that the map $\{ (g,h)\in\mathcal{G}\times \mathcal{H} : h\in \mathcal{H}_{s(g)}\} \rightarrow \mathcal{H}$, $(g,h)\mapsto \alpha _g(h)$, is Borel.

\subsection{Semidirect products of groupoids}
Let $\G$ be a discrete p.m.p. groupoid and let $q: \mathcal{H}\to \mathcal{G}^0$ be a Borel bundle of discrete p.m.p.\ groupoids over $\mathcal{G}^0$, i.e., $q: \mathcal{H}\rightarrow \mathcal{G}^0$ is a fibered Borel space over $\mathcal{G}^0$ and for each $x\in \mathcal{G}^0$ the fiber $\mathcal{H}_x$ is equipped with the structure of a discrete p.m.p.\ groupoid. 
Note that $\mathcal{H}$ itself is naturally a discrete p.m.p.\ groupoid when $\mathcal{H}^0=\bigsqcup _{x\in \mathcal{G}^0}\mathcal{H}_x^0$ is equipped with the measure $\mathcal{\mu}_{\mathcal{H}^0}\coloneqq \int _{\mathcal{G}^0}\mu _{\mathcal{H}_x^0}\, d\mu _{\mathcal{G}^0}(x)$.
Suppose that $\alpha :\G \acts \h$ is a Borel action of $\G$ on $\mathcal{H}$ by measure preserving groupoid isomorphisms, i.e., for each $g\in \mathcal{G}$ the map $\alpha _g : \mathcal{H}_{s(g)}\to\mathcal{H}_{r(g)}$ is a measure preserving groupoid isomorphism.
The associated \textbf{semidirect product groupoid}, denoted $\G\ltimes\mathcal{H}$, is the 
discrete p.m.p.\ groupoid defined as follows:
\begin{itemize}[leftmargin=*]
\item the morphism set consists of all pairs $(g,h)$ with $g\in\mathcal{G}$ and $h\in\mathcal{H}_{s(g)}$, equipped with the standard Borel structure it inherits as a Borel subset of $\mathcal{G}\times\mathcal{H}$.
\item the set of objects is $\mathrm{Ob}(\mathcal{G}\ltimes\mathcal{H} )\coloneqq \mathcal{H}^0$, and for each $x\in\mathcal{G}^0$ and $y\in \mathcal{H}_x^0$, the unit at $y$ is the morphism $e_y\coloneqq (x,y) \in (\mathcal{G}\ltimes \mathcal{H})^0$;
\item the probability measure $\mu _{(\mathcal{G}\ltimes\mathcal{H} )^0}$ on the unit space is the pushforward of $\mu _{\mathcal{H}^0}$ under $y\mapsto e_y$.
\item the source and range maps are given by $s(g,h) \coloneqq s_{\mathcal{H}}(h)$ and $r(g,h)\coloneqq \alpha _g(r_{\mathcal{H}}(h))$;
\item the composition law in $\G\ltimes\mathcal{H}$ is given by $(g_1,h_1)(g_0,h_0) \coloneqq (g_1g_0, \alpha _{g_0}^{-1}(h_1)h_0)$ whenever $s(g_1, h_1)=r(g_0, h_0)$.
\end{itemize}
The map $\iota : \mathcal{H}\to \mathcal{G}\ltimes\mathcal{H}$, defined fiberwise by $\iota (h)\coloneqq (x,h)$ for $h\in \mathcal{H}_x$, is a measure preserving groupoid embedding.
Then $\iota (\mathcal{H})$ complements the subgroupoid $\mathcal{G}\ltimes \mathcal{H}^0$, 
in the sense that each $\gamma \in \mathcal{G}\ltimes\mathcal{H}$ can be expressed uniquely as a product $\gamma = \delta h$, with $\delta \in \mathcal{G}\ltimes\mathcal{H}^0$ and $h\in \iota (\mathcal{H})$: namely, if $\gamma = (g,h_0)$ then we have $\gamma = \delta h$ where $\delta = (g,r(h_0)) \in \mathcal{G}\ltimes\mathcal{H}^0$ and $h=\iota (h_0)\in\iota (\mathcal{H})$. 
Using the embedding $\iota$, in what follows we will identify $\mathcal{H}$ with its image, $\iota (\mathcal{H})$, in $\mathcal{G}\ltimes \mathcal{H}$.

\begin{remark}
    The action groupoid construction is a special case of the semidirect product of groupoids. 
    Indeed, let $\Gamma\acts(X,\mu)$ be a p.m.p\ action of a discrete countable group $\Gamma$.
    We may view $X$ as a discrete p.m.p.\ groupoid consisting solely of units, $\mathrm{Ob}(X)=X=X^0$, and equipped with the measure $\mu$. 
    The action of $\Gamma$ on $X$ is then by measure preserving groupoid automorphisms, and the associated semidirect product groupoid $\Gamma\ltimes X$ is precisely the action groupoid $\Gamma\ltimes (X,\mu )$.
\end{remark}

\subsection{Direct sums of groupoids}\label{subsection:direct_sum_groupoid}
Given a countable set $I$ and discrete p.m.p.\ groupoids $\mathcal{K}_i$, $i\in I$, their {\bf direct sum}, denoted $\bigoplus _I\mathcal{K}_i$ is the discrete p.m.p.\ groupoid with underlying set
\[
\bigoplus _I \mathcal{K}_i = \{ k\in \prod _I\mathcal{K}_i : k_i \in \mathcal{K}_i^0\text{ for all but finitely many }i\in I \} ,
\]
with unit space $(\bigoplus _I \mathcal{K}_i)^0 = \prod _I \mathcal{K}^0_i$, and with source and range maps determined by $s(k)_i = s_{\mathcal{K}_i}(k_i)$ and $r(k)_i=r_{\mathcal{K}_i}(k_i)$ for $i\in I$, and multiplication performed coordinate-wise: for $(kl)_i = k_il_i$ for $k,l \in \bigoplus _I\mathcal{K}_i$ with $s(k)=r(l)$. The measure on $(\bigoplus _I \mathcal{K}_i)^0$ is product measure. When $\mathcal{K}_i=\mathcal{K}$ for all $i\in I$ then we denote the direct sum by $\bigoplus _I \mathcal{K}$. 

If $\mathcal{R}_i$ is a p.m.p.\ countable Borel equivalence relation on $(X _i ,\mu _i )$ for each $i\in I$, then we will naturally identify the groupoid $\bigoplus _I \mathcal{R}_i$ with the p.m.p.\ countable Borel equivalence relation on $(\prod _I X_i , \prod _I \mu _i )$ consisting of all pairs $(x,y)\in (\prod _I X_i)^2$ satisfying $(x_i,y_i ) \in \mathcal{R}_i$ for all $i\in I$ and $x_i=y_i$ for all but finitely many $i\in I$.

\subsection{Wreath products of groupoids} Wreath products of groupoids are discussed in \cite{houghton1975wreath}; here we describe the construction in the discrete p.m.p.\ setting.
Let $\mathcal{K}$ and $\mathcal{G}$ be discrete p.m.p.\ groupoids and suppose $\mathcal{G}\curvearrowright W$ is a Borel action of $\mathcal{G}$ on a fibered Borel space $W\to \mathcal{G}^0$ with each fiber countable. 
Define $q: \boldsymbol{\bigoplus}_W \K\to \mathcal{G}^0$ to be the Borel bundle of discrete p.m.p.\ groupoids over $\mathcal{G}^0$ whose fiber over $x\in \mathcal{G}^0$ is the direct sum $(\boldsymbol{\bigoplus}_{W}\mathcal{K})_x\coloneqq \bigoplus_{W_x}\K$. 
(We use the boldface symbol $\boldsymbol{\bigoplus}$ to distinguish a bundle of direct sum groupoids from an ordinary direct sum groupoid.) 
Here, $\boldsymbol{\bigoplus}_W \K$ is equipped with the sigma-algebra generated by the map $q$ together with all maps of the form $\boldsymbol{\bigoplus}_W \K\rightarrow \K$, $k\mapsto k_{\sigma (q(k))}$, where $\sigma$ varies over all Borel sections of $W\to \mathcal{G}^0$;
this makes $\boldsymbol{\bigoplus}_W \K$ into a standard Borel space (as can be verified using the Lusin-Novikov Uniformization Theorem).  
Thus, $\boldsymbol{\bigoplus}_W \K$ is a discrete p.m.p.\ groupoid whose unit space $(\boldsymbol{\bigoplus}_W \K )^0 = \bigsqcup _{x\in\mathcal{G}^0}\mathcal{K}^{W_x}$ is equipped with the probability measure $\mu _{(\boldsymbol{\bigoplus}_W \K )^0} = \int _{\mathcal{G}^0}\mu _{\mathcal{K}^0}^{W_x}\, d\mu _{\mathcal{G}^0}(x)$.

The groupoid $\G$ acts by measure preserving groupoid automorphisms on $\boldsymbol{\bigoplus}_W \K$ via the shift action given by $(g.k )_w \coloneqq k_{g^{-1} . w}$ for $g\in \G$, $k\in \bigoplus _{W_{s(g)}}\K$, and $w\in W_{r(g)}$. 
The associated {\bf (restricted) wreath product groupoid}, denoted $\K\wr_W\G$, is the semidirect product $\K\wr_W\G\coloneqq \G\ltimes (\boldsymbol{\bigoplus}_W \K)$.
As indicated after the definition of semidirect product groupoids, we naturally view the direct sum bundle $\boldsymbol{\bigoplus} _W \mathcal{K}$ as a subgroupoid of $\K\wr _W\G$.  
We write $\K\wr\G$ in the case where $W=\mathcal{G}$ fibers over $\mathcal{G}^0$ via the range map $r$, and the action of $\mathcal{G}$ is by left translation.

\begin{remark}\label{rem:action_wreath} If $\G=\Gamma\ltimes(X,\mu)$ and $\mathcal{B}=B\ltimes(Y,\nu)$ are action groupoids associated to p.m.p.\ actions of countable groups $\Gamma$ and $B$ respectively, then the wreath product groupoid $\mathcal{B}\wr \G$ is isomorphic to the action groupoid $(B\wr \Gamma )\ltimes (X\times Y^{\Gamma} , \mu \times \nu ^{\Gamma})$ associated to the product action $B\wr\Gamma \curvearrowright (X\times Y^{\Gamma} , \mu \times \nu ^{\Gamma})$, given on the $Y^{\Gamma}$-coordinate by the wreath product action of $B\wr\Gamma$, and on the $X$-coordinate by the action factoring through the quotient map $B\wr \Gamma \rightarrow \Gamma$ to the initial action of $\Gamma$.
This was first observed and proven in \cite[Proposition 7.1]{DelabieKoivistoLeMaitreTessera2022}.
\end{remark}

\subsection{Free factors of groupoids} Let $\mathcal{G}$ be a groupoid and let $(\mathcal{H}_i)_{i\in I}$ be a family of subgroupoids of $\mathcal{G}$ (we do not require the subgroupoids to contain all of $\mathcal{G}^0$). 
We say that the family $(\mathcal{H}_i)_{i\in I}$ is \textbf{freely independent} if there does not exist a finite sequence $i_0, \dots, i_{n-1}\in I$ of $n\geq 2$ indices with $i_j\neq i_{j+1}$, along with non-unit elements $h_0\in \mathcal{H}_{i_0} -\mathcal{G}^0,  \dots, h_{n-1} \in \mathcal{H}_{i_{n-1}}-\mathcal{G}^0$ such that $h_0\cdots h_{n-1} \in \mathcal{G}^0$.
We say that the subgroupoids $\mathcal{H}_0,\dots , \mathcal{H}_{n-1}$ of $\mathcal{G}$ are \textbf{mutually freely independent} if the family $(\mathcal{H}_i)_{0\leq i<n}$ is freely independent; when $n=2$ we drop the word ``mutual'' and simply say that $\mathcal{H}_0$ and $\mathcal{H}_1$ are freely independent. 
If $\mathcal{H}_0$ and $\mathcal{H}_1$ are subgroupoids of $\mathcal{G}$ that are freely independent and together generate $\mathcal{G}$, then we write $\mathcal{G}=\mathcal{H}_0\ast \mathcal{H}_1$, and we call this data a \textbf{(free) splitting} of $\mathcal{G}$. 
We call a subgroupoid $\mathcal{H}$ of $\mathcal{G}$ a {\bf free factor} of $\mathcal{G}$ if $\mathcal{G}$ splits as 
$\mathcal{G}=\mathcal{H}\ast \mathcal{K}$ for some subgroupoid $\mathcal{K}$ of $\mathcal{G}$. 
The reader can find more about free factors of groupoids and equivalence relations in the works of Alvarez \cite{AlvarezThesis}, Alvarez-Gaboriau \cite{AlvarezGaboriau2012}, Carderi \cite{CarderiThesis}, and Gaboriau \cite{Ga00}.

\begin{prop}\label{prop:freeFactorsHaveNiceCosetRep}
Let $\h$ be a free factor of a discrete Borel groupoid $\G$ with $\mathcal{H}^0=\mathcal{G}^0$. Then there exists a Borel set $\Sigma \subseteq \mathcal{G}$ of left coset representatives for $\mathcal{H}$ in $\mathcal{G}$ satisfying $|g(s(g)\Sigma )\triangle (r(g)\Sigma )|<\infty$ for all $g\in \mathcal{G}$.
\end{prop}

\begin{proof}
Let $\mathcal{K}\leq \G$ be such that $\h\ast \mathcal{K} =\G$. Every non-unit element $g\in\G -\G ^0$ has a unique normal form as a product where the terms of the product alternate between non-unit elements in $\h- \G ^0$ and non-unit elements in $\mathcal{K}- \G ^0$. Let $\Sigma\leq \G$ consist of all units together with all non-units whose normal form has rightmost term in $\K$. 
The set $\Sigma$ is invariant under left multiplication by elements of $\mathcal{K}$, and $h(s(h)\Sigma )\triangle (r(h)\Sigma )=\{h,r(h)\}$ for non-unit $h\in\h$. Since $\h\cup \mathcal{K}$ generates $\G$, it follows that $|g(s(g)\Sigma )\triangle (r(g)\Sigma )|<\infty$ for all $g\in \mathcal{G}$.
\end{proof}

\section{Measurable splittings}\label{section:measurable_splittings}
We recall some definitions from (and closely related to) \cite{AlvarezGaboriau2012}.
Let $\mathcal{R}$ be a countable Borel equivalence relation on a standard Borel space $X$.
A splitting $\R =\mathcal{R}_0\ast\mathcal{R}_1$ of $\mathcal{R}$ is called \textbf{inessential} if there exists a Borel partition of $X$ into $\mathcal{R}$-invariant sets $X_0$ and $X_1$ along with, for each $i\in \{ 0 , 1\}$, a Borel complete section $U_i$ of $\mathcal{R}|_{X_i}$ such that $\mathcal{R}|_{U_i}=\mathcal{R}_i|_{U_i}$.
If $\mu$ is a Borel probability measure on $X$, then we say that a splitting of $\mathcal{R}$ is $\mu$-\textbf{inessential} if there is an $\mathcal{R}$-invariant $\mu$-conull set on which it is inessential, and we say that the splitting is $\mu$\textbf{-essential} otherwise.

Suppose that $U\subseteq X$ is a Borel complete section for $\R$.
A \textbf{sliding} of $\mathcal{R}$ to $U$ is a splitting of $\mathcal{R}$ of the form $\mathcal{R}=\mathcal{R}|_U\ast\mathcal{T}$, where $\mathcal{T}$ is a smooth (hence treeable) subrelation of $\mathcal{R}$ that has $U$ as a transversal, and that has the same domain as $\mathcal{R}$ (i.e., all of $X$).

The following lemma is closely related to \cite[Example 1.3]{AlvarezGaboriau2012} and the proof of \cite[Theorem 2.1]{GaboriaurelativeT}.
We will need it for the proof of Theorem \ref{thm:MELampGroups} below.

\begin{lemma}\label{lem:freeDecomposableImpliesTreeableFreeFactor}
    Let $\R$ be an ergodic p.m.p.\ countable Borel equivalence relation on $(X,\mu )$ that is $\mu$-nowhere amenable and admits a $\mu$-essential splitting. 
    Then, after discarding a null set, 
    \begin{enumerate}
    \item $\R$ has a free factor $\mathcal{T}$ with domain all of $X$ that is both treeable and $\mu$-nowhere amenable.

    \item $\R$ has an aperiodic amenable free factor $\s$ with domain all of $X$.
    \end{enumerate}
\end{lemma}

\begin{proof}
(2) follows directly from (1): given $\mathcal{T}$ from (1), take $\s$ to be the equivalence relation generated by any aperiodic hyperfinite subgraph of a treeing of $\mathcal{T}$ (which always exists; see, e.g., \cite{Ga00} or \cite{KeMi04}). 

Toward proving (1), let $\mathcal{R}=\mathcal{R}_0\ast\mathcal{R}_1$ be a $\mu$-essential splitting of $\mathcal{R}$.
Note that it is enough to show that for some positive measure subset $X_0$ of $X$, the restriction $\mathcal{R}|_{X_0}$ has a free factor that is treeable and $\mu |_{X_0}$-nowhere amenable.

If all classes of both $\mathcal{R}_0$ and $\mathcal{R}_1$ are finite then $\mathcal{R}$ is itself treeable, so we are done.
We may therefore assume that the set $D_0$, where $\mathcal{R}_0$ is aperiodic, has positive measure.
Then by \cite[Theorem 2.63]{AlvarezThesis}, $\mathcal{R}|_{D_0}$ splits as $\mathcal{R}_0|_{D_0}\ast \mathcal{Q}$ for some subrelation $\mathcal{Q}$ of $\mathcal{R}|_{D_0}$. This splitting is $\mu$-essential: since the original splitting of $\mathcal{R}$ is $\mu$-essential, there cannot be a positive measure set on which $\mathcal{R}$ coincides with $\mathcal{R}_0$, and there cannot be a positive measure subset $U$ of $D_0$ on which $\mathcal{R}$ coincides with $\mathcal{Q}$ since $\mathcal{R}_0|_U$ is still (a.e.) aperiodic on $U$.
Thus, after restricting to $D_0$ if necessary, we may assume without loss of generality that $\mathcal{R}_0$ is aperiodic with domain all of $X$. 

The splitting $\mathcal{R}=\mathcal{R}_0\ast\mathcal{R}_1$ being $\mu$-essential implies that, after discarding an $\mathcal{R}$-invariant null set, the set $D_1$, where $\mathcal{R}_1$ is nontrivial, is a complete section for $\mathcal{R}$ and, again by \cite[Theorem 2.63]{AlvarezThesis} (together with aperiodicity of $\mathcal{R}_0$), after restricting to $D_1$ we may additionally assume without loss of generality that  $\mathcal{R}_1$ has domain all of $X$ and every $\mathcal{R}_1$-class is nontrivial. 

Since $\mathcal{R}_0$ is aperiodic, we may find a Borel complete section $Y_0$ for $\mathcal{R}_0$ that meets each ergodic component of $\mathcal{R}_0$ in a set of measure at most $\tfrac{1}{3}$ with respect to the $\mathcal{R}_0$-ergodic measure on that component; 
likewise, since each $\mathcal{R}_1$-class is nontrivial, we may find a Borel complete section $Y_1$ for $\mathcal{R}_1$ that meets each ergodic component of $\mathcal{R}_1$ in a set of measure at most $\tfrac{1}{2}$ with respect to the $\mathcal{R}_1$-ergodic measure on that component. 
Then, after discarding another $\mathcal{R}$-invariant null set, for each $i\in \{ 0,1\}$ we may find a Borel retraction $f_i:X\to Y_i$ whose graph is contained in $\mathcal{R}_i$, with $|f_0^{-1}(x)|\geq 3$ for each $x\in Y_0$ and $|f_1^{-1}(x)|\geq 2$ for each $x\in Y_1$. 
The equivalence relation $\mathcal{T}_i$ generated by $f_i$ therefore gives a sliding $\mathcal{R}_i=\mathcal{R}_i|_{Y_i} \ast \mathcal{T}_i$ of $\mathcal{R}_i$ to $Y_i$.
Observe that each $\mathcal{T}_0$-class has cardinality at least $3$, and each $\mathcal{T}_1$-class has cardinality at least $2$; moreover the relations $\mathcal{R}_0|_{Y_0}$, $\mathcal{R}_1|_{Y_1}$, $\mathcal{T}_0$, and $\mathcal{T}_1$ are mutually freely independent and generate $\mathcal{R}$. 
It follows that the relation $\mathcal{T}\coloneqq  \mathcal{T}_0\ast\mathcal{T}_1$ is is treeable (as witnessed by the treeing coming from $f_0$ and $f_1$), and it is $\mu$-nowhere amenable (using either \cite{AdamsTreeAmenable} or \cite{Ga00}). 
Since $\mathcal{R}$ splits as $\mathcal{R}=(\mathcal{R}_0|_{Y_0} \ast \mathcal{R}_1|_{Y_1}) \ast \mathcal{T}$, the proof is complete.
\end{proof}

\begin{remark}
    The naive group theoretic analogue of Lemma \ref{lem:freeDecomposableImpliesTreeableFreeFactor} is false. 
    Let $\Gamma = \Gamma_0\ast\Gamma_1$, where $\Gamma _i$ are both simple nonamenable groups. 
    Each nontrivial free factor of $\Gamma$ is nonamenable since it is a retract of $\Gamma$ and therefore contains an isomorphic copy of either $\Gamma_0$ or $\Gamma_1$.
    However, see \cite[Example 1.3]{AlvarezGaboriau2012} for an interesting family of examples where the group theoretic analogue virtually holds.
\end{remark}

\begin{prop}\label{prop:freeProductRelationsAdmitEssentialSplittings}
Let $\mathcal{R}$ be a p.m.p.\ countable Borel equivalence relation on $(X,\mu )$.
Suppose that $\mathcal{R}$ splits as $\mathcal{R}=\mathcal{R}_0\ast\mathcal{R}_1$ where for each $i\in \{ 0 , 1 \}$ the subrelation $\mathcal{R}_i$ has domain all of $X$ and all $\mathcal{R}_i$-classes are nontrivial. 
Then this splitting of $\mathcal{R}$ is $\mu$-essential.
\end{prop}

\begin{proof}
It is enough to show that if $D\subseteq X$ is an $\mathcal{R}_0$-invariant set with $\mathcal{R}|_D =\mathcal{R}_0|_D$, then $D$ must be $\mu$-null.
Given such an $D$, let $Y\subseteq X$ be the $\mathcal{R}$-saturation of $D$.
The fact that $\mathcal{R}|_D=\mathcal{R}_0|_D$ means that for each $x\in Y$, the set $D$ meets the $\mathcal{R}$-class $[x]_{\mathcal{R}}$ in exactly one $\mathcal{R}_0$-class which we denote by $C_x$.

Consider the bipartite graph $G$  on the vertex set $V\coloneqq Y\sqcup ((Y/\mathcal{R}_0) \sqcup (Y/\mathcal{R}_1))$, where $Y$ is one piece of the bipartition, and with an edge between $x$ and $[x]_{\mathcal{R}_i}$ for each $x\in Y$ and $i\in \{ 0, 1 \}$. 
The relations $\mathcal{R}_0$ and $\mathcal{R}_1$ being freely independent means that this graph is acyclic.
The fact that $\mathcal{R}_0$ and $\mathcal{R}_1$ generate $\mathcal{R}$ means that the vertices of $V$ contained in a given $\mathcal{R}|_Y$-class constitute a single connected component of $G$.
For each $x\in Y$ there is a unique geodesic path through $G$ from $x$ to the $\mathcal{R}_0$-class $C_x\in Y/\mathcal{R}_0$, and we let $\pi (x)\in D$ be the penultimate vertex along this path.

Define a mass transport where each $x\in Y$ sends mass $1$ to $\pi (x)$, and sends mass $0$ to all other points. 
Then each $x\in Y$ sends out mass exactly $1$, but each $z\in D$ receives infinite mass: since all $\mathcal{R}_0$ and $\mathcal{R}_1$ classes are nontrivial, there are infinitely many geodesic paths in $G$ starting at $z$ and ending in $Y$ that never pass through the vertex $[z]_{\mathcal{R}_0}$, and the end point of each such path sends mass $1$ to $z$.  
If $D$ were not $\mu$-null then this would contradict the mass transport principle.
\end{proof}

We say a group $\Gamma$ admits an \textbf{essential measurable splitting} if there exists a free p.m.p.\ action $\Gamma\acts (X,\mu)$ whose associated orbit equivalence relation admits a $\mu$-essential splitting (by considering ergodic decompositions, it is clear that the existence of such an action implies the existence of one which is additionally ergodic). 
This is equivalent to the group \textbf{not} being measurably freely indecomposable in the sense of Alvarez and Gaboriau \cite{AlvarezGaboriau2012} (cf. \cite[Proposition 4.8]{AlvarezGaboriau2012}). 

\begin{prop}[\cite{AlvarezGaboriau2012}]\label{prop:MEtoFreeProductsAdmitEssentialSplittings}
The class of countable groups that admit an essential measurable splitting is invariant under measure equivalence. 
Every countable group that is measure equivalent to a free product of two nontrivial groups belongs to this class. 
\end{prop}

\begin{proof}
The first statement is \cite[Proposition 4.13]{AlvarezGaboriau2012}.
The second statement is an immediate consequence of the first and Proposition \ref{prop:freeProductRelationsAdmitEssentialSplittings}.
\end{proof}

Some examples of groups $\Gamma$ that admit essential measurable splittings (such as infinite treeable groups and more generally groups measure equivalent to free products of nontrivial groups) can be found in \cite{AlvarezGaboriau2012,Bridson2007, CGMTD2021,Gaboriau2005}.

It is known that, for a nonamenable group $\Gamma$, if $\beta^{(2)}_1(\Gamma)=0$ then $\Gamma$ does not admit an essential measurable splitting \cite{AlvarezGaboriau2012}. 
The converse (a question first raised in \cite{AlvarezGaboriau2012}) remains an open problem.

\section{Bi-cofinitely equivariant maps in the groupoid setting}\label{section:bicofGroupoids}

In this section, we prove Theorem \ref{thm:cofequivfreefactorgroupoids} which is a groupoid version of Theorem \ref{cofFinCosetReprError}.
The results in this subsection allow us to conclude, for example, that the groups $B\wr \pi_1(\Sigma_g)$ and $(B\times \Z )\wr\pi_1(\Sigma_g)$ are orbit equivalent for every nontrivial group $B$, where $\Sigma_g$ is a closed orientable surface of genus $g\geq 1$. 

Fix a discrete p.m.p.\ groupoid $\mathcal{G}$. 
Given a standard probability space $(Y,\nu )$, we make the set $Y^{\otimes \mathcal{G}}\coloneqq \bigsqcup _{x\in \mathcal{G}^0}Y^{x\mathcal{G}}$ into a standard Borel space
by equipping it with the sigma-algebra generated by the projection map $p : Y^{\otimes\mathcal{G}}\to \mathcal{G}^0$, 
together with all maps of the form $Y^{\otimes\mathcal{G}}\to Y$, $y\mapsto y_{\sigma (p(y))}$, 
where $\sigma$ varies over all Borel sections of the range map $r:\mathcal{G}\to \mathcal{G}^0$.
We equip $Y^{\otimes \mathcal{G}}$ with the probability measure $\nu ^{\otimes \mathcal{G}}\coloneqq \int _{\mathcal{G}^0}\nu ^{x\mathcal{G}}\, d\mu _{\mathcal{G}^0}(x)$.
Thus, given a discrete p.m.p.\ groupoid $\mathcal{K}$, the unit space of the direct sum bundle $\boldsymbol{\bigoplus}_{\mathcal{G}}\mathcal{K}$ is precisely $((\mathcal{K}^0)^{\otimes\mathcal{G}}, \mu _{\mathcal{K}^0}^{\otimes\mathcal{G}})$.

Let $\mathcal{B}$ and $\mathcal{C}$ be principal discrete p.m.p.\ groupoids, and consider the corresponding actions of $\mathcal{G}$ on the direct sum bundles $\boldsymbol{\bigoplus}_{\mathcal{G}}\mathcal{B}$ and $\boldsymbol{\bigoplus}_{\mathcal{G}}\mathcal{C}$.
Suppose that $\phi : ((\mathcal{B}^0)^{\otimes \mathcal{G}}, \mu _{\mathcal{B}^0}^{\otimes\mathcal{G}} )\rightarrow ((\mathcal{C}^0)^{\otimes \mathcal{G}},\mu _{\mathcal{C}^0}^{\otimes \mathcal{G}})$ is a fiberwise isomorphism of measure spaces, i.e., 
there is a $\mu _{\mathcal{B}^0}^{\otimes\mathcal{G}}$-conull set on which $\phi$ is injective and $\phi _*\mu _{\mathcal{B}^0} ^{x\mathcal{G}} = \mu _{\mathcal{C}^0} ^{x\mathcal{G}}$ for almost every $x\in \mathcal{G}^0$.
We say that $\phi$ is {\bf cofinitely $\mathcal{G}$-equivariant (with respect to $\mathcal{B}$ and $\mathcal{C}$)} if for each $g\in\mathcal{G}$, for almost every $b\in \bigoplus _{s(g)\mathcal{G}}\mathcal{B}$ there exists some $c\in\bigoplus _{s(g)\mathcal{G}}\mathcal{C}$ with $s(c)=\phi (s(b))$ and $r(c)=g^{-1}.\phi (g.r(b))$.
Note that such a $c$ is necessarily unique (since $\mathcal{C}$ is principal), so we denote it by $\phi _g(b)\coloneqq c$. 
We say $\phi$ is {\bf bi-cofinitely $\G$-equivariant} if both $\phi$ and $\phi^{-1}$ are cofinitely equivariant.

\begin{prop}\label{prop:cofEqToOEGroupoid}
Let $\mathcal{B}$ and $\mathcal{C}$ be principal discrete p.m.p.\ groupoids. 
Let $\G$ be a discrete p.m.p.\ groupoid, and assume that $\phi: ((\mathcal{B}^0)^{\otimes\G},\mu _{\mathcal{B}^0}^{\otimes\G})\rightarrow ((\mathcal{C}^0)^{\otimes\G},\mu _{\mathcal{C}^0}^{\otimes\G})$ is a bi-cofinitely $\G$-equivariant fiberwise isomorphism with respect to $\mathcal{B}$ and $\mathcal{C}$. 
Then the map $\Phi :\mathcal{B}\wr\mathcal{G}\to\mathcal{C}\wr\mathcal{G}$ defined by $\Phi(g,b)\coloneqq (g,\phi _g(b))$ for $g\in \G$ and $b\in \bigoplus_{s(g)\G}\mathcal{B}$ is an isomorphism of the wreath product groupoids $\mathcal{B}\wr\G$ and $\mathcal{C}\wr \G$. 
\end{prop}

\begin{proof}
The direct sum bundle groupoids are both principal since $\mathcal{B}$ and $\mathcal{C}$ are.
Thus, for $g,h\in \G$, $a\in \bigoplus_{s(g)\G}\mathcal{B}$, and $b\in \bigoplus_{s(h)\G}\mathcal{B}$ with $s(g,a)=r(h,b)$, checking sources and ranges verifies the equality $\phi _{gh}((h^{-1}.a)b)=(h^{-1}.\phi _g(a))\phi _h(b)$, and hence
\[
    \Phi((g,a)(h,b))=(gh,\phi_{gh}((h^{-1}.a)b))=(gh,(h^{-1}.\phi_g(a))\phi_h(b))=\Phi(g,a)\Phi(h,b),
\]
so that $\Phi$ is homomorphism. 
Defining $\tilde\Phi(g,c)\coloneqq (g,(\phi^{-1})_{g}(c))$ confirms that $\Phi$ is invertible with inverse $\tilde \Phi$.
\end{proof}

We now adapt Theorems \ref{cofAmenable} and \ref{cofFinCosetReprError} to the setting of wreath product groupoids. 

\begin{thm} \label{cofAmenableR}
Let $\R$ be an aperiodic amenable p.m.p.\ equivalence relation on $(X,\mu)$. 
Let $A_0\acts(Y_0,\nu _0)$, $A_1 \acts(Y_1,\nu _1)$ be free ergodic p.m.p.\ actions of (possibly finite) amenable countable groups. 
Then there exists a bi-cofinitely $\R$-equivariant fiberwise isomorphism (with respect to the orbit equivalence relations of $A_0\acts(Y_0,\nu _0)$ and $A_1 \acts(Y_1,\nu _1)$) between $(Y_0^{\otimes \R},\nu _0^{\otimes\R})$ and $(Y_1^{\otimes \R},\nu _1^{\otimes \R})$. 
\end{thm}

\begin{proof}
Fix an infinite amenable group $\Gamma$ and consider the wreath product actions $A_0\wr \Gamma\curvearrowright (Y_0^\Gamma, \nu_0 ^\Gamma )$ and $A_1\wr\Gamma\acts (Y_1^\Gamma,\nu_1^\Gamma)$. 
Let $\psi :(A_0\wr\Gamma )/\bigoplus_\Gamma A_0\rightarrow (A_1\wr\Gamma )/\bigoplus_\Gamma A_1$ denote the natural isomorphism $\psi(\gamma (\bigoplus_\Gamma A_0))=\gamma (\bigoplus_\Gamma A_1)$. 
By Theorem \ref{FSZ} there exists an orbit equivalence $\theta : (Y_0^\Gamma ,\nu_0 ^\Gamma )\rightarrow (Y_1^\Gamma , \nu_1 ^\Gamma )$ between the wreath product actions which satisfies $\theta(\gamma (\bigoplus_\Gamma A_0) .y)=\gamma (\bigoplus_\Gamma A_1).\theta(y)$ for all $\gamma\in A_0\wr\Gamma$ and almost every $y\in Y_0^\Gamma$.

By Dye's Theorem and Connes-Feldman-Weiss \cite{CFW81}, since the equivalence relation $\R$ is aperiodic and amenable, it can be generated by a free p.m.p.\ action $\Gamma\acts(X,\mu)$.
Consider the diagonal action $A_0\wr\Gamma\acts (X\times Y_0^\Gamma,\mu\times \nu_0^\Gamma)$, where the action on the $Y_0^\Gamma$-coordinate is the wreath product action and the action on the $X$-coordinate is the action of $\Gamma$ on $X$ applied through the quotient map from $A_0\wr\Gamma$ to $\Gamma$. 
Likewise, we consider the diagonal action $A_1\wr\Gamma\acts(X\times Y_1^\Gamma,\mu\times \nu_1^\Gamma)$. 
These two actions are free, and the map $\mathrm{id}_X \times \theta: X\times Y_0^\Gamma \rightarrow X\times Y_1^\Gamma$, where $\mathrm{id}_X$ is the identity map on the $X$, is an orbit equivalence between them. This yields an isomorphism of action groupoids 
\[
\tilde{\theta} : (A_0\wr\Gamma )\ltimes (X\times Y_0^\Gamma,\mu\times \nu_0^\Gamma) \rightarrow (A_1\wr\Gamma )\ltimes (X\times Y_1^\Gamma,\mu\times \nu_1 ^\Gamma) .
\]
Let $\mathcal{A}_0 = A_0\ltimes (Y_0^{\Gamma},\nu_0 ^{\Gamma})$ and let $\mathcal{A}_1=A_1\ltimes (Y_1^{\Gamma},\nu_1 ^{\Gamma} )$. 
By identifying $\mathcal{R}$ with $\Gamma\ltimes (X,\mu )$, we may naturally identify $\mathcal{A}_0\wr\mathcal{R}$ and $\mathcal{A}_1\wr\mathcal{R}$ with $(A_0\wr\Gamma )\ltimes (X\times Y_0^\Gamma,\mu\times \nu_0^\Gamma)$ and $(A_1\wr\Gamma )\ltimes (X\times Y_1^\Gamma,\mu\times \nu_1 ^\Gamma)$ respectively (see Remark \ref{rem:action_wreath}). Under these identifications, $\tilde{\theta}$ corresponds to an isomorphism of $\mathcal{A}_0\wr\mathcal{R}$ with $\mathcal{A}_1\wr\mathcal{R}$ which takes $\{ x \}\times\bigoplus _{x\mathcal{R}}\mathcal{A}_0$ to $\{x\}\times\bigoplus _{x\mathcal{R}}\mathcal{A}_1$ for each $x\in X$, and hence we obtain a bi-cofinitely $\mathcal{R}$-equivariant fiberwise isomorphism.
\end{proof}

\begin{thm}\label{thm:cofequivfreefactorgroupoids}
Let $\mathcal{H}$ be a subgroupoid of a discrete p.m.p.\ groupoid $\mathcal{G}$ with $\mathcal{H}^0=\mathcal{G}^0$, and assume that there exists a Borel set $\Sigma \subseteq \mathcal{G}$ of left coset representatives for $\mathcal{H}$ in $\mathcal{G}$ satisfying $|g(s(g)\Sigma )\triangle (r(g)\Sigma )|<\infty$ for all $g\in \mathcal{G}$.
Let $\mathcal{B}$ and $\mathcal{C}$ be principal discrete p.m.p. groupoids with measures $\nu$ and $\eta$ on their unit spaces, respectively.  

If there exists a bi-cofinitely $\mathcal{H}$-equivariant fiberwise isomorphism of measure spaces $\phi _{\mathcal{H}} : ((\mathcal{B}^0)^{\otimes \mathcal{H}} , \nu ^{\otimes\mathcal{H}} )\rightarrow ((\mathcal{C}^0)^{\otimes \mathcal{H}}, \eta ^{\otimes\mathcal{H}})$ then there exists a bi-cofinitely $\mathcal{G}$-equivariant fiberwise isomorphism of measure spaces $\phi : ((\mathcal{B}^0)^{\otimes \mathcal{G}} , \nu ^{\otimes\mathcal{G}} )\rightarrow ((\mathcal{C}^0)^{\otimes \mathcal{G}}, \eta ^{\otimes\mathcal{G}})$.
\end{thm}

\begin{proof}
The proof mirrors that of Theorem \ref{cofFinCosetReprError}. 
Throughout the proof we write $\sim _{\mathcal{B}}$ for both the relations $\mathcal{R}_{\boldsymbol{\bigoplus}_\G\mathcal{B}}$ and $\mathcal{R}_{\boldsymbol{\bigoplus}_\mathcal{H}\mathcal{B}}$,
and we write $\sim _{\mathcal{C}}$ for both the relations $\mathcal{R}_{\boldsymbol{\bigoplus}_\G\mathcal{C}}$ and $\mathcal{R}_{\boldsymbol{\bigoplus}_\mathcal{H}\mathcal{C}}$. 
Let $\sigma :\mathcal{G}/\mathcal{H} \rightarrow \Sigma$ be the map sending the left coset $g\mathcal{H}$ to its coset representative in $\Sigma$. We will also write $\sigma (g)$ for $\sigma (g\mathcal{H})$. Let $\rho : \mathcal{G} \otimes _{\mathcal{G}^0}  \mathcal{G}/\mathcal{H} \rightarrow \mathcal{H}$ denote the associated cocycle into $\mathcal{H}$, defined by $\rho (g, \alpha \mathcal{H} ) = \sigma (g \alpha \mathcal{H} )^{-1}g \sigma (\alpha \mathcal{H})$. Define $\phi :(\mathcal{B}^0)^{\otimes \mathcal{G}} \rightarrow (\mathcal{C}^0)^{\otimes \mathcal{G}}$ by 
\[
\phi (y)_g \coloneqq \phi _{\mathcal{H}} \big( (\sigma (g)^{-1}.y)|_{s(\sigma (g))\mathcal{H}}\big)_{\sigma (g)^{-1}g}
\]
for $y\in (\mathcal{B}^0)^{\otimes \G}$ and $g\in \G$. 
Note that $\phi$ is bijective with inverse given by 
\[
\phi ^{-1}(z)_g = \phi _{\h}^{-1} \big( (\sigma (g )^{-1}.z)|_{s(\sigma(g))\h}\big) _{\sigma (g )^{-1}g}
\]
for $z\in (\mathcal{C}^0)^{\otimes \G}$ and $g\in \G$. 

Given $x\in \mathcal{G}^0$ and $y_0,y_1\in (\mathcal{B}^0)^{x\mathcal{G}}$ with $y_0\sim_{\mathcal{B}} y_1$, we have the equality $(\alpha ^{-1}.y_0)|_{s(\alpha )\mathcal{H}}= (\alpha ^{-1}.y_1)|_{s(\alpha )\mathcal{H}}$ for all but finitely many $\alpha \in x\Sigma$. 
For each $\alpha \in x\Sigma$ for which this equality holds we have that $\phi (y_0)_{g}=\phi (y_1)_{g}$ for all $g \in \alpha\mathcal{H}$. 
For each of the finitely many $\alpha\in x\Sigma$ for which we do not have equality we still have almost equality $(\alpha ^{-1}.y_0)|_{s(\alpha )\mathcal{H}}\sim_{\mathcal{B}} (\alpha ^{-1}.y_1)|_{s(\alpha )\mathcal{H}}$, 
so since $\phi _{\mathcal{H}}$ is bi-cofinitely equivariant we see that $\phi (y_0)_g=\phi (y_1)_g$ for all but finitely many $g \in \alpha \mathcal{H}$. 
Therefore, $\phi (y_0)\sim_{\mathcal{C}} \phi (y_1)$.

To show that $\phi$ is cofinitely equivariant, it remains to show that for $k \in \mathcal{G}$ and $y\in (\mathcal{B}^0)^{s(k)\mathcal{G}}$ we have $\phi (k.y)\sim_{\mathcal{C}} k.\phi (y)$. 
Given $k\in\mathcal{G}$, for each $\alpha \in r(k)\Sigma$ and $g \in \alpha \mathcal{H}$ we have 
\begin{align}
\nonumber \phi (k.y)_g&= \phi _{\mathcal{H}} \big( (\alpha ^{-1}k.y) |_{s(\alpha )\mathcal{H}} \big)_{\alpha ^{-1}g} \\
\label{eqn:groupoidcofeq} (k .\phi (y) )_{g} &= \phi _{\mathcal{H}}\big( (\sigma (k ^{-1}g )^{-1}.y)|_{s(\sigma (k^{-1}g))\mathcal{H}}\big) _{\sigma (k ^{-1}g)^{-1}k^{-1}g}  \\
\nonumber &= \phi _{\mathcal{H}}\big(\rho (k^{-1},\alpha \mathcal{H}).((\alpha ^{-1}k.y)|_{s(\alpha )\mathcal{H}})\big) _{\rho (k ^{-1},\alpha \mathcal{H})\alpha ^{-1}g}
\end{align}
If $\alpha\in r(k)\Sigma\cap ks(k)\Sigma$ then $\sigma (k^{-1}g)=k^{-1}\alpha$ and $\rho (k^{-1}, \alpha\mathcal{H} ) = s(\alpha )$, 
hence we have equality $\phi (k.y)_g = (k.\phi (y) )_g$ for all $g \in \alpha\mathcal{H}$. If $\alpha\in r(k)\Sigma\setminus k (s(k)\Sigma )$, then since $\phi _{\mathcal{H}}$ is bi-cofinitely equivariant and $\rho (k ^{-1},\alpha \mathcal{H} )\in \mathcal{H}$ we have 
\[
\phi _{\mathcal{H}}\big( \rho (k ^{-1},\alpha \mathcal{H}).((\alpha ^{-1}k.y)|_{s(\alpha )\mathcal{H}})\big) \sim_{\mathcal{C}} \rho (k^{-1},\alpha\mathcal{H}).\phi _{\mathcal{H}}\big((\alpha ^{-1}ky)|_{s(\alpha )\mathcal{H}} \big) ,
\]
and hence \eqref{eqn:groupoidcofeq} shows that $(k.\phi (y) )_g=\phi (k.y)_g$ for all but finitely many $g \in \alpha\mathcal{H}$. Since $r(k)\Sigma\setminus k (s(k)\Sigma )$ is finite and $r(k)\mathcal{G} = \bigcup _{\alpha\in r(k)\Sigma}\alpha\mathcal{H}$ 
we conclude that $k.\phi (y)\sim_{\mathcal{C}} \phi (k.y)$ and thus $\phi$ is cofinitely equivariant. 
Similarly, we conclude that $\phi^{-1}$ is cofinitely equivariant.

It remains to show that for each $x\in \mathcal{G}^0$, the restriction $\phi ^x :((\mathcal{B}^0)^{x\mathcal{G}},\nu ^{x\mathcal{G}}) \rightarrow ((\mathcal{C}^0)^{x\mathcal{G}}, \eta ^{x\mathcal{G}})$ of $\phi$ to the fiber over $x$ is a measure space isomorphism. 
For each $\alpha \in x\Sigma$ the map $(\mathcal{B}^0)^{x\mathcal{G}}\rightarrow (\mathcal{C}^0)^{\alpha\mathcal{H}}$, $y\mapsto \phi (y)|_{\alpha \mathcal{H}}$ only depends on $y|_{\alpha\mathcal{H}}$, and hence descends to a map $\phi ^x_{\alpha}:(\mathcal{B}^0)^{\alpha\mathcal{H}}\rightarrow (\mathcal{C}^0)^{\alpha\mathcal{H}}$. 
Therefore, under the natural identifications $(\mathcal{B}^0)^{x\mathcal{G}} = \prod _{\alpha \in x\Sigma} (\mathcal{B}^0)^{\alpha \mathcal{H}}$ and $(\mathcal{C}^0)^{x\mathcal{G}} = \prod _{\alpha \in x\Sigma} (\mathcal{C}^0)^{\alpha\mathcal{H}}$, we may identify $\phi ^x$ with $\prod _{\alpha\in x\Sigma} \phi ^x_{\alpha}$, and hence it is enough to show that each $\phi ^x_{\alpha}$ is a measure space isomorphism. 
Each of the maps $\phi ^x_{\alpha}$ is a composition $\phi ^x_{\alpha} = \alpha _{\mathcal{C}}\circ \phi ^x _{\mathcal{H}}\circ \alpha _{\mathcal{B}}^{-1}$, where $\alpha _{\mathcal{B}} : (\mathcal{B}^0)^{s(\alpha )\mathcal{H}}\rightarrow (\mathcal{B}^0)^{\alpha \mathcal{H}}$ is the map $\alpha _{\mathcal{B}}(y)_{\alpha h}=y_h$, the map $\alpha _{\mathcal{C}}$ is defined similarly, and $\phi ^x_{\mathcal{H}}$ denotes the restriction of $\phi _{\mathcal{H}}$ to the fiber over $x$. 
Therefore, $\phi ^x_{\alpha}$ is a measure space isomorphism, since each of the maps $\alpha _{\mathcal{B}}$, $\phi ^x_{\mathcal{H}}$, and $\alpha _{\mathcal{C}}$ are.
\end{proof}

\section{Measure equivalence and wreath products}
\label{section:wreathME}
\subsection{Fixed wreath complement}
It is fairly immediate that \ref{prop:freeFactorsHaveNiceCosetRep} and \ref{prop:cofEqToOEGroupoid}, Lemma \ref{lem:freeDecomposableImpliesTreeableFreeFactor}, and Theorems \ref{cofAmenableR} and \ref{thm:cofequivfreefactorgroupoids} together imply that $A_0\wr \Gamma$ is orbit equivalent to $A_1\wr \Gamma$ whenever $A_0$ and $A_1$ are nontrivial (possibly finite) amenable groups and $\Gamma$ is a group that admits an essential measurable splitting (generalizing Corollary \ref{cor:freeFactorAmenableLampGroups}). 
However, a bit more work allows us to relax the amenability assumption on $A_0$ and $A_1$, replacing it instead with the assumption that the groups $A_0\times \mathbf{Z}$ and $A_1\times \mathbf{Z}$ are measure equivalent.

\begin{thm}\label{thm:MELampGroups}
Let $\Gamma$ be a countable group that admits an essential measurable splitting.
\begin{enumerate}
    \item The groups $B\wr\Gamma$ and $(B\times A) \wr \Gamma$ are orbit equivalent for every nontrivial countable group $B$ and every countable amenable group $A$.
    \item Let $B_0$ and $B_1$ be nontrivial countable groups and let $A_0$ and $A_1$ be countable amenable groups. 
    If $B_0\times A_0$ and $B_1\times A_1$ are measure equivalent, then $B_0\wr\Gamma$ and $B_1\wr\Gamma$ are orbit equivalent.
\end{enumerate}
In particular, if $B_0$ and $B_1$ are nontrivial countable groups which are measure equivalent, then $B_0\wr\Gamma$ and $B_1\wr\Gamma$ are orbit equivalent.
\end{thm}

\begin{proof}
We begin with the proof of (1). Fix a free ergodic p.m.p.\ action $\Gamma\acts(X,\mu)$ with action groupoid $\G$ that admits a $\mu$-essential splitting. By Lemma \ref{lem:freeDecomposableImpliesTreeableFreeFactor}, $\G$ contains an aperiodic amenable equivalence relation $\R$ as a free factor.  

If $B$ is infinite then let $n=2$, and otherwise let $n$ denote the cardinality of $B$.
Let $Z_n$ be a set of cardinality $n$, let $\eta _n$ be normalized counting measure on $Z_n$, and let $\mathcal{C}_n$ be the p.m.p.\ equivalence relation on $Z_n$ where all points of $Z_n$ are equivalent to each other.
Fix a free ergodic p.m.p.\ action $A\acts (Z,\eta)$ and denote the associated orbit equivalence relations by $\mathcal{A}$, respectively. 
By Theorem \ref{cofAmenableR}, there is a bi-cofinitely $\R$-equivariant fiberwise isomorphism, with respect to $\mathcal{C}_n$ and $\mathcal{C}_n\oplus \mathcal{A}$, from $(Z_n^{\otimes \R},\eta_n ^{\otimes \R})$ to $((Z_n\times Z)^{\otimes \R},(\eta _n\times \eta)^{\otimes\R})$.
Applying Theorem \ref{thm:cofequivfreefactorgroupoids} and Proposition \ref{prop:freeFactorsHaveNiceCosetRep} gives a bi-cofinitely $\G$-equivariant fiberwise isomorphism $\phi : (Z_n^{\otimes \G},\eta _n^{\otimes \G})\to ((Z_n\times Z)^{\otimes \G},(\eta_n\times\eta)^{\otimes\G})$,
with respect to $\mathcal{C}_n$ and $\mathcal{C}_n\oplus \mathcal{A}$.

Fix any free ergodic p.m.p.\ action $B\acts(W,\rho)$ along with the associated orbit equivalence relation $\mathcal{B}$. 
Take any subset $U\subseteq W$ of measure $\frac{1}{n}$. 
Then $\mathcal{B}$ is isomorphic to the direct sum $\mathcal{B}_U\oplus \mathcal{C}_n$. 
Define a fiberwise measure space isomorphism $\tilde\phi: ((U\times Z_n)^{\otimes \G},(\rho|_U\times\eta_n)^{\otimes \G})\rightarrow((U\times Z_n\times Z)^{\otimes \G},(\rho|_U \times\eta_n\times\eta)^{\otimes\G})$ as follows. 
For each $x\in \G^0$, the map $\tilde \phi$ takes the fiber $(U\times Z_n)^{x\G}$ to the fiber $(U\times Z_n\times Z)^{x\G}$ via $\tilde{\phi}(u,z)=(u, \phi(z))$ for $u\in U^{x\G}$ and $z\in Z_n^{x\G}$;
here we are using the natural identification of $(U\times Z_n)^{x\G}$ with $U^{x\G}\times Z_n^{x\G}$ and of $(U\times Z_n\times Z)^{x\G}$ with $U^{x\G}\times (Z_n\times Z)^{x\G}$.

The map $\tilde \phi$ is bi-cofinitely $\G$-equivariant with respect to $\mathcal{B}_U\oplus \mathcal{C}_n$ and $\mathcal{B}_U\oplus \mathcal{C}_n\oplus\mathcal{A}$ (which are isomorphic to $\mathcal{B}$ and $\mathcal{B}\oplus \mathcal{A}$, respectively) and hence, by Proposition \ref{prop:cofEqToOEGroupoid}, the groupoids $\mathcal{B}\wr\G$ and $(\mathcal{B}\oplus \mathcal{A})\wr\G$ are isomorphic.

By Remark \ref{rem:action_wreath}, the groupoids $\mathcal{B}\wr\G$ and $(\mathcal{B}\oplus \mathcal{A})\wr\G$ are naturally isomorphic to orbit equivalence relations of free p.m.p.\ ergodic actions of $B\wr\Gamma$ and $(B\times A)\wr\Gamma$, respectively, and therefore $B\wr\Gamma$ and $(B\times A)\wr\Gamma$ are orbit equivalent. This concludes point (1).

(2): Assume that $B_0\times A_0$ and $B_1\times A_1$ are measure equivalent. 
Then $B_0\times A_0\times \mathbf{Z}$ is measure equivalent to $B_1\times A_1\times \mathbf{Z}$, and since these groups have a copy of $\mathbf{Z}$ as a direct factor, it follows that $B_0\times A_0\times \mathbf{Z}$ and $B_1\times A_1\times \mathbf{Z}$ are in fact orbit equivalent. 
By Proposition \ref{OElampsWreath}, $(B_0\times A_0\times \mathbf{Z})\wr\Gamma$ and $(B_1\times A_1\times \mathbf{Z})\wr\Gamma$ are orbit equivalent. 
Part (2) now follows from part (1), since by part (1) for each $i\in \{ 0, 1 \}$ the group $B_i\wr \Gamma$ is orbit equivalent to $(B_i\times A_i\times \mathbf{Z})\wr \Gamma$.
\end{proof}

\subsection{Fixed wreath kernel}

Towards understanding the situation of measure equivalent wreath complement groups, we prove the following lemma that analyzes what occurs when we restrict the wreath complement groupoid of a wreath product to a positive measure set.

\begin{lemma}\label{lem:wreathProductRestriction}
    Let $\G$ be a discrete p.m.p.\ groupoid. Assume $\K$ is a discrete p.m.p.\ groupoid that satisfies $\K\cong \bigoplus _\mathbf{N} \K$. Let $p: (\K\wr\G)^0\to \G^0$ be the natural projection, and let $D$ be a positive measure subset of $\mathcal{G}^0$. Then the groupoids $(\K\wr \G)_{p^{-1}(D)}$ and $\K\wr (\G_D )$ are isomorphic.
\end{lemma}

\begin{proof} 
For each $\G$-invariant subset $C$ of $\G^0$, the preimage $p^{-1}(C)$ is a $\K\wr\G$-invariant subset of $(\K\wr\G)^0$. Therefore, by considering the ergodic decomposition of $\G$, we may assume without loss of generality that $\G$ is ergodic.

Consider the fibered Borel spaces $r: \G\to \G^0$ and $r: \G D\to \G^0$, each equipped with the left translation action of $\G$, along with the associated Borel bundles of discrete p.m.p.\ groupoids $\boldsymbol{\bigoplus}_{\G} \K$ and $\boldsymbol{\bigoplus}_{\G D} \K$ each equipped with the associated shift action of $\G$.  

\begin{claim}
    There exists a measure-preserving groupoid isomorphism $\theta: \boldsymbol{\bigoplus}_{\G} \K\to \boldsymbol{\bigoplus}_{\G D} \K$ that is $\G$-equivariant, i.e., satisfies $\theta(\bigoplus_{s(g)\G} \K)=\bigoplus_{s(g)\G D}\K$ and $\theta(g. k)=g. \theta(k)$ for each $g\in \G$ and $k\in \bigoplus_{s(g)\G}\K$.
\end{claim}

\begin{proof}[Proof of Claim] 
We first prove the claim in the case that $\mu_{\G^0}(D)=\frac{1}{n}$ for some integer $n\geq 1$.

Fix a partition $(E_j)_{j=0}^{n-1}$ of $\G^0$ into sets of measure $\frac{1}{n}$ with $E_0=D$. 
Since $\G$ is ergodic and p.m.p., for each $0\leq j <n-1$ we may find a partial bisection $\sigma _j \subseteq \G$ with $r_\G(\sigma _j)=D$ and $s_\G(\sigma _j)=E_j$.
The assumption implies the groupoids $\K$ and $\bigoplus_{i=0}^{n-1}\K$ are isomorphic, say via the isomorphism 
$\psi : \bigoplus_{i=0}^{n-1}\K\rightarrow \K$. Define the map $\theta : \boldsymbol{\bigoplus}_{\G} \K\to \boldsymbol{\bigoplus}_{\G D} \K$ by \[
\theta(k)_h = \psi(k_{h\sigma _0}, \dots, k_{h\sigma _{n-1}})
\]
for $k\in \bigoplus _{x\G} \K$, $h\in x\G D$, and $x\in \mathcal{G}^0$. 
The map $\theta$ is the desired measure preserving $\G$-equivariant groupoid isomorphism in this case.

We now consider the case of a general $D$. Let $\{n_i\}_{i\in I}$ be a countable collection of positive integers such that $\sum_{i\in I} \frac{1}{n_i}=\mu_{\G^0}(D)$ and let $(D_i)_{i\in I}$ be a Borel partition of $D$ with $\mu_{\G^0}(D_i)=\frac{1}{n_i}$. By the previous case, for each $i\in I$ there exists a measure-preserving $\mathcal{G}$-equivariant groupoid isomorphism $\theta _i : \boldsymbol{\bigoplus}_{\G} \K\to\boldsymbol{\bigoplus}_{\G D_i} \K$. Write $\theta _i ^x$ for the restriction of $\theta _i$ to the fiber over $x\in \mathcal{G}^0$. 

We define $\theta :\boldsymbol{\bigoplus}_{\G} \K\to \boldsymbol{\bigoplus}_{\G D} \K$ fiberwise as follows. Fix a groupoid isomorphism $\Psi:\K\to\bigoplus _{I} \K$. For each $x\in \G^0$ let $\Psi^{x\G}: \bigoplus _{x\G}\K \to \bigoplus_{x\G} (\bigoplus _{I}\K)$ be the map that applies $\Psi$ in every coordinate.
Let $\tau^x: \bigoplus_{x\G} (\bigoplus _{I}\K)\rightarrow \bigoplus_{I} (\bigoplus _{x\G}\K)$ be the ``transpose'' map that interchanges the $I$-{} and $x\G$-coordinates.
Define the map $\theta^x: \bigoplus_{x\G}\K\to \bigoplus_{x\G D}\K$ by 
\[
\theta^x(k)_h\coloneqq \theta_{i}^x(\tau^x(\Psi^{x\G}(k))_i)_h
\]
for $k\in \bigoplus_{x\G}\K$ and $h\in x\G D_i$. The map $\theta^x$ is the composition of the following measure preserving groupoid isomorphisms
\begin{equation*}
\label{eqn:central_sum_ses} 
\begin{tikzcd}
\bigoplus_{x\G}\K\arrow{r}{\Psi^{x\G}}& [+3pt]\bigoplus_{x\G} \bigoplus _{I}\K\arrow{r}{\tau^x}
&[-10pt]\bigoplus_{I} \bigoplus _{x\G}\K\arrow{r}{\bigoplus _i\theta_{i}^x} & [+13pt] \bigoplus_i\bigoplus_{x\G D_i}\K \arrow{r}{\cong}& [-8pt]\bigoplus_{x\G D}\K 
\end{tikzcd} 
\end{equation*}
where the last isomorphism is the obvious one. The resulting map $\theta$ is as claimed.
\end{proof}

Let $\theta: \boldsymbol{\bigoplus}_{\G} \K\to \boldsymbol{\bigoplus}_{\G D} \K$ be given by the claim. 
Then $\theta$ restricts to a $\mathcal{G}_D$-equivariant measure-preserving groupoid isomorphism $\theta _D$, from $\boldsymbol{\bigoplus}_{D\G} \K$ to $\boldsymbol{\bigoplus}_{D\G D} \K$.
The map $\theta_D$ in turn induces an isomorphism of the semidirect product groupoids $\mathcal{G}_D\ltimes \boldsymbol{\bigoplus}_{D\G} \K$ and $\mathcal{G}_D\ltimes \boldsymbol{\bigoplus}_{D\G D} \K$, which coincide with $(\mathcal{K}\wr\mathcal{G})_{p^{-1}(D)}$ and $\mathcal{K}\wr (\mathcal{G}_D)$ respectively, so the proof is complete.
\end{proof}

\begin{thm}\label{thm:METopGroups}
Let $\Gamma _0$, $\Gamma _1$, and $B$ be countable groups, and assume that $\Gamma_0$ and $\Gamma_1$ admit a measure equivalence coupling with coupling index $\alpha$. 
Then the groups $(\bigoplus _{\mathbf{N}}B)\wr \Gamma_0$ and $(\bigoplus _{\mathbf{N}}B)\wr\Gamma_1$ admit a measure equivalence coupling with coupling index $\alpha$.

It follows that if $A$ is an infinite amenable group and the groups $\Gamma_0$ and $\Gamma_1$ are measure equivalent, then the groups $ A\wr \Gamma_0$ and $A\wr\Gamma_1$ are measure equivalent.
\end{thm}

\begin{proof}
Fix a free ergodic p.m.p.\ action of $B\acts (Z,\eta)$ and consider the free ergodic action $\bigoplus_{\mathbf{N}} B\acts (Z^{\mathbf{N}},\eta^{\mathbf{N}})$, where $\bigoplus _{\mathbf{N}}B$ acts coordinatewise. 
The orbit equivalence relation $\s$ associated to this action satisfies $\s\cong \bigoplus _{\mathbf{N}}\s$.

After interchanging $\Gamma _0$ and $\Gamma _1$ if necessary, we may assume that $0<\alpha \leq 1$. 
Let $\Gamma_0\acts(X_0,\mu_0)$ and $\Gamma_1\acts(X_1,\mu_1)$ be a pair of free ergodic p.m.p.\ actions that realize a stable orbit equivalence of index $\alpha$ and let $\R_0$ and $\R_1$ be the associated orbit equivalence relations, respectively, so that there exists a Borel set $D\subseteq X_0$ of measure $\alpha$ such that $(\R_0)_D\cong \R_1$. 
We now consider the wreath product equivalence relations $\s\wr \R_0$ and $\s\wr\R_1$, which are themselves orbit equivalence relations of free ergodic p.m.p.\ actions of the wreath product groups $(\bigoplus _{\mathbf{N}}B)\wr \Gamma_0$ and $(\bigoplus _{\mathbf{N}}B)\wr \Gamma_1$, respectively. 
Let $p:(\s\wr\R_0)^0\to X_0$ be the natural projection map. 
By Lemma \ref{lem:wreathProductRestriction}, $(\s\wr\R_0)_{p^{-1}(D)}\cong \s\wr(\R_0)_D\cong \s\wr\R_1$, which completes the proof. \end{proof}

\begin{remark}
In general, one cannot expect orbit equivalence in the conclusion of the previous theorem. By \cite{Sako09}, if $B\wr (G \times H)$ and $C\wr (\Gamma\times \Lambda)$ are orbit equivalent for nonamenable exact groups $G$ and $\Gamma$, and infinite exact groups $\Lambda$ and $H$, then $G\times H$ and $\Gamma\times \Lambda$ must have been orbit equivalent in the first place. 
\end{remark}

\begin{remark}
If $\Gamma_0$ and $\Gamma_1$ are orbit equivalent, then it was previously known that $B\wr \Gamma_0$ and $B\wr\Gamma_1$ are orbit equivalent for any countable group $B$. 
This is implied by \cite[Proposition 7.1]{DelabieKoivistoLeMaitreTessera2022}.
\end{remark}

\begin{remark}
    Given an integer $d\geq 3$ and permutation groups $F\lneq F'\leq \mathrm{Sym}(d)$ with $F$ semi-regular, Le Boudec defines a certain group $G(F,F')$ of tree automorphisms, and shows that it is a lattice in a locally compact group that also admits the wreath product group $C_n\wr W_d$ as a lattice \cite{LeBoudecLocPrescribed, LeBoudecSimpleWreath21}; 
    here, $n$ is the index of $F$ in $F'$ and $W_d$ is the free product of $d$ copies of $\mathbf{C}_2$. In particular, $G(F,F')$ and $W_d$ are measure equivalent.
    Therefore, by applying Theorem \ref{thm:METopGroups} and Theorem \ref{thm:MELampGroups}, all such groups $G(F,F')$ are measure equivalent to one another and to $\mathbf{Z}\wr\mathbf{F}_2$. 
\end{remark}

\section{Rigidity}\label{section:rigidity}
In this section, our main goal is to show that there exist groups $\Gamma$ for which the wreath product actions $\mathbf{C}_2\wr\Gamma\acts 2^\Gamma$ and $\mathbf{C}_3\wr\Gamma\acts 3^\Gamma$ are not orbit equivalent.

Let $\Gamma$ be a countable group and let $\Gamma\acts (X,\mu)$ be a p.m.p.\ action of $\Gamma$. 
Given a Polish group $L$, a \textbf{measurable $L$-valued cocycle} of this action is a measurable map $c: \Gamma\times X\to L$ satisfying $c(\gamma _1\gamma _0, x)=c(\gamma _1,\gamma _0.x)c(\gamma _0, x)$ for all $\gamma _0,\gamma _1\in \Gamma$ and almost every $x\in X$. 
Two measurable $L$-valued cocycles $c_0$ and $c_1$ of this action are called \textbf{cohomologous} if there is a measurable map $f :X\rightarrow L$ such that $f (\gamma .x)c_0(\gamma , x)f (x)^{-1} = c_1(\gamma , x)$ for a.e.\ $x\in X$, for all $\gamma\in\Gamma$. 
The action $\Gamma\acts (X,\mu)$ is called {\bf $L$-cocycle superrigid} if for every $L$-valued cocycle $c: \Gamma \times X\rightarrow L$, there exists a group homomorphism $\rho: \Gamma\rightarrow L$ and measurable map $f: X\rightarrow L$ such that $c(\gamma,x)=f(\gamma .x)\rho(\gamma)f(x)^{-1}$. 
The action is called {\bf $\mathcal{G}_{\mathrm{ctble}}$-cocycle superrigid} if it is $L$-cocycle superrigid for every countable discrete group $L$.

As a consequence of Popa's cocycle superrigidity theorem \cite{Popa2007,Po08}, the Bernoulli shift $\Gamma\acts ([0,1]^\Gamma,\mathrm{Leb}^\Gamma)$ is $L$-cocycle superrigid for every Polish group $L$ that is either countable or compact, and every countable group $\Gamma$ that admits an infinite normal subgroup $N\triangleleft \Gamma$ such that either $(\Gamma,N)$ has relative property (T), or $N\cong H\times K$ where $H$ is nonamenable and $K$ is infinite.

\subsection{Preliminary lemmas}

Given groups $B$ and $\Lambda$, we equip $B^{\Lambda}$ with its product group structure, and for $\lambda \in \Lambda$ and $u\in B^{\Lambda}$ we write $\lambda  . u$ for the (left) shift of $u$ by $\lambda$, i.e., $(\lambda  . u)_{\delta}\coloneqq u_{\lambda ^{-1}\delta}$ for $\delta \in \Lambda$. 
We let $B\wr ^u \Lambda$ denote the \textbf{unrestricted (regular) wreath product} of $B$ with $\Lambda$, i.e., the associated semidirect product $\Lambda\ltimes B^{\Lambda}$, which we view as an internal semidirect product. Thus, $B\wr ^u \Lambda$ is the group generated by $B^{\Lambda}$ and $\Lambda$, subject to the relations $\lambda u\lambda ^{-1} = \lambda  . u$ for all $\lambda \in \Lambda$ and $u\in B^{\Lambda}$. 

\begin{lemma}[Houghton {\cite[Theorem 4.6]{houghton1972ends}}] \label{prop:One_ended_Implies_Conj}
Let $B$ and $\Lambda$ be groups.  If $\Gamma$ is a subgroup of $B\wr\Lambda$ that has trivial intersection with $\bigoplus _{\Lambda}B$, then $\Gamma$ is conjugate in $B\wr ^u \Lambda$ to a subgroup of $\Lambda$ via an element of $B^\Lambda$. 

Moreover, if $\Gamma$ has at most one end, then $\Gamma$ is conjugate in $B\wr \Lambda$ to a subgroup of $\Lambda$ via an element of $\bigoplus_\Lambda B$.
\end{lemma}

See also \cite[Theorem 2.4]{Hassanabadi1978} for related results. 

\begin{lemma}[Furman {\cite[Theorem 1.8]{Furman2007}}, Popa \cite{Popa2007}] \label{lemma:rigidity_prep}
Let $\Gamma$ and $L$ be countable groups, and assume that $\Gamma$ has no nontrivial finite normal subgroups.
Let $\Gamma\acts (X,\mu)$ be a free ergodic p.m.p.\ action of $\Gamma$ that is $L$-cocycle superrigid, and let $L\acts (Z, \eta )$ be a free ergodic p.m.p.\ action. 

Suppose that there exist positive measure sets $X_0\subseteq X$ and $Z_0\subseteq Z$ and a measure space isomorphism $\phi _0: (X_0,\mu _0)\rightarrow (Z_0,\eta _0)$ such that $\phi _0(\Gamma. x)\cap X_0\subseteq L.\phi _0(x)\cap Z_0$ for a.e.\ $x\in X_0$, where $\mu _0$ and $\eta _0$ are the normalized restrictions of $\mu$ and $\eta$ respectively.

Then, after discarding a null set, there exist measurable functions $f:X\to L$ and $\phi :X\rightarrow Z_0$ with $\phi |_{X_0}=\phi _0$, an injective group homomorphism $\rho :\Gamma\to L$ and a positive measure $\rho (\Gamma )$-invariant set $Z_1\subseteq Z$ such that the map $\psi :(X,\mu )\rightarrow (Z_1,\eta _1)$ defined by $\psi (x)\coloneqq f(x).\phi (x)$ is a measure space isomorphism (where  $\eta _1$ is the normalized restriction of $\eta$) satisfying $\psi (\gamma. x)=\rho (\gamma ).\psi  (x)$ for all $\gamma \in \Gamma$ and $x\in X$.
\end{lemma}

\begin{lemma}\label{lem:cocSuperrigidImpliesStronglyErgodic}
Let $\Gamma$ be a countable group, and let $\Gamma\acts (X,\mu )$ be an ergodic p.m.p.\ action that is $\mathbf{Z}$-cocycle superrigid. Then the action is strongly ergodic.
\end{lemma}

\begin{proof}
    Assume toward a contradiction that the action of $\Gamma$ is not strongly ergodic.  
    By Jones-Schmidt \cite[Theorem 2.1 and Remark 2.5]{JS}, there exists a free ergodic p.m.p.\ action $\mathbf{Z}\acts(Y,\nu)$ and a measure preserving map $\Psi: (X,\mu )\to (Y,\nu )$ such that $\Psi(\Gamma. x)=\mathbf{Z}.\Psi(x)$ for almost every $x\in X$. 
    Define the measurable cocycle $v: \Gamma\times X\to \mathbf{Z}$ implicitly by $\Psi(\gamma. x)=v(\gamma,x).\Psi(x)$. 
    Applying the superrigidity assumption yields a group homomorphism $\rho: \Gamma\to \mathbf{Z}$ and measurable map $g: X\to \mathbf{Z}$ such that $g(\gamma .x)v(\gamma, x)g(x)^{-1}=\rho(\gamma)$ for all $\gamma \in \Gamma$ and a.e.\ $x\in X$.

    Define $\theta: X\to Y$ by $\theta(x)= g(x).\Psi(x)$. 
    Then $\theta$ is $\Gamma$-equivariant with respect to the action of $\Gamma$ on $Y$ implemented through $\rho$. 
    Since  $\theta _*\mu$ is absolutely continuous with respect to $\nu$, and $\theta _*\mu$ and $\nu$ are both invariant ergodic probability measures for this action, it follows that $\theta _*\mu = \nu$. 
    Ergodicity of the action of $\Gamma$ on $X$ implies that $\rho (\Gamma )$ is nontrivial, and hence infinite cyclic.
    
    Let $N=\ker (\rho )$, and let $\theta _0 :(X,\mu )\rightarrow (Y_0,\nu _0)$ be the ergodic decomposition map for the action of $N$ on $(X,\mu )$. 
    Then there is an ergodic p.m.p.\ action of $\rho (\Gamma )$ on $(Y_0,\nu _0)$ satisfying $\theta _0 (\gamma .x)=\rho (\gamma ).\theta _0 (x)$ for all $\gamma \in \Gamma$ and $x\in X$.
    Note that, since $\theta$ is $N$-invariant, $\theta$ factors through $\theta _0$, and in particular $\nu _0$ is atomless.

    By Dye's Theorem and Lemma \ref{lemma:rigidity_prep}, we may find a measurable $\mathbf{Z}$-valued cocycle $w_0: \rho(\Gamma) \times Y_0\rightarrow \mathbf{Z}$ that is not cohomologous to a homomorphism, e.g.\ consider the orbit equivalence cocycle to a non-conjugate free p.m.p.\ action of $\mathbf{Z}$. 
    Let $w:\Gamma\times X\rightarrow \mathbf{Z}$ be its lift, given by $w(\gamma ,x)=w_0(\rho (\gamma ),\theta_0 (x))$. 
    Applying the superrigidity assumption again yields a group homomorphism $\pi :\Gamma \rightarrow \mathbf{Z}$ and a measurable map $h:X\rightarrow \mathbf{Z}$ such that $h(\gamma x)w(\gamma ,x )h(x)^{-1} =\pi (\gamma  )$, i.e.,
    \[
h(\gamma x)=\pi (\gamma  )h(x)w_0(\rho (\gamma ),\theta_0 (x))^{-1}.
    \]
    Then $h_*\mu$ is a $\pi (N)$-invariant probability measure on $\mathbf{Z}$, so $\pi (N)$ must be finite and hence trivial. 
    Thus, the map $\pi _0:\rho (\Gamma )\rightarrow \pi (\Gamma )$ given by $\pi _0 (\rho (\gamma ))=\pi (\gamma )$ is a well-defined group homomorphism.  Since the map $h$ is $N$-invariant, it descends through $\theta_0$ to a map $h_0:Y_0\rightarrow \mathbf{Z}$ satisfying
    \[
h_0(\rho (\gamma )y)w_0(\rho (\gamma ),y)h_0(y)^{-1}=\pi _0(\rho (\gamma ))
    \]
for all $\gamma \in \Gamma$ and a.e.\ $y\in Y_0$. Thus, $w_0$ is cohomologous to the homomorphism $\pi _0$, a contradiction.
\end{proof}

\subsection{Cocycle superrigidity and ends of groups}

\begin{lemma}\label{lm:actTreeCocycle}
    Let $\Gamma$ be a countable group acting on a tree $T$ without inversion and with no global fixed points, and let $E$ denote the set of all directed edges of $T$. 
    Let $L$ be a nontrivial Polish group admitting a compatible bi-invariant metric $d$, and let $(X,\mu )$ be a nontrivial standard probability space. Then the generalized Bernoulli action $\Gamma\curvearrowright (X^{E},\mu ^{E})$ admits a measurable $L$-valued cocycle that is not cohomologous to a homomorphism.
\end{lemma}

\begin{proof}
Let $e\mapsto e^{-1}$ denote the fixed point free involution on $E$ sending an edge to its inverse. We may write $E$ as the disjoint union $E=E^+\sqcup E^-$ where both $E^+$ and $E^-$ are $\Gamma$-invariant, and the inversion map gives an equivariant bijection between $E^+$ and $E^-$. 

By \cite[Lemma 3.2]{Furman2007}(see also \cite[Lemma 2.11]{Popa2007}), we may assume without loss of generality that $X$ has cardinality 2, and since $L$ is nontrivial we may additionally assume that $X=\{ l_0, l_1 \}$ is a two element subset of $L$. Similarly, it is also enough to prove the conclusion of the Lemma for the generalized Bernoulli action on $(X^{E^+},\mu ^{E^+})$ in place of the action on $(X^{E},\mu ^{E})$.

We may assume that the bi-invariant metric $d$ on $L$ is bounded by $1$. Given two measurable  maps $\varphi _0 , \varphi _1: X^{E^+}\rightarrow L$, define  
    \[
    \tilde{d}(\varphi _0, \varphi _1) \coloneqq \int _X d(\varphi _0(x),\varphi _1(x)) \, d\mu (x) ,
    \] 
so that $\tilde{d}$ is a pseudometric on the set of all measurable maps from $X^{E^+}$ to $L$.
    
Fix a vertex $v$ of the tree $T$. 
For each $\gamma\in \Gamma$, let $p_{\gamma} = e^{\gamma}_1\cdots e^{\gamma}_{n_{\gamma}}$ denote the geodesic path from $v$ to $\gamma .v$ in $T$ (so each $e^{\gamma}_i$ belongs to $E$, the origin of $e^{\gamma}_1$ is $v$, the terminus of $e^{\gamma}_{n_{\gamma}}$ is $\gamma .v$, and the terminus of $e^{\gamma }_i$ equals the origin of $e^{\gamma}_{i+1}$ for all $1\leq i<n_{\gamma}$).
Define the map $c: \Gamma \times X^{E^+ }\to L$ by $c(\gamma, x) \coloneqq x(e^\gamma_{n_\gamma})\cdots x(e^\gamma_{1})$ where $x(e)\coloneqq x(e^{-1})^{-1}$ for $e\in E^{-}$, and the product refers to the group multiplication in $L$. 
Then it is straightforward to check that $c$ is a measurable $L$-valued cocycle. 
    
Assume toward a contradiction that $c$ is cohomologous to a homomorphism. Then, after discarding a null set, there exists a measurable map $f: X^{E^+}\to L$ and a homomorphism $\rho: \Gamma\to L$ such that $c(\gamma,x)=f(\gamma. x)\rho(\gamma)f(x)^{-1}$ for all $\gamma \in \Gamma$ and $x\in X$.

Let $r \geq 1$ be the maximum of the two ratios $\frac{\mu(l_0)}{\mu(l_1)}$ and $\frac{\mu(l_1)}{\mu(l_0)}$, and fix $\epsilon$ with $0<\epsilon<\frac{d(l_0,l_1)}{4r}$.
    We may then find a finite set $F\subseteq E^+$ of edges, together with a measurable map $\varphi: X^{E^+}\to L$ that depends only on the coordinates from $F$ and satisfies $\tilde d(f,\varphi)<\epsilon$.
    Since the action of $\Gamma$ on $T$ has no fixed points, the orbit of $v$ is unbounded, and hence there exists some $\gamma \in \Gamma$ such that the path $p_\gamma$ has length at least $3|F|$. 
    There is therefore an edge $e\in E^+$ not belonging to the set $F\cup \gamma^{-1} .F$, such that either $e$ or $e^{-1}$ belongs to $p_\gamma$.

    Let $\sigma: X^{E^+}\to X^{E^+}$ be the involution that changes the value at the $e$-coordinate and leaves all other coordinates unchanged. 
    Then $\varphi\gamma \sigma=\varphi\gamma$ and $\varphi\sigma=\varphi$ where we write $\gamma$ for the map $x\mapsto \gamma. x$.
    The Radon-Nikodym derivative $\frac{d\sigma _*\mu }{d\mu}$ is bounded above by $r$, hence $\tilde d (\varphi \sigma, f \sigma)<r\epsilon$. 
    Therefore, by the triangle inequality, $\tilde d(f\gamma ,f\gamma \sigma)<2r\epsilon$ and $\tilde d(f,f\sigma)<2r\epsilon$. Thus
    \[
    \int _X d(f (\gamma .x) ,f (\gamma .\sigma x)) + d(f (x) ,f (\sigma x)) \, d\mu (x) < d(l_0,l_1 ),
    \]
    so there exists some $x_0\in X$ for which $d(f(\gamma .x_0),f (\gamma .\sigma x_0)) + d(f (x_0) ,f (\sigma x_0))<d(l_0,l_1 )$. By the triangle inequality and bi-invariance of $d$ we have
    \begin{align*}
    d(c(\gamma , \sigma x_0),c(\gamma , x_0))&= d(f(\gamma .\sigma x_0)\rho(\gamma)f(\sigma x_0)^{-1},f(\gamma . x_0)\rho(\gamma)f(x_0)^{-1}) \\
    &\leq  d(f(\gamma .x_0 ),f(\gamma .\sigma x_0)) + d(f (x_0) ,f( \sigma x_0))\\
    &<d(l_0,l_1 ).
    \end{align*}
    Since $\sigma x_0$ differs from $x_0$ on exactly one edge along the path $p_\gamma$, the definition of $c$ together with bi-invariance of $d$ implies that $d(c(\gamma, \sigma x_0),c(\gamma, x_0))=d(l_0,l_1)$, a contradiction. 
\end{proof}

\begin{cor}\label{cor:LcocSuperrigidImpliesOneEnd}
    Let $\Gamma$ be a countable group and let $(X,\mu)$ be a nontrivial standard probability space. Assume the Bernoulli shift $\Gamma\acts (X^\Gamma, \mu^\Gamma)$ is $L$-cocycle superrigid for some nontrivial Polish group $L$ that admits a compatible bi-invariant metric. Then $\Gamma$ has at most one end, i.e., $H^1(\Gamma, \mathbf{Z}_2\Gamma)$ is trivial. 
\end{cor}

\begin{proof}
    By \cite[Lemma 3.2]{Furman2007}, it is enough to consider the case where $(X,\mu )$ is atomless. 
    We prove the contrapositive, so assume that $\Gamma$ has more than one-end. 
    Then by \cite[Theorem IV.6.10]{Dicks1989} $\Gamma$ acts without inversion on a tree $T$ with finite edge stabilizers and no fixed points. Let $E$ denote the set of directed edges of $T$.
    The associated orthogonal representation of $\Gamma$ on $\ell ^2E$ is then isomorphic to a subrepresentation of a multiple of the left-regular representation of $\Gamma$. Thus, applying the Gaussian action functor to this subrepresentation and its orthogonal complement, we see that the Bernoulli shift of $\Gamma$ over an atomless base space $(X,\mu )$ is isomorphic to the direct product of the generalized Bernoulli action $\Gamma\curvearrowright (X^{E},\mu ^{E})$ with a mixing action of $\Gamma$.
    By Lemma \ref{lm:actTreeCocycle}, the action of $\Gamma$ on $\Gamma\curvearrowright (X^{E},\mu ^{E})$ is not $L$-cocycle superrigid, so another application of \cite[Lemma 3.2]{Furman2007} shows that the Bernoulli shift $\Gamma\curvearrowright (X^\Gamma , \mu ^\Gamma )$ is also not $L$-cocycle supperrigid.     
\end{proof}

\begin{remark}
    In the case when $L$ is the circle, Corollary \ref{cor:LcocSuperrigidImpliesOneEnd} follows from Peterson-Sinclair \cite[Theorem 1.1]{PetersonSinclair12} where they obtain the stronger conclusion (see \cite{BekkaValette97} and \cite[Corollary 2.4]{PetersonThom2011}) that the first $\ell^2$-Betti number of $\Gamma$ vanishes. 
\end{remark}

\subsection{Rigidity for wreath product actions}
We now consider a stable orbit embedding of a $\mathcal{G}_{\mathrm{ctble}}$-cocycle superrigid action of a group $\Gamma$ into a wreath product action of a wreath product group $B\wr\Lambda$. 

In this context, assuming $\Gamma$ has no nontrivial finite normal subgroups, Lemma \ref{lemma:rigidity_prep} shows that $\Gamma$ must be isomorphic to a subgroup of $B\wr\Lambda$, and the stable orbit embedding can be rearranged into an isomorphism of the action of $\Gamma$ with an ergodic component of the action of this subgroup of $B\wr\Lambda$. 
In the following theorem we show that if the action of $\Gamma$ is moreover {\bf totally ergodic} -- meaning that every infinite subgroup of $\Gamma$ acts ergodically -- then even more is true: 
$\Gamma$ is in fact isomorphic to a subgroup of $\Lambda$, and the stable orbit embedding can be rearranged into an isomorphism of the action of $\Gamma$ with the action of this subgroup of $\Lambda$. 

\begin{thm}\label{Rigid}
Let $\Gamma$ be a countably infinite group with no nontrivial finite normal subgroups, and let $B$ and $\Lambda$ be countable groups, with $B$ nontrivial and $\Lambda$ infinite.
Let $\Gamma\acts (X,\mu)$ be a totally ergodic p.m.p.\ action of $\Gamma$ that is $\mathcal{G}_\mathrm{ctble}$-cocycle superrigid. 
Let $B\acts (Y,\nu)$ be a free p.m.p.\ action of $B$, and let $B\wr \Lambda \acts (Y^\Lambda,\nu^\Lambda )$ be the associated wreath product action. 

Suppose that there exist positive measure sets $X_0\subseteq X$ and $Z_0\subseteq Y^\Lambda$ and a measure space isomorphism $\phi _0: (X_0,\mu _0)\rightarrow (Z_0,\eta _0)$ such that 
\[
\phi _0(\Gamma .x)\cap X_0\subseteq (B\wr \Lambda).\phi _0(x)\cap Z_0
\]
for a.e.\ $x\in X_0$, where $\mu _0$ and $\eta _0$ denote the normalized restrictions of $\mu$ and $\nu ^{\Lambda}$ respectively. 

Then, after discarding a null set, there exist measurable functions $g: X\rightarrow B\wr \Lambda$ and $\phi: X\rightarrow Z_0$ with $\phi|_{X_0}=\phi _0$, and an injective group homomorphism $\sigma : \Gamma\rightarrow\Lambda$ such that the map $\theta :(X,\mu)\rightarrow(Y^\Lambda,\nu^\Lambda)$, given by $\theta (x)\coloneqq g(x).\phi (x)$, is a measure space isomorphism satisfying $\theta (\gamma. x)= \sigma (\gamma).\theta (x)$ for all $\gamma \in \Gamma$ and $x\in X$.

In particular, the action of $\Gamma$ on $(X,\mu )$ is isomorphic to the Bernoulli shift of $\Gamma$ with base space $(Y^{\Lambda /\sigma (\Gamma )}, \nu ^{\Lambda /\sigma  (\Gamma )})$.
\end{thm}

\begin{proof}
By \cite{TD15}, total ergodicity together with $\Gamma$ having no nontrivial finite normal subgroups implies that the action $\Gamma\acts(X,\mu)$ is essentially free. 

After discarding a null set, we may find $f$, $\phi$, $\rho$, $(Z_1, \eta _1)$, and $\psi :(X,\mu )\rightarrow (Z_1,\eta _1 )$ as in the conclusion of Lemma \ref{lemma:rigidity_prep} applied to $L=B\wr\Lambda$ and $(Z,\eta )=(Y^\Lambda , \nu ^{\Lambda})$. 
In particular, $\eta _1$ is the normalized restriction of $\nu ^\Lambda$ to the positive measure subset $Z_1$ of $Y^{\Lambda}$, and $\psi$ is a measure preserving injection satisfying
\begin{equation}\label{eqn:psi_equivariance}
    \psi(\gamma .x) =\rho(\gamma).\psi(x)
\end{equation}
for all $\gamma \in \Gamma$ and $x\in X$.

The homomorphism $\rho$ has the form $\rho (\gamma ) = \rho_1 (\gamma ) \rho_2 (\gamma )$, where $\rho_2$ is a homomorphism from $\Gamma$ into $\Lambda$ and $\rho_1$ is a map from $\Gamma$ into $\bigoplus_{\Lambda} B$. 
We will show that $\rho _2$ is injective, but first we need to show that $\rho _2 (\Gamma )$ is infinite.

By Lemma \ref{lem:cocSuperrigidImpliesStronglyErgodic}, the action of $\Gamma$ on $(X,\mu )$ is strongly ergodic, hence by total ergodicity and \cite[Lemma 3.1]{abertelek2012dynamical}, the action of each finite index subgroup of $\Gamma$ is also strongly ergodic on $(X,\mu )$. 
The group $\Lambda$ being infinite implies that the action of $\bigoplus _{\Lambda}B$ on $(Y^{\Lambda},\nu ^{\Lambda})$ is not strongly ergodic, and moreover there does not exist a subgroup $H$ of $\bigoplus _{\Lambda}B$ and a positive measure $H$-invariant subset of $Y^{\Lambda}$ on which the action of $H$ is strongly ergodic \cite{JS}.
In particular, the action of $\rho (\ker (\rho _2))$ on $(Z_1, \eta _1 )$ is not strongly ergodic, and hence by \eqref{eqn:psi_equivariance} the action of $\ker (\rho _2)$ on $(X,\mu )$ is also not strongly ergodic. 
Thus, $\ker (\rho _2)$ must have infinite index in $\Gamma$, and we see that $\rho _2(\Gamma )$ is infinite. 

Assume now toward a contradiction that $\rho_2$ is not injective. 
Fix a nonidentity element $\gamma \in \ker \rho _2$ and let $b=\rho (\gamma )$. Since $\rho$ is injective, $b$ is a nonidentity element of $\bigoplus _{\Lambda}B$, hence its support $\mathrm{supp}(b) \coloneqq \{ \lambda \in \Lambda : b_{\lambda} \neq e \}$ is a nonempty finite subset of $\Lambda$. 
Since $\rho_2(\Gamma)$ is infinite, there exists a sequence $(\lambda _n)_{n\geq 0}$ of elements of $\rho _2 (\Gamma )$ such that the sets $\lambda _n \mathrm{supp}(b)$, $n\geq 0$, are pairwise disjoint. 
For each $n\geq 0$ let $\delta _n \in \Gamma$ be such that $\rho _2(\delta _n)=\lambda _n$. 
Observe that $\delta _n\gamma \delta _n ^{-1}\in \ker \rho_2$, and $\mathrm{supp}(\rho (\delta _n\gamma \delta _n^{-1})) = \lambda _n\mathrm{supp}(b)$.
For each $k\geq 0$ let $C_k$ be the subgroup of $\ker \rho_2$ generated by $\{ \delta _n\gamma\delta _n ^{-1} \} _{n\geq k}$, let $S_k\coloneqq\bigcup _{0\leq i< k}\lambda _i\mathrm{supp}(b)$, and let $k_0$ be so large that each atom (if any exist) of the measure $\nu ^{S_{k_0}}$ has measure strictly less than $\nu ^{\Lambda}(Z_1)$. 
Then $\rho (C_{k_0})$ is an infinite subgroup of $\bigoplus _{\Lambda}B$ whose elements all have support disjoint from $S_{k_0}$.  
The restriction map from $(Y^\Lambda ,\nu ^{\Lambda})$ to $(Y^{S_{k_0}},\nu ^{S_{k_0}})$ is a measure-preserving $\rho (C_{k_0})$-invariant map, which witnesses that each ergodic component of the action of $\rho (C_{k_0})$ on $(Y^{\Lambda}, \nu ^{\Lambda})$ has $\nu ^{\Lambda}$-measure strictly less than $\nu ^{\Lambda}(Z_1)$; in particular $\rho (C_{k_0})$ does not act ergodically on $(Z_1, \eta _1 )$.
However, since the action $\Gamma \acts (X,\mu)$ is totally ergodic, the infinite subgroup $C_{k_0}$ acts ergodically on $(X,\mu)$ and hence, by \eqref{eqn:psi_equivariance}, $\rho(C_{k_0})$ acts ergodically on $(Z_1, \eta _1 )$, a contradiction.

Since $\rho_2$ is injective, we can apply Lemma \ref{prop:One_ended_Implies_Conj} to obtain some $\xi\in B^\Lambda$ satisfying $\xi\rho(\Gamma)\xi^{-1}\leq\Lambda$, so that in fact $\xi \rho (\gamma  )\xi ^{-1}=\rho _2(\gamma )$ for all $\gamma \in \Gamma$. 
Note that applying $\xi$ coordinate-wise gives a measure space automorphism of $Y^\Lambda$ via $(\xi y)_{\lambda}\coloneqq \xi _{\lambda}.y_{\lambda}$ for $y\in Y^{\Lambda}$, $\lambda \in \Lambda$. 
This is furthermore an isomorphism from the action of $\Gamma$ on $(Y^{\Lambda}, \nu ^{\Lambda})$ implemented through $\rho$, to the action of $\Gamma$ on $(Y^{\Lambda}, \nu ^{\Lambda})$ implemented through $\rho _2$. 
Since $\rho _2$ is injective, the latter action is a Bernoulli action of $\Gamma$, hence the former is Bernoulli as well. In particular, these actions are ergodic, so $Z_1$ is $\nu ^{\Lambda}$-conull in $Y^{\Lambda}$, and $\eta_1 = \nu ^{\Lambda}$.      

The map $\psi$ therefore gives an isomorphism between the action of $\Gamma$ on $(X,\mu)$ and a Bernoulli action of $\Gamma$. 
Corollary \ref{cor:LcocSuperrigidImpliesOneEnd} now implies that the group $\Gamma$ is at most one ended, hence the same is true of $\rho (\Gamma )$, since $\rho$ is injective. 
We apply Lemma \ref{prop:One_ended_Implies_Conj} to get that $\rho(\Gamma)$ is conjugate via some element $d\in \bigoplus _{\Lambda}B$ to $\rho_2(\Gamma)\leq\Lambda$ (in fact, one may show that $d$ differs from $\xi$ by a constant map).

The maps $g (x)\coloneqq d f(x)$,  $\sigma (\gamma )\coloneqq d^{-1}\rho(\gamma) d$, and $\theta (x)\coloneqq d.\psi (x)$ then satisfy the conclusion of the theorem. 
\end{proof}

Let $B$ and $\Gamma$ be countable groups along with an action of $\Gamma$ on a set $V$.
Define the \textbf{(restricted) permutational wreath product group} $B\wr_V \Gamma\coloneqq \Gamma\ltimes\bigoplus_V B$ where $\Gamma$ acts on $\bigoplus _V B$ via the shift action. 
Let $B\acts (X,\mu)$ be a p.m.p.\ action of $B$. 
We define the associated wreath product action $B\wr_V\Gamma\acts (X^V,\mu^V)$ by $(b.x)_v = b_v. x_{v}$ and $(\gamma. x)_v=x_{\gamma ^{-1}.v}$ for $b\in \bigoplus _V B$, $\gamma \in \Gamma$, $x\in X^V$, $v\in V$.
We write $\mathrm{supp}(b)$ for the support of $b\in \bigoplus _V  B$, i.e., $\mathrm{supp}(b)\coloneqq \{ v\in V : b_v\neq e \}$.

\begin{lemma}\label{lem:EquivOEWreathProdActions}
Let $\Gamma$ be a countable group acting on countable sets $V$ and $W$ without finite orbits, and assume that each $\gamma \in \Gamma - \{ e \}$ fixes at most finitely many points in $W$.  
Let $B\acts (X,\mu )$ and $C\acts (Y,\nu )$ be p.m.p.\ actions of countable groups $B$ and $C$, and let
$B\wr_V \Gamma\acts (X^V, \mu^V)$ and $C\wr_W \Gamma\acts (Y^W, \nu^W)$ be the associated wreath product actions. 
Suppose that $\theta: (X^V,\mu^V)\to (Y^W,\nu^W)$ is a measure preserving $\Gamma$-equivariant map satisfying
\begin{equation}\label{eqn:orbits_into_orbits}
\theta ((B\wr_V \Gamma ).x)\subseteq (C\wr_W\Gamma ).\theta (x)
\end{equation}
for a.e.\ $x\in X^V$. Then 
\[
\textstyle{\theta ((\bigoplus _V B).x)\subseteq (\bigoplus _W C).\theta (x)}
\]
for a.e.\ $x\in X^V$.

\end{lemma}

\begin{proof} 
We may of course assume that $\nu$ is not a point mass. 

By \eqref{eqn:orbits_into_orbits}, for each $b\in \bigoplus _{V}B$ there are measurable functions  $f_b: X^V\to \bigoplus_W C$ and $g_b: X^V\to \Gamma$ such that $\theta(b.x)=f_b(x)g_b(x).\theta(x)$ for a.e.\ $x\in X^V$. 

Suppose toward a contradiction that there exists some $b\in \bigoplus _{V}B$ and a positive measure subset $E$ of $X^{V}$ on which $g_b$ takes a constant non-trivial value $\delta\in \Gamma -\{ e \}$.
Fix a measurable subset $Y_0$ of $Y$ with $0<\nu(Y_0)<1$ and let $r\coloneqq \nu (Y_0)^2 + (1-\nu (Y_0))^2$. 
Fix $\epsilon >0$ so small that
\[
r\mu ^{V}(E) + \epsilon < \mu ^{V}  (E) -\epsilon .
\]
Fix $w_0\in W$.
For a subset $S$ of $V$, let $\pi _S : X^{V}\to X^S$ denote the restriction map $\pi _S(x)=x|_S$. 
We may find a finite subset $P$ of $V$ and a measurable function $\varphi :X^{V}\to \{ 0 , 1\}$ that factors through $\pi _{P}$ such that the set 
\[
E_P \coloneqq \{ x\in X^{V} : 1_{Y_0}(\theta (x)_{w_0})=\varphi (x)\}
\]
has measure at least  $1-\epsilon /3$. 
For each $\gamma \in \Gamma$ the map $x\mapsto \varphi (\gamma ^{-1}.x)$ then factors through $\pi _{\gamma .P}$, and by $\Gamma$-equivariance of $\theta$, for each $x\in \gamma .E_P$ we have $1_{Y_0}(\theta(x)_{\gamma .w_0})=\varphi (\gamma ^{-1}. x)$. 

We may also find a large enough finite subset $Q$ of $W$ such that the set 
\[
E_Q \coloneqq \{ x\in X^{V} : \mathrm{supp}(f_b(x))\subseteq Q\}
\]
has measure at least $1-\epsilon /3$.

For each $\gamma \in \Gamma$ with $\gamma .P$ disjoint from $\mathrm{supp}(b)$ we have $\varphi (\gamma ^{-1}b.x) =\varphi (\gamma ^{-1}.x)$ for all $x\in X^{V}$.
In addition, if $\gamma. w_0\not\in Q$ then each $x\in E_Q\cap E$ satisfies
\[
\theta (b.x)_{\gamma .w_0}=(f_b(x)\delta.\theta (x))_{\gamma. w_0} = (\delta.\theta (x))_{\gamma. w_0} = \theta (x)_{\delta^{-1}\gamma .w_0}.
\]
Therefore, if $\gamma .w_0\not\in Q$ and if $\gamma. P$ is disjoint from $\mathrm{supp}(b)$, then each $x$ belonging to the intersection $\gamma. E_P \cap b^{-1}\gamma. E_P\cap E_Q\cap E$ satisfies
\[
1_{Y_0}(\theta (x)_{\delta^{-1}\gamma. w_0}) = 1_{Y_0}(\theta (b.x)_{\gamma. w_0})=\varphi (\gamma ^{-1}b.x)=\varphi (\gamma ^{-1}.x)=1_{Y_0}(\theta (x)_{\gamma .w_0}).
\]
Since the set $F$ of fixed points of $\delta$ in $W$ is finite, and the sets $Q$, $P$, and $\mathrm{supp}(b)$ are all finite, repeated applications of Neumann's Lemma allow us to obtain an infinite sequence $(\gamma_n)_{n\in \mathbf{N}}$ in $\Gamma$ with $\gamma _n. P$ disjoint from $\mathrm{supp}(b)$ for all $n$, and such that $\gamma _n. w_0$, $n\in \mathbf{N}$, are distinct elements of $W\setminus (F\cup Q)$. 

Thus, for every $n\in \mathbf{N}$ the set $\gamma_n. E_P \cap b^{-1}\gamma_n .E_P\cap E_Q\cap E$ is contained in $\theta ^{-1}(D_{\gamma_n})\cap E$, where
\[
D_{\gamma_n} \coloneqq \{ y\in Y^{W} : 1_{Y_0}(y_{\delta^{-1}\gamma_n .w_0})=1_{Y_0}(y_{\gamma_n .w_0})\},
\]
and hence  
\begin{equation}\label{eqn:DgammaE}
\mu ^{V}(\theta^{-1}(D_{\gamma_n}) \cap E)\geq\mu ^{V}(E)-\epsilon.
\end{equation}
On the other hand, the sequence $(1_{D_{\gamma _n}})_{n\in\mathbf{N}}$ converges weakly in $L^1(Y^W ,\nu ^W)$ to the constant function $r=\nu(Y_0)^2+(1-\nu(Y_0))^2$, hence the sequence $(1_{\theta ^{-1}(D_{\gamma _n})})_{n\in\mathbf{N}}$ also converges weakly in $L^1(X^V,\mu ^V)$ to $r$ (since the map $h\mapsto h\circ \theta$ isometrically embeds $L^1(Y^W ,\nu ^W)$ into $L^1(X^V,\mu ^V)$ as a norm-closed subspace).
Therefore 
\[
\lim _{n\to\infty}\mu ^{V} (\theta^{-1}(D_{\gamma_n}) \cap  E) = r\mu ^{V}(E),
\]
so for all but finitely many $n\in \mathbf{N}$ we have
\[
\mu ^{V}(\theta^{-1}(D_{\gamma_n})\cap E)< r\mu ^{V}(E) + \epsilon ,
\]
contradicting \eqref{eqn:DgammaE} and our choice of $\epsilon$.
\end{proof}

Let $(X,\mu)$ be a standard probability space. 
Recall that the \textbf{Shannon entropy} of $(X,\mu )$ is defined by $h(X,\mu)\coloneqq-\sum_X \mu(x)\ln \mu(x)$ when $\mu$ is purely atomic, and $h(X,\mu)\coloneqq+\infty$ otherwise.

\begin{cor}\label{cor:Rigid}
Let $\Gamma$ be a countably infinite group having no nontrivial finite normal subgroups, and whose Bernoulli shift action is $\mathcal{G}_{\mathrm{ctble}}$-cocycle superrigid. 
Let $\Lambda$ be an arbitrary infinite countable group.

Let $B$ and $C$ be two nontrivial countable groups acting in a free p.m.p.\ manner on $(X,\mu)$ and $(Y,\nu)$, respectively. 
Assume that the associated wreath product actions 
$B\wr \Gamma\acts (X^\Gamma, \mu^\Gamma)$ and $C\wr\Lambda \acts (Y^\Lambda, \nu^\Lambda )$ are stably orbit equivalent. 

Then $\Gamma$ and $\Lambda$ are isomorphic, and $\bigoplus_\Gamma B$ and $\bigoplus_\Lambda C$ are orbit equivalent. 
Moreover, if $\Gamma$ is sofic and the actions of $B$ on $(X,\mu )$ and of $C$ on $(Y,\nu )$ are both ergodic, then $B$ and $C$ have the same cardinality.
\end{cor}

\begin{proof} By assumption, there exist positive measure sets $X_0\subseteq X^\Gamma$ and $Y_0\subseteq Y^\Gamma$ and a stable orbit equivalence $\phi_0: X_0\rightarrow Y_0$. 
By Theorem \ref{Rigid}, there exist measurable functions $g: X^\Gamma\rightarrow C\wr \Lambda$ and $\phi: X^\Gamma\rightarrow Y_0$ with $\phi|_{X_0}=\phi _0$, and an injective group homomorphism $\sigma : \Gamma\rightarrow\Lambda$ such that the map $\theta :(X^\Gamma,\mu^\Gamma)\rightarrow(Y^\Lambda,\nu^\Lambda)$, given by $\theta (x)\coloneqq g(x).\phi (x)$, is a measure space isomorphism conjugating the action of $\Gamma$ on $X^\Gamma$ to the action of $\Gamma$ on $Y^\Lambda$ implemented through $\sigma$.

We first show that $\theta$ is an orbit equivalence between the $B\wr\Gamma$ and $C\wr\Lambda$ actions. 
First note that by ergodicity of the $B\wr\Gamma$ action, after discarding a null set we have that every $(B\wr\Gamma )$-orbit is of the form $(B\wr\Gamma). x$ for some $x\in X_0$. 
In addition, for $x\in X_0$, we have $\theta((B\wr\Gamma ).x)\subseteq (C\wr\Lambda). \phi_0(x)$ and, since $\phi_0$ is a stable orbit equivalence, it follows that $\theta$ takes distinct $B\wr\Gamma$ orbits to distinct $C\wr\Lambda$ orbits. 
After discarding further null sets we may assume that $\theta$ is bijective, which implies $\theta((B\wr\Gamma ).x)= (C\wr\Lambda) .\phi_0(x)=(C\wr\Lambda).\theta(x)$ for almost every $x\in X_0$. 
This shows that $\theta$ is an orbit equivalence between the two wreath product actions.
It now follows from Lemma \ref{lem:EquivOEWreathProdActions} that $\theta ((\bigoplus _{\Gamma}B).x)\subseteq (\bigoplus _{\Lambda}C).\theta (x)$ for a.e.\ $x\in X^\Gamma$.

We now show that the homomorphism $\sigma$ is surjective. 
Given $\lambda\in \Lambda$, if $b\in \bigoplus_\Gamma B$ and $\gamma\in \Gamma$ are such that $\lambda .\theta(x)=\theta(b\gamma .x)$ for a positive measure set of $x\in X$, then there is some $c\in \bigoplus _\Lambda C$ such that 
\[
\lambda.\theta(x)=c.\theta(\gamma .x)=c\sigma(\gamma).\theta(x)
\]
for a possibly smaller positive measure set of $x\in X$.
Freeness of the $C\wr\Lambda$ action therefore implies that $\lambda=c\sigma(\gamma)$, hence $\lambda=\sigma(\gamma)$ and $c$ and $b$ are both trivial. 

Applying Lemma \ref{lem:EquivOEWreathProdActions} to both $\theta$ and $\theta^{-1}$ shows that $\bigoplus_\Gamma B$ and $\bigoplus_\Gamma C$ are orbit equivalent via $\theta$. 

Consequently the Bernoulli shifts $\Gamma\acts(X^\Gamma, \mu^\Gamma)$ and $\Gamma\acts (Y^\Gamma,\nu^\Gamma)$ are isomorphic. 
Utilizing soficity of $\Gamma$ and ergodicity of the $B-$ and $C-$actions, by Bowen \cite{Bowen2010} and Kerr-Li \cite{KerrLi2011}, the Shannon entropies $\log|B|=h(X,\mu)=h(Y,\nu)=\log |C|$ must be equal, concluding the proof.  
\end{proof}

\subsection{Automorphisms of wreath product equivalence relations}

Let $\Gamma$ be a countably infinite group.
Let $B\curvearrowright (X,\mu )$ be a free action of a nontrivial countable group $B$, and let $\mathcal{B}$ be the orbit equivalence relation of this action.

We naturally identify the wreath product groupoid $\mathcal{B}\wr\Gamma$ with the orbit equivalence relation of the associated wreath product action $B\wr\Gamma \curvearrowright (X^\Gamma ,\mu ^\Gamma )$.
Under this identification, the direct sum subgroupoid $\bigoplus _\Gamma \mathcal{B}$ is identified with the equivalence relation on $X^\Gamma$ described in the last paragraph of $\S$\ref{subsection:direct_sum_groupoid}.
Thus, automorphisms of $\mathcal{B}\wr\Gamma$ and $\bigoplus _{\Gamma}\mathcal{B}$ are determined by their restriction to the set of objects, which is identified with $X^\Gamma$.
In this way, we view both $\mathrm{Aut}(\mathcal{B}\wr\Gamma )$ and $\mathrm{Aut}(\bigoplus_\Gamma\mathcal{B})$ as subgroups of $\mathrm{Aut}(X^\Gamma, \mu ^\Gamma )$, and the full group $[\mathcal{B}\wr\Gamma ]$ as a subgroup of $\mathrm{Aut}(\mathcal{B}\wr\Gamma )$.

We describe an injective group homomorphism $\Aut(\Gamma)\hookrightarrow \Aut(\mathcal{B}\wr\Gamma)$. 
Given an automorphism $\varphi\in \Aut(\Gamma)$, define $\varphi ^*: X^\Gamma\to X^\Gamma$ by $\varphi ^*(x)_\gamma = x_{\varphi^{-1}(\gamma)}$ for all $x\in X^{\Gamma}, \gamma \in \Gamma$.

Let $\mathrm{Aut}_\Gamma (\bigoplus_\Gamma \mathcal{B})$ denote the subgroup of $\mathrm{Aut}(\bigoplus_\Gamma \mathcal{B})$ consisting of all automorphism of $\bigoplus_\Gamma\mathcal{B}$ that are $\Gamma$-equivariant. 
Every such automorphism is also an automorphism of $\mathcal{B}\wr\Gamma$, and thus $\mathrm{Aut}_\Gamma(\bigoplus_\Gamma\mathcal{B})$ is a subgroup of $\mathrm{Aut}(B\wr\Gamma )$. 
By Lemma \ref{lem:EquivOEWreathProdActions}, every automorphism of $\mathcal{B}\wr\Gamma$ that is $\Gamma$-equivariant is itself an automorphism of $\bigoplus _\Gamma \mathcal{B}$, and hence belongs to $\Aut_\Gamma(\bigoplus _\Gamma\mathcal{B})$.

\begin{cor}\label{cor:autWreathEqRels}
Assume that $\Gamma$ is a countably infinite group with no nontrivial finite normal subgroups, and that the Bernoulli shift action of $\Gamma$ on $(X^\Gamma ,\mu^\Gamma )$ is $\mathcal{G}_{\mathrm{ctble}}$-cocycle superrigid.
Then 
\[
\mathrm{Aut}(\mathcal{B}\wr\Gamma) = \mathrm{Aut}(\Gamma )^*\cdot \big(\mathrm{Aut}_\Gamma (\textstyle{\bigoplus_\Gamma \mathcal{B})}\cdot [\mathcal{B}\wr\Gamma]\big).
\]
and $\mathrm{Aut}_\Gamma (\textstyle{\bigoplus_\Gamma \mathcal{B})}\cdot [\mathcal{B}\wr\Gamma]$ is a normal subgroup of $\Aut(\mathcal{B}\wr\Gamma)^*$.

In addition, $\mathrm{Aut}_\Gamma (\bigoplus_\Gamma \mathcal{B})\cap [\mathcal{B}\wr\Gamma]=Z(\Gamma)$ where $Z(\Gamma)$ is the center of $\Gamma$, naturally viewed as a subgroup of $[\mathcal{B}\wr\Gamma]$. 
\end{cor}
\begin{proof}
    Let $\phi\in \mathrm{Aut}(\mathcal{B}\wr\Gamma)$ be an arbitrary automorphism. Apply Theorem \ref{Rigid} to $\phi$ to get a measurable map $g:X^\Gamma\to B\wr\Gamma$ and an injective homorphism $\sigma:\Gamma\to \Gamma$ such that the measure space isomorphism $\theta\in \Aut(X^\Gamma,\mu^\Gamma)$ given by $x\mapsto g(x).\phi(x)$ satisfies $\theta(\gamma.x)=\sigma(\gamma).\theta(x)$ for every $\gamma\in \Gamma$. As in the proof of Corollary \ref{cor:Rigid}, the map $\sigma$ is surjective, and hence $\sigma\in \Aut(\Gamma)$. Define $\tilde g: X^\Gamma\to X^\Gamma$ by $x\mapsto g(\phi^{-1}x).x$ and note that $\tilde g$ is a measure space isomorphism since $\tilde g\circ\phi=\theta$. For almost every $x\in X^\Gamma$, and every $\gamma,\delta\in \Gamma$ we have
    \begin{align*}
        ((\sigma^{-1})^*(\theta(\gamma.x)))_\delta=(\sigma(\gamma).\theta(x))_{\sigma(\delta)}=\theta(x)_{\sigma(\gamma^{-1}\delta)}=(\gamma.(\sigma^{-1})^*(\theta(x)))_\delta,
    \end{align*}
    which shows that $(\sigma^{-1})^*\circ\theta$ is $\Gamma$-equivariant. 
    We can now rewrite $\phi=\tilde g^{-1}\circ\sigma^*\circ((\sigma^{-1})^*\circ \theta)$ to see that $\phi\in[\mathcal{B}\wr\Gamma]\cdot \Aut(\Gamma)^*\cdot \Aut_\Gamma(\bigoplus_\Gamma \mathcal{B})$.

    For $\rho\in \mathrm{Aut}(\Gamma)$ and $\gamma,\delta\in \Gamma$, we have
    \[
    \gamma.\rho^*(x)_\delta=\rho^*(x)_{\gamma^{-1}\delta}=x_{\rho^{-1}(\gamma^{-1})\rho^{-1}(\delta)}=\rho^*(\rho^{-1}(\gamma).x)_\delta
    \]
    which implies $\gamma.\rho^*(x)=\rho(\rho^{-1}(\gamma).x)$. Letting $\theta\in \mathrm{Aut}_\Gamma(\bigoplus_\Gamma \mathcal{B})$, we can compute
    \[
    \gamma.((\rho^*)^{-1}\theta\rho^*)(x)= ((\rho^*)^{-1}\theta)(\rho(\gamma).\rho^*(x))=((\rho^*)^{-1}\theta\rho^*)(\gamma.x)
    \]
    showing $(\rho^*)^{-1}\theta\rho^*$ is $\Gamma$-equivariant and confirming that $\mathrm{Aut}_\Gamma(\bigoplus_\Gamma \mathcal{B})$ is normalized by $\mathrm{Aut}(\Gamma)^*$.

    The group $\mathrm{Aut}(\mathcal{B}\wr\Gamma)$ is exactly the normalizer of $[\mathcal{B}\wr\Gamma]$ in $\mathrm{Aut}(X^\Gamma,\mu^\Gamma)$ (see \cite[\S6]{Kec10}). 
    This implies that $\Aut_\Gamma(\bigoplus_\Gamma \mathcal{B})\cdot[\mathcal{B}\wr\Gamma]$ is in fact a normal subgroup of $\Aut(\mathcal{B}\wr\Gamma)$.

    Consider an automorphism $\theta\in [\mathcal{B}\wr\Gamma] \cap \mathrm{Aut}_\Gamma (\bigoplus_\Gamma \mathcal{B})$ and for $x\in X^\Gamma$ let $b_x\in \bigoplus_\Gamma B$ and $\delta_x\in \Gamma$ be such that $\theta(x)=\delta_xb_x.x$.
    We will show $\theta$ is exactly the map given by $x\mapsto \delta.x$ for some $\delta\in Z(\Gamma)$. 
    By $\Gamma$-equivariance and freeness of the action of $B\wr\Gamma$, we have that $\gamma\delta_x b_x=\delta_{\gamma. x} b_{\gamma .x}\gamma$, and hence $\delta_{\gamma .x}=\gamma\delta_x\gamma^{-1}$ and $b_{\gamma .x}=\gamma b_x\gamma ^{-1}$ for every $\gamma\in \Gamma$ and $x\in X^\Gamma$.
    Thus $x\mapsto \mathrm{supp}(b_x)$ is a $\Gamma$-equivariant 
    map to the left translation action of $\Gamma$ on the set of all finite subsets of $\Gamma$; since $\Gamma$ is infinite the pushforward of $\mu$ must be the point mass at the empty set, i.e., $b_x$ is trivial for almost every $x\in X^\Gamma$.
    In addition, the set $\{ (x,y) \in X^\Gamma \times X^\Gamma : \delta _x =\delta _y \}$ is a positive measure $\Gamma$-invariant subset of $X^\Gamma \times X^\Gamma$, hence must be conull since the Bernoulli shift action of $\Gamma$ is weakly mixing.
    Thus, $\delta _x =\delta$ is constant on a conull subset of $X^\Gamma$, and it follows that $\delta\in Z(\Gamma)$ and $\Aut_\Gamma(\bigoplus_\Gamma \mathcal{B})\cap[\mathcal{B}\wr\Gamma]\subseteq Z(\Gamma)$. 
    These two subgroups are equal since for each $\gamma\in Z(\Gamma )$ the map $x\mapsto \gamma. x$ is a $\Gamma$-equivariant element of $[\mathcal{B}\wr\Gamma]$.  
\end{proof}

\printbibliography
\end{document}